\documentclass[11pt]{amsart}
\usepackage{amsmath,amssymb,amsthm,graphicx,bbm}
\usepackage[left=1in,right=1in,top=1in, bottom=1in]{geometry}
\usepackage{cite}
\usepackage[dvipsnames]{xcolor}
\usepackage[british]{babel}
\usepackage{hyperref} 
\usepackage{enumitem}
\usepackage{mathtools}
\usepackage{comment}
\usepackage{lipsum}    
\hypersetup{
    colorlinks=false,
    pdfborder={0 0 0},
}
\usepackage{tikz}
\theoremstyle{plain}
\newtheorem{thm}{Theorem}[section]
\newtheorem{lem}[thm]{Lemma}
\newtheorem{prop}[thm]{Proposition}
\newtheorem{cor}[thm]{Corollary}

\newtheorem*{assumption*}{Assumption}

\theoremstyle{remark}
\newtheorem{rem}[thm]{Remark}
\newtheorem{exam}[thm]{Example}

\newcommand{\ind}{{\mathbbm{1}}}
\newcommand{\1}[1]{{\ind\mkern -1.5mu}{\{#1\}}}

\DeclareMathOperator{\Tr}{Tr}
\DeclareMathOperator{\Hessian}{Hess}

\DeclareMathOperator{\E}{\mathbb E}
\renewcommand{\P}{\mathbb P}
\renewcommand{\tilde}{\widetilde}

\renewcommand{\epsilon}{\varepsilon}

\newcommand{\TV}{\mathrm{TV}}

\newcommand{\eps}{\varepsilon}

\newcommand{\ud}{{\mathrm d}}

\newcommand{\R}{{\mathbb R}}

\newcommand{\N}{{\mathbb N}}

\newcommand{\RP}{{\mathbb R}_+}

\newcommand{\cA}{{\mathcal A}}
\newcommand{\cB}{{\mathcal B}}

\newcommand{\cD}{{\mathcal D}}

\newcommand{\cF}{{\mathcal F}}

\newcommand{\cT}{{\mathcal T}}

\newcommand{\cX}{{\mathcal X}}

\makeatletter
\def\namedlabel#1#2{\begingroup  
    #2%
    \def\@currentlabel{#2}%
    \phantomsection\label{#1}\endgroup
}
\makeatother

\newlist{myenumi}{enumerate}{10}
\setlist[myenumi]{leftmargin=0pt, labelindent=\parindent, listparindent=\parindent, labelwidth=0pt, itemindent=!, itemsep=1pt, parsep=4pt}

\newlist{thmenumi}{enumerate}{10}
\setlist[thmenumi]{leftmargin=0pt, labelindent=\parindent, listparindent=\parindent, labelwidth=0pt, itemindent=!}

\title{Subexponential lower bounds for $f$-ergodic Markov processes}
\author{Miha Bre\v{s}ar} 
\address{Department of Statistics, University of Warwick, UK}
\email{miha.bresar.2@warwick.ac.uk}
\author{Aleksandar Mijatovi\'c} 
\address{Department of Statistics, University of Warwick, UK}
\email{a.mijatovic@warwick.ac.uk}

\begin{document}

\begin{abstract}
We provide a  criterion for establishing lower bounds
on the rate of convergence in $f$-variation
of a continuous-time ergodic Markov process to its invariant measure. The criterion consists of  novel super- and submartingale  conditions for certain functionals of the
Markov process. It provides a general approach for proving lower bounds on the tails of the invariant measure and the rate of convergence in $f$-variation of a Markov process, analogous to the widely used  Lyapunov drift conditions for upper bounds. Our key technical innovation produces lower bounds on the tails of the heights and durations of the excursions from bounded sets of a continuous-time  Markov process using path-wise arguments.   

We apply our
theory to elliptic diffusions and  L\'evy-driven stochastic differential equations  with known polynomial/stretched exponential upper bounds on their rates of convergence. Our lower bounds match asymptotically the known upper bounds for these classes of models, thus establishing their rate of convergence to stationarity. The generality of the approach suggests that, analogous to the Lyapunov drift conditions for upper bounds, our methods can be expected to find  applications in many other settings.
\end{abstract}

\maketitle


\smallskip


\section{Introduction}
 Quantifying the rate of convergence of a Markov process $X$ towards its invariant measure $\pi$ is a fundamental problem in probability and its applications (see~\cite{MeynTweedie,Moulin18} and the references therein).  The  literature typically focuses on  \textit{upper} bounds on the rates of convergence (see e.g.~\cite{meyn1993stability1,douc2009subgeometric,hairer2010convergence}). One of the most effective approaches to address this problem is via Lyapunov functions and the corresponding drift conditions. This approach has been widely used since the seminal work of Meyn and Tweedie~\cite{MeynTweedie,meyn1993stability2} and, due to its broad applicability and robustness,  remains one of the most popular methods for quantifying upper bounds on the convergence rates of Markov processes. 
 However, an important limitation of this method, particularly in the case when the upper bounds are subexponential, is the lack of the corresponding lower bounds, which would allow the user to establish the rates of convergence to stationarity. 
 In this paper we develop a theory based on novel drift conditions that yield \textit{lower} subexponential bounds on the rates of convergence in a general setting. We apply it to models previously studied using Lyapunov functions for upper bounds and establish asymptotically matching lower bounds.

Our results can be summarized as follows: 
take a Lyapunov function $V:\cX\to[1,\infty)$ (where $\cX$ is the state space of $X$), which may (but need not) satisfy the upper bound drift conditions in~\cite{douc2009subgeometric} (also given in~\eqref{eq:douc_drif_condition} below). 
Construct scalar functions $\varphi:(0,1] \to\RP$ and $\Psi:[1,\infty)\to[1,\infty)$ such that, outside of a compact set,  the following processes 
\begin{equation}
\label{eq:bacis_assumption}
1/V(X) - \int_0^\cdot \varphi \circ (1/V)(X_s)\ud s \qquad \text{and}\qquad \Psi\circ V(X)
\end{equation}
are a supermartingale and a submartingale, respectively. 
Then we obtain \textit{lower} bounds on the tails of the return times of $X$ to compact sets, the tails of the invariant measure $\pi$ of $X$  and the rate of convergence in $f$-variation (including total variation) distance of the law of $X_t$
to $\pi$. 

Building on ideas in~\cite{Hairer09,hairer2010convergence}, 
we provide general Foster-Lyapunov-type drift conditions implying \textit{lower} bounds for \textit{all times} on the convergence rates of ergodic Markov processes. 
The super- and submartingale properties of the processes in~\eqref{eq:bacis_assumption} are typically verified using the infinitesimal characteristics of $X$ and It\^o's formula. 
In conjunction with the classical drift conditions for upper bounds~\cite{douc2009subgeometric,Fort2005,meyn1993stability1}, our results offer a comprehensive and robust approach to quantifying the convergence rate of ergodic Markov processes. 

As pointed out in~\cite[Sec.~5]{hairer2010convergence}, deriving a lower bound on the rate of convergence to stationarity typically requires a deep understanding of the tail of the invariant measure, which is available only in very special cases (e.g.~Langevin diffusions). The main contribution of this paper consists of providing a robust framework for  deriving lower bounds on the convergence rates without prior knowledge of the tail behavior of the invariant measure. For example, in the context of (possibly L\'evy-driven) stochastic differential equations with multiplicative noise, our theory  makes it possible to establish lower bounds on $f$-variation for \textit{all times} (as opposed to along a sparse sequence of times going to infinity as in~\cite{Hairer09}). To the best of our knowledge, such results were previously only available for Langevin diffusions~\cite[Thm~2.1]{Hairer09}, cf. Section~\ref{subsec:related_literature} below.

The remainder of the paper is structured as follows. 
Section~\ref{sec:main_results} states the novel $\mathbf{L}$-drift condition and formulates our main theorems for ergodic Markov processes. Section~\ref{sec:examples} presents the applications of the results of Section~\ref{sec:main_results} to classes of models, studied in~\cite{douc2009subgeometric,Fort2005}, exhibiting subexponential (and exponential) ergodicity. Sections~\ref{sec:return_times} and~\ref{sec:proofs} contain the proofs of the results of Section~\ref{sec:main_results}.
More precisely, Sections~\ref{sec:return_times} states and proves lower bounds for the tails of return times to bounded sets for c\`adl\`ag semimartingales.
Section~\ref{sec:proofs} applies these results in the context of c\`adl\`ag ergodic Markov processes to prove the theorems of Section~\ref{sec:main_results}.
The proofs of the  results, stated in Section~\ref{sec:examples}, for ergodic elliptic diffusions, L\'evy-driven stochastic differential equations and hypoelliptic stochastic damping Hamiltonian system  are in Section~\ref{sec:examples_proofs}. Finally, Section~\ref{sec:conclusion} offers concluding remarks, highlights some open questions and describes future directions of research.  (See~\cite{Presentation_AM} for short YouTube presentations describing the \href{https://youtu.be/r3eiRywC0js?si=XAuPbHMpDoe-bHud}{applications} and  \href{https://youtu.be/r3eiRywC0js?si=sJ_NA4GuQP8cVGcJ}{proofs} of our main results.)

\section{Main results}
\label{sec:main_results}
\subsection{Basic definitions}
\label{subsec:Definitions}
Let 
$X=(X_t)_{t\in\RP}$, where $\RP\coloneqq[0,\infty)$, be a strong Markov process on a filtered space 
$(\Omega,\cF,(\cF_t)_{t\in\RP})$, taking values 
in a locally compact separable metric space $\cX$, endowed with the Borel  $\sigma$-algebra $\cB(\cX)$.
For any $x\in\cX$,
denote by 
$\P_x$ and (resp. $\E_x$)  the associated probability measure (resp. expectation), 
satisfying $\P_x(X_0=x)=1$, and assume that $X$ has c\`adl\`ag  (i.e. right-continuous paths with left limits) paths.
The process $X$ is \textit{$\nu$-irreducible}
(resp. \textit{Harris recurrent}), where $\nu$ is a $\sigma$-finite measure on $\cB(\cX)$, if for every $A\in\cB(\cX)$, such that $\nu(A)>0$, and $x\in\cX$ we have $\E_x[\int_0^\infty \mathbbm{1}\{X_t\in A\}\ud t] >0$ 
(resp. $\P_x(\int_0^\infty \mathbbm{1}\{X_t\in A\}\ud t =\infty)=1$).
The probability measure $\pi$ on $(\cX,\cB(\cX))$ is an \textit{invariant measure} for $X$  if for all bounded measurable functions $g:\cX\to\RP$ and $t\in\RP$ we have 
$\int_{\cX}\E_x[g(X_t)]\pi(\ud x) = \int_{\cX} g(x)\pi(\ud x)$. A Harris recurrent process is  \textit{positive Harris recurrent} if it admits an invariant measure.
For any measurable function $f:\cX\to[1,\infty)$,
the \textit{$f$-variation} of a signed measure $\mu$ on $(\cX,\cB(\cX))$ is given by
 $\|\mu\|_{f} \coloneqq \sup\{\vert\int_\cX g(x)\mu(\ud x)\vert:g:\cX\to \R \text{ measurable, } |g|\leq f\}$.
 In the special case $f\equiv 1$ we obtain the \textit{total variation} $\|\mu\|_{\TV} \coloneqq\|\mu\|_f$ of the measure $\mu$. The process $X$ is \textit{ergodic} if $\lim_{t\to\infty}\|\P_x(X_t\in\cdot)-\pi(\cdot)\|_{\TV} = 0$ for all $x\in\cX$.

As we are interested in convergent Markov processes, unless explicitly stated otherwise, the following \textbf{standard assumption} 
 holds throughout the paper: 
$X=(X_t)_{t\in\RP}$ \textit{is an ergodic, positive Harris recurrent  Markov process with c\`adl\`ag paths and invariant measure $\pi$ on $(\cX,\cB(\cX))$.}

\subsection{Lower bounds: the tails of \texorpdfstring{$\pi$}{pi} and \texorpdfstring{$f$}{f}-variation}
The lower bounds on the tails of the invariant measure and the rate of convergence of $X$ will be implied by the following $\mathbf{L}$-\textit{drift condition}.

\begin{assumption*}[\textbf{L($V$,$\varphi$,$\Psi$)}]
\namedlabel{sub_drift_conditions}
Let $V:\cX \to [1,\infty)$  be a continuous function, such that, for all $x\in\cX$, $\limsup_{t\to\infty} V(X_t) = \infty$ $\P_x$-a.s.  Assume also there exists $\ell_0\in[1,\infty)$ such that~\ref{sub_drift_conditions(i)} and~\ref{sub_drift_conditions(ii)} hold.\\
\namedlabel{sub_drift_conditions(i)}{\textbf{(i)}}
Let $\varphi:(0,1] \to\RP$ 
be non-decreasing\footnote{A non-decreasing function may have intervals of constancy, while an increasing function does not; similarly for non-increasing and decreasing.} and continuous,
such that  $r\mapsto r\varphi(1/r)$ is decreasing on $[1,\infty)$
and $\lim_{r\to\infty}r\varphi(1/r)=0$.
Assume that for  $b\in\RP$ and any $x\in\cX$, the process
\begin{equation*}
\label{eq:submartingale_drift}
1/V(X)-\int_0^\cdot \varphi(1/ V(X_u))\ud u -b \int_0^\cdot \mathbbm{1}\{V(X_u)\leq\ell_0\}\ud u
\quad\text{is an $(\cF_t)$-supermartingale under $\P_x$.}
\end{equation*}
\namedlabel{sub_drift_conditions(ii)}{\textbf{(ii)}}
Let $\Psi:[1,\infty)\to[1,\infty)$ be a differentiable, increasing, submultiplicative\footnote{A function $\Psi:[1,\infty) \to[1,\infty)$ is
 \textit{submultiplicative} if it satisfies
$\Psi(r_1+r_2)\leq C\Psi(r_1)\Psi(r_2)$
for some constant $C\in(0,\infty)$ and all $r_1,r_2\in[1,\infty)$~\cite[Def.~25.2]{MR3185174}.}\label{footnote:submultiplicative} function satisfying the following: for any $\ell\in(\ell_0,\infty)$, there exists a 
constant $C_\ell\in(0,\infty)$ such that 
\begin{equation}
\label{eq:assumption_heavy_tail_bound}
\P_x(T^{(r)}<S_{(\ell)}) \geq 
C_\ell/\Psi(r) \quad \text{for all $r\in(\ell+1,\infty)$ and $x\in\{\ell+1\leq V\}$,}
\end{equation}
where $T^{(r)}\coloneqq\inf\{t\geq0:V(X_t)>r\}$ and $S_{(\ell)} \coloneqq \inf\{t\geq 0: V(X_t)<\ell\}$.
\end{assumption*}

In applications, the $\mathbf{L}$-drift condition~\nameref{sub_drift_conditions}
is verified via the infinitesimal generator of $X$, see Theorem~\ref{thm:generator}, Subsection~\ref{subsec:generator} and examples in Section~\ref{sec:examples} 
below.
In particular, the exit probability estimate in~\eqref{eq:assumption_heavy_tail_bound} is typically implied by a submartingale property of the (appropriately stopped) process $\Psi(V(X))$, see Lemma~\ref{lem:assumption_submart_exit_prob}  below for details.

Theorem~\ref{thm:invariant} gives a lower bound on the tail of invariant measure $\pi$ of the process $X$ under Assumption~\nameref{sub_drift_conditions}. Our lower bounds are in terms  of the functions $V$, $\varphi$, $\Psi$ and a logarithmic correction:\footnote{In this paper we adopt the convention that $\log x$
equals $1$ for $x\in(0,\mathrm{e})$
and  the natural logarithm on $[\mathrm{e},\infty)$.}
for any $\eps,q\in(0,1)$, define the function
\begin{equation}
\label{eq:def_L_eps_q}
    L_{\eps,q}(r)\coloneqq r\varphi(1/r)\Psi(2r/(1-q))(\log\log r)^{\eps},\quad\text{$r\in[1,\infty)$.}
\end{equation}
 
\begin{thm}[Tails of the invariant measure]
\label{thm:invariant}
Let Assumption~\nameref{sub_drift_conditions} hold. Then for any $q,\eps\in(0,1)$  there exists a constant $c_{\eps,q}\in(0,1)$ such that
\begin{equation}
\label{eq:main_result_invariant}
c_{\eps,q}/L_{\eps,q}(r)\leq \pi(x\in\cX:V(x) \geq r )\qquad\text{for all $r\in[1,\infty)$.}
\end{equation}
 \end{thm}

Theorem~\ref{thm:f_rate} provides a lower bound on the rate of $f$-variation convergence of $X$ to its invariant measure $\pi$ under Assumption~\nameref{sub_drift_conditions}. 

\begin{thm}[Lower bounds for $f$-variation]
\label{thm:f_rate}
Let Assumption~\nameref{sub_drift_conditions} hold. Consider a function $f:\cX\to[1,\infty)$, satisfying $f =f_\star \circ V$ for some differentiable $f_\star :[1,\infty)\to[1,\infty)$.
Assume further that~\ref{assumption:f_convergence_a} and~\ref{assumption:f_convergence_b} hold. 
\begin{myenumi}[label=(\alph*)]
\item\label{assumption:f_convergence_a}
Let a continuous function $h:[1,\infty)\to[1,\infty)$ be such that  the function $g\coloneqq h/f_\star$ is increasing (and without loss of generality $g(1)= 1$) and $\lim_{r\to\infty}g(r)=\infty$. Let  $v:\cX\times\RP\to[1,\infty)$ be increasing in the second argument and satisfy 
\begin{equation*}
\E_x[h\circ V(X_t)]\leq v(x,t)\quad\text{for all $x\in\cX$ and $t\in\RP$.}
\end{equation*}
\item\label{assumption:f_convergence_b}
Pick $\eps,q\in(0,1)$ and  a constant $c_{\eps,q}\in(0,1)$, such that the inequality in~\eqref{eq:main_result_invariant} holds with the function $L_{\eps,q}$, and consider a continuous function $a:[1,\infty)\to\RP$, satisfying
\begin{equation*}
a(t) \leq  f_\star(g^{-1}(t))c_{\eps,q}/L_{\eps,q}(g^{-1}(t))\quad\text{for all $t\in[1,\infty)$,}
\end{equation*} where $g^{-1}$ is the inverse of the increasing function $g$ in~\ref{assumption:f_convergence_a}. Suppose also that the function $A(t)\coloneqq ta(t)$ is increasing, $\lim_{t\to\infty} A(t) = \infty$ and denote its inverse by $A^{-1}$.
\end{myenumi}
Define the function $r_f:\cX \times [1,\infty)\to\RP$ by $r_f \coloneqq a\circ A^{-1}\circ (2 v)$. Then 
\begin{equation}
\label{eq:main_result_rate_of_convergence}
r_f(x,t)/2\leq \| \P_x(X_t\in \cdot )-\pi(\cdot)\|_{f} \quad \text{for all $x\in\cX$ and $t\in[1,\infty)$.}
\end{equation}
\end{thm}

We obtain the \textit{lower} bound on the rate of convergence in the total variation distance by choosing $f_\star\equiv 1$ (and hence $f=f_\star\circ V \equiv 1$) in the previous theorem.

\begin{cor}[Lower bounds for total variation]
\label{cor:rate}
Let Assumption~\nameref{sub_drift_conditions} hold and set $f_\star \equiv 1$.
Assume there exist
 continuous functions  $h,a:[1,\infty)\to\RP$ and $v:\cX\times\RP\to[1,\infty)$, satisfying conditions~\ref{assumption:f_convergence_a} and~\ref{assumption:f_convergence_b} in Theorem~\ref{thm:f_rate}.
Let $r_1:\cX\times [1,\infty)\to\RP$,  $r_1 \coloneqq a\circ A^{-1}\circ (2 v)$, be as in Theorem~\ref{thm:f_rate}. Then the following lower bound holds
\begin{equation}
\label{eq:result_rate_of_convergence_TV}
r_1(x,t)/2\leq \| \P_x(X_t\in \cdot )-\pi(\cdot)\|_{\TV} \quad \text{for all $x\in\cX$ and $t\in[1,\infty)$}.
\end{equation}
\end{cor}

\begin{rem}
\label{rem:after_main_theorems}
\phantomsection
\begin{myenumi}[label=(\alph*)]

\item \label{rem:after_main_theorems(a)}
The lower bound in Theorem~\ref{thm:f_rate} is obtained by comparing the tail (with respect to $V$) of the invariant measure $\pi$  to the tail of marginal distribution of $V(X_t)$ at time $t$.  The function $h$ in condition~\ref{assumption:f_convergence_a} of Theorem~\ref{thm:f_rate} aims to maximise the growth of $r\mapsto r\pi(\{h\circ V\geq r\})$ as $r\to\infty$ with respect to the growth of $t\mapsto \E_x[h\circ V(X_t)]$  as $t \to\infty$, since this expectation controls the tail of $h\circ V(X_t)$ via Markov's inequality (see also Remark~\ref{rem:lem_lower_bound} below for more details).

\item 
Given a Lyapunov function $V$, our methods generate lower bounds on the $f$-variation distances with the property that the level sets of functions $f$ form a subset of the level sets of $V$. This is analogous to results concerning the upper bounds using a Lyapunov function $V$, see e.g.~\cite[Thm~3.2]{douc2009subgeometric}.

\item The $\mathbf{L}$-drift condition~\nameref{sub_drift_conditions} is not restricted to Markov processes with a subexponential invariant measure and rates of convergence. Indeed, in Section~\ref{subsubsec:Exponential} below, Theorems~\ref{thm:invariant} and~\ref{thm:f_rate}  yield exponential lower bounds both on the tails of the invariant measures and the convergence rates in a class of elliptic diffusion models where 
exponential upper bounds are known.

\item By Theorem~\ref{thm:invariant},  
the function $a(t)$ in Assumption~\ref{assumption:f_convergence_b} of Theorem~\ref{thm:f_rate} (with $f\equiv1$) provides a lower bound on the tail of $\pi$. In applications it is often simpler to work with some function $a$ rather than the actual lower bound on the tail of $\pi$ provided by Theorem~\ref{thm:invariant},
since the rate of convergence in  Theorem~\ref{thm:f_rate} is given in terms of the inverse of $t\mapsto A(t)=ta(t)$. In particular, under assumptions of Theorem~\ref{thm:f_rate}, inequality~\eqref{eq:main_result_invariant} implies  $0<f_\star(r)c_{\eps,q}/L_{\eps,q}(r)\leq f_\star(r) \pi(V\geq r)\to0$ as $r\to\infty$, making
the growth of $A(t)$ (as $t\to\infty$) sublinear. This typically leads to subexponential decay of $a\circ A^{-1}$ and thus subexponential lower bounds on $f$-variation.

\item The iterated logarithm  term in~\eqref{eq:def_L_eps_q}, and thus in~\eqref{eq:main_result_rate_of_convergence}, is an artefact of the proof of Theorem~\ref{thm:invariant}, where the lower bounds on modulated moments are used to establish  lower bounds on the invariant measure $\pi$. In all our examples in Section~\ref{sec:examples}, the iterated logarithm term is negligible, suggesting that it will not affect other applications. While we cannot fully remove this term, it is possible to modify the proof so that only
arbitrarily many (instead of two) iterations of the logarithm remain.
\item The submultiplicative function~\cite[Def.~25.2]{MR3185174} (cf. footnote on page~\pageref{footnote:submultiplicative})  $\Psi$ in~\nameref{sub_drift_conditions} transforms the process $V(X)$ into a submartingale. The function $\Psi$  features in lower bounds in our theorems through definition~\eqref{eq:def_L_eps_q}.
Since a product of submultiplicative functions is submultiplicative and, by~\cite[Prop.~25.4]{MR3185174}, $r\mapsto g(cr+\gamma)^\alpha$ (where $c,\gamma,\alpha>0$) is submultiplicative if $g$ is, it  follows that  
$r\mapsto r^a (\log r)^d \exp(b r^p) $ is submultiplicative for (i) $b>0$ and $p\in(0,1)$, (ii)  $b=0$ and $a>0$ or (iii) $b=a=0$ and $d\geq0$.
This form is similar to the subgeometric rate functions used in~\cite{douc2009subgeometric, Fort2005}. In fact,  the rate functions in~\cite{douc2009subgeometric, Fort2005}  are submultiplicative~\cite[Eq.~(5)]{Touminen94}, a crucial fact used in the proofs of~\cite{douc2009subgeometric, Fort2005}. 
\end{myenumi}
\end{rem}

\subsection{Modulated moments for the process}
A classical approach to the stability of Markov processes relies heavily  on decomposing the path of the process  as the sum of excursions from some petite set (see definition of a petite set in~\eqref{eq:petite} in Section~\ref{subsec:return_times} below). Many results, including bounds on the invariant measure, the rate of convergence and moderate deviations  have been established using such decomposition (see~\cite{Douc08} and the references therein for more details). 
It is thus natural for quantitative bounds on the tail of the \textit{modulated moments} (i.e. expectations of additive functionals of excursions from petite sets) to contain essential information necessary for bounding the rate of convergence to the invariant measure and related quantities.

The upper bounds on the modulated moments are well understood, see e.g.~\cite{douc2009subgeometric,Douc08}. In contrast, very little is known about the lower bounds on modulated moments. There are some known results regarding the finiteness of return time moments, e.g.~\cite{menshikovWillliams96}, which however are not sufficiently strong to either characterize the tail behavior of the return time or  complement the findings of~\cite[Thm 4.1]{douc2009subgeometric} for general ergodic Markov processes. In the following theorem, we both generalize and strengthen these results by accommodating a broader range of processes and providing lower bounds for the tail behavior. Moreover, in Section~\ref{sec:examples}, we show that, for a wide range of  models used in applications, our lower bounds on the return times match the upper bounds from~\cite{douc2009subgeometric}.

For a set $D\in \cB(\cX)$ and $\delta\in(0,\infty)$, let 
$\tau_D(\delta) \coloneqq \inf\{t>\delta: X_t\in D\}$ (with convention $\inf\emptyset \coloneqq\infty$)
be the first hitting time of $D$ after time $\delta$ (recall that $\tau_D(\delta)$ is an $(\cF_t)$-stopping time by~\cite[Thm~1.27]{Jacod2003}). In Theorem~\ref{thm:modulated_moments} we provide, under the \textbf{L}-drift condition~\nameref{sub_drift_conditions},
lower bounds on the tails of return times to arbitrary subsets of sublevel sets of the Lyapunov function~$V$.
Crucially, in Section~\ref{subsec:return_times} below, we prove that, under the \textbf{L}-drift condition, all petite sets are contained in the sublevel sets of $V$, making Theorem~\ref{thm:modulated_moments} applicable to any petite set of $X$.

\begin{thm}
\label{thm:modulated_moments}
Let Assumption~\nameref{sub_drift_conditions} hold. Consider a set $D\in \cB(\cX)$, with $D \subset \{V\leq m\}$ for some $m\in(1,\infty)$, and fix $q\in(0,1)$ and $\eps = (1-q)/2$. Then the following statements hold.
\begin{myenumi}[label=(\alph*)]
\item\label{thm:modulated_moments_a}
Let $h:[1,\infty)\to[1,\infty)$ be a non-decreasing continuous function and  $G_h$ the inverse of the increasing continuous function $v\mapsto \eps h(v)/(v\varphi(1/v))$ on $[1,\infty)$. Then
for every $x\in\cX$ there exist constants $C,r_0,\delta\in(0,\infty)$, such that 
$$
\P_x\left(\int_0^{\tau_{D}(\delta)} h\circ V(X_s)\ud s\geq r\right)\geq \frac{C}{\Psi(2G_h(r)/(1-q))} \quad \text{for all $r\in(r_0,\infty)$.}
$$
\item\label{thm:modulated_moments_b}
Let $G_1$ be the inverse of the increasing continuous function $v\mapsto \eps/(v\varphi(1/v))$ on $[1,\infty)$. Then for every $x\in\cX$ there exist constants $C,\delta,t_0\in(0,\infty)$, such that 
$$
\P_x(\tau_D(\delta)\geq t)\geq \frac{C}{\Psi(2G_1(t)/(1-q))}\quad \text{for all $t\in(t_0,\infty)$.}
$$
\end{myenumi}
\end{thm}

\begin{rem}
\label{rem:modulated_every_delta}
   The proof of Theorem~\ref{thm:modulated_moments} in Section~\ref{sec:proofs} below in fact shows that  the inequalities in Theorem~\ref{thm:modulated_moments} hold for  every $\delta\in(0,\infty)$, satisfying $\P_x(V(X_\delta)>r)>0$
    for all $x\in\cX$ and $r\in[1,\infty)$.
   This condition holds for all $\delta>0$ in the models of Section~\ref{sec:examples} below, because their marginal distributions at  positive times have full support with respect to the Lebesgue measure.
\end{rem}

\subsection{How are  Theorems~\ref{thm:invariant},~\ref{thm:f_rate} and~\ref{thm:modulated_moments} applied in practice?}
\label{subsec:generator}
Continuous-time Markov processes,
where upper bounds on the rate of convergence have been established, are typically Feller~\cite{douc2009subgeometric,Fort2005}. It is thus natural to give sufficient conditions for the assumptions of Theorems~\ref{thm:invariant},~\ref{thm:f_rate} and~\ref{thm:modulated_moments} in terms of the infinitesimal characteristics of $X$ expressed via its extended generator. In this section we first provide tools for verifying the assumptions of our main theorems using the  generator of the process and then discuss their application in practice. 

\subsubsection{Generators and drift conditions}
\label{subsubsec:generators_and_drift} Following the monograph~\cite[Ch~1, Def~(14.15)]{Davis}, let $\cD(\cA)$ denote the set of measurable functions $g:\cX \to \R$ with the following property: there exists a measurable $h:\cX\to\R$, such that, for each $x\in\cX$, $t\rightarrow h(X_t)$ is integrable $\P_x$-a.s. 
 and the process $$
g(X)-g(x)-\int_0^\cdot h(X_s)\ud s\quad\text{is a $\P_x$-local martingale.}
$$
 Then we write $h = \cA g$ and call $(\cA,\cD(\cA))$ the \textit{extended generator} of the process $X$. Define the left limit at $t\in(0,\infty)$ of the process $X$ by $X_{t-} \coloneqq\lim_{s\uparrow t} X_s$ and $X_{0-} = X_0$. The following theorem provides a sufficient condition for the validity of Assumptions~\ref{sub_drift_conditions(i)} and~\ref{sub_drift_conditions(ii)} in the $\mathbf{L}$-drift condition~\nameref{sub_drift_conditions}.

\begin{thm}
\label{thm:generator} Let a continuous $V:\cX \to [1,\infty)$  satisfy  $\limsup_{t\to\infty} V(X_t) = \infty$ $\P_x$-a.s.,
 all $x\in\cX$.
\begin{myenumi}[label=(\alph*)]
    \item \label{generator_a}  
    Let $\varphi:(0,1]\to\RP$ be a non-decreasing, continuous function, such that $r\mapsto 1/(r\varphi(1/r))$ is  increasing on $[1,\infty)$ and $\lim_{r\to\infty}1/(r\varphi(1/r))=\infty$.
    If $1/V\in\cD(\cA)$ and there exist $b,\ell_0\in(0,\infty)$ such that
\begin{equation}
    \label{eq:generator_condition_super}
\cA (1/V)(x)\leq \varphi (1/V(x)) + b\1{V(x) \leq \ell_0}\quad \text{for all $x\in\cX$,}
\end{equation}
then the process in~\nameref{sub_drift_conditions}\ref{sub_drift_conditions(i)} is a supermartingale.
\item \label{generator_b}
Let $\Psi:[1,\infty)\to[1,\infty)$ be a differentiable, increasing, submultiplicative function. Assume $X$ has bounded jumps: for some constant $d\in\RP$, we have $\P_x\left(V(X_t)-V(X_{t-})\leq d \text{ for all $t\in\RP$}\right)=1$ for each $x\in \cX$. If $\Psi\circ V\in\cD(\cA)$ and there exist $c,\ell_0\in(0,\infty)$ such that
\begin{equation}
    \label{eq:generator_condition_sub}
    \cA(\Psi \circ V)(x) \geq - c\1{V(x)\leq \ell_0}\quad \text{for all $x\in\cX$,}
\end{equation}
then inequality~\eqref{eq:assumption_heavy_tail_bound} in~\nameref{sub_drift_conditions}\ref{sub_drift_conditions(ii)} holds for every $\ell\in(\ell_0,\infty)$ and some constant $C_\ell\in(0,\infty)$.
\end{myenumi}
\end{thm}

 If $X$ has bounded jumps, 
 Theorem~\ref{thm:generator} has a natural converse:
Assumption~\nameref{sub_drift_conditions}\ref{sub_drift_conditions(i)} 
and the submartingale property in~\eqref{eq:subm_mart_bounded_jumps} of Lemma~\ref{lem:assumption_submart_exit_prob} below (which implies~\nameref{sub_drift_conditions}\ref{sub_drift_conditions(ii)})
yield the inequalities in~\eqref{eq:generator_condition_super} and~\eqref{eq:generator_condition_sub} involving
$\cA(1/V)$ and $\cA(\Psi\circ V)$, respectively. As this fact is not used in the paper, the details are omitted.

We proceed with Lemma~\ref{lem:assumption_submart_exit_prob}, which  provides a key step in the proof of Theorem~\ref{thm:generator}\ref{generator_b}. We state it here because it is of independent interest in applications as it gives a  sufficient condition for the bound on the exit probability in~\nameref{sub_drift_conditions}\ref{sub_drift_conditions(ii)} in terms of  the submartingale condition in~\eqref{eq:subm_mart_bounded_jumps} below. Denote 
$t\wedge s\coloneqq \min\{s,t\}$, $t,s\in\RP$, and recall  $T^{(r)} = \inf\{t\geq 0: V(X_t)>r\}$, $r\in\RP$.

\begin{lem}
\label{lem:assumption_submart_exit_prob}
Let a continuous $V:\cX \to [1,\infty)$  satisfy  $\limsup_{t\to\infty} V(X_t) = \infty$ $\P_x$-a.s.,
for all $x\in\cX$.
 Let $\Psi:[1,\infty)\to[1,\infty)$ be a differentiable, increasing, submultiplicative function.
Assume that for some $d\in\RP$, we have $\P_x(V(X_t)-V(X_{t-})\leq d\text{ for all }t\in\RP)=1$ for each $x\in \cX$. 
If for 
some $\ell_0,c\in(0,\infty)$ and all $r\in(\ell_0,\infty)$, the process 
\begin{equation}
\label{eq:subm_mart_bounded_jumps}
\Psi \circ V(X_{\cdot\wedge T^{(r)} }) + c\int_0^{\cdot\wedge T^{(r)}}\1{V(X_u) \leq\ell_0}\ud u\quad\text{is an $(\cF_t)$-submartingale under $\P_x$}
\end{equation}
for all $x\in\cX$,
then the condition~\nameref{sub_drift_conditions}\ref{sub_drift_conditions(ii)} holds with  functions $V$ and $\Psi$.
\end{lem}

\begin{rem}
\label{rem:why_prob_L(ii)}
If $X$ has jumps with heavy tails, a submartingale argument of Lemma~\ref{lem:assumption_submart_exit_prob} may fail to imply the inequality in Assumption~\nameref{sub_drift_conditions}\ref{sub_drift_conditions(ii)}. This is because the overshoot of the process $\Psi\circ V(X)$  need not  be integrable
(see e.g. the class of models in Section~\ref{sec:levy} below).
It is thus crucial that  condition~\ref{sub_drift_conditions(ii)} in Assumption~\nameref{sub_drift_conditions} is given in terms of the probability $\P_x(T^{(r)}<S_{(\ell)})$ directly, rather than the submartingale property of $\Psi\circ V(X)$. As demonstrated in Section~\ref{sec:levy}, in such heavy-tailed cases it is possible to apply path-wise arguments directly to obtain the lower bound on $\P_x(T^{(r)}<S_{(\ell)})$, see Section~\ref{subsec:levy_proofs} for details.
\end{rem}

The expected growth condition assumed in 
Theorem~\ref{thm:f_rate}\ref{assumption:f_convergence_a} is easily verified via Lemma~\ref{lem:bounded_generator}.

\begin{lem}
\label{lem:bounded_generator}
Let  $H:\cX\to[1,\infty)$  be continuous with $H\in\mathcal{D}(\cA)$ and let $\xi:[1,\infty)\to[1,\infty)$ be concave, non-decreasing and differentiable, satisfying  $\cA H\leq \xi\circ H$  on $\cX$.
Define the function $\Xi(u)\coloneqq \int_1^u \ud s/\xi(s)$ for $u\in[1,\infty)$.
Then we have
$$\E_x[H(X_t)]\leq \Xi^{-1}(\Xi(H(x))+t)\quad\text{for all $x\in\cX$ and $t\in\RP$.}$$
\end{lem}

Lemma~\ref{lem:bounded_generator} plays a key role in establishing the upper bound on the expected growth of the process $h\circ V(X)$ in condition~\ref{assumption:f_convergence_a} of Theorem~\ref{thm:f_rate}. If $\xi$ is constant, the bound on $\E_x[h\circ V(X_t)]$ is linear in time. Crucially Lemma~\ref{lem:bounded_generator}
permits an unbounded $\xi$, thus allowing the process $h\circ V(X)$ to exhibiting superlinear expected growth. This is key for establishing matching stretched exponential (see Sections~\ref{subsubsec_Subexponential} and~\ref{subsubsec:subexp_proofs} below) and exponential (Sections~\ref{subsubsec:Exponential} and~\ref{subsubsec:exponential_proofs} below)  lower bounds on the rate of convergence.

\subsubsection{How to find functions $\varphi$, $\Psi$ (in~\nameref{sub_drift_conditions}) and $h$ (in Theorem~\ref{thm:f_rate}\ref{assumption:f_convergence_a}) for a given Lyapunov function $V$?}
\label{subsubsec:applications_in_practice}
The classical Lyapunov drift condition~\cite{douc2009subgeometric,Fort2005,meyn1993stability1} requires the process $V(X)$ to satisfy the supermartingale condition in~\eqref{eq:douc_drif_condition} below. 
If such $V$ is available, verifying~\nameref{sub_drift_conditions} reduces to finding scalar functions $\varphi$ and $\Psi$, such that $1/V(X) - \int_0^\cdot \varphi \circ (1/V)(X_s)\ud s$ and $\Psi\circ V(X)$ are a super- and a submartingale, respectively. Since $\Psi$ is increasing, the Lyapunov function $V:\cX\to[1,\infty)$ determines the level sets of $\Psi \circ V$, while $\Psi:[1,\infty)\to[1,\infty)$ modulates only the growth of $\Psi \circ V$. Thus, identifying $\Psi$  is typically  straightforward if $V$ is given and the task is to find the slowest growing $\Psi$ so that $\Psi \circ V(X)$ is a submartingale. 
Identifying $\varphi$ for a given $V$ is also typically a simple task, as it reduces to bounding the drift of $1/V(X)$ using Theorem~\ref{thm:generator}\ref{generator_a} in
Subsection~\ref{subsubsec:generators_and_drift}.

The lower bound on the rate in Theorem~\ref{thm:f_rate} is based on the comparison of the tails of $\pi$ and the law of $X_t$.
Recall from Remark~\ref{rem:after_main_theorems}\ref{rem:after_main_theorems(a)} that the role of the function $h$ 
in assumption~\ref{assumption:f_convergence_a} of 
Theorem~\ref{thm:f_rate}
is to balance a lower bound on $r\mapsto r\pi(\{h\circ V\geq r\})$ and an upper bound on the $t\mapsto \E_x[h\circ V(X_t)]$. Typically, a good choice for $h$ is such that $r\mapsto r\pi(\{h\circ V\geq r\})$ grows to infinity as $r\to\infty$ at a polynomial (necessarily sublinear) rate, see Remark~\ref{rem:lem_lower_bound} below for more details. Lemma~\ref{lem:bounded_generator}, applied to $H=h\circ V$, yields  the desired upper bound on $t\mapsto \E_x[h\circ V(X_t)]$ (and hence by Theorem~\ref{thm:f_rate} a lower bound on the rate of convergence).

\subsection{Sketch of the proofs of the main results: Theorems~\ref{thm:invariant},~\ref{thm:f_rate} and~\ref{thm:modulated_moments}}
\label{subsec:sketch}
The $\mathbf{L}$-drift condition~\nameref{sub_drift_conditions} is the crucial ingredient of all the main theorems. The  following implications constitute  key steps in their proofs: 

\begin{table}[htp] 
\centering
\begin{tabular}{lclclcl}
$\nameref{sub_drift_conditions}$ & $\xRightarrow{\text{(I)}}$ & \begin{tabular}[c]{@{}l@{}@{}@{}}Lower bounds \\on the tails of \\ return times to\\ bounded sets\\ (Thm~\ref{thm:modulated_moments})\end{tabular} & $\xRightarrow{\text{(II)}}$ & \begin{tabular}[c]{@{}l@{}}Lower bounds\\ on the tails of\\  invariant measure\\ (Thm~\ref{thm:invariant})\end{tabular} 
 &  $\xRightarrow{\text{(III)}}$ & \begin{tabular}[c]{@{}l@{}}Lower bounds on\\ convergence rates\\ (Thm~\ref{thm:f_rate})\end{tabular} 
\end{tabular}
\caption{The sequence of implications from $\mathbf{L}$-drift conditions to convergence rates.}
\label{tab:implications}
\end{table}

\begin{figure}[hbt]
\centering
    \includegraphics[width=150mm]{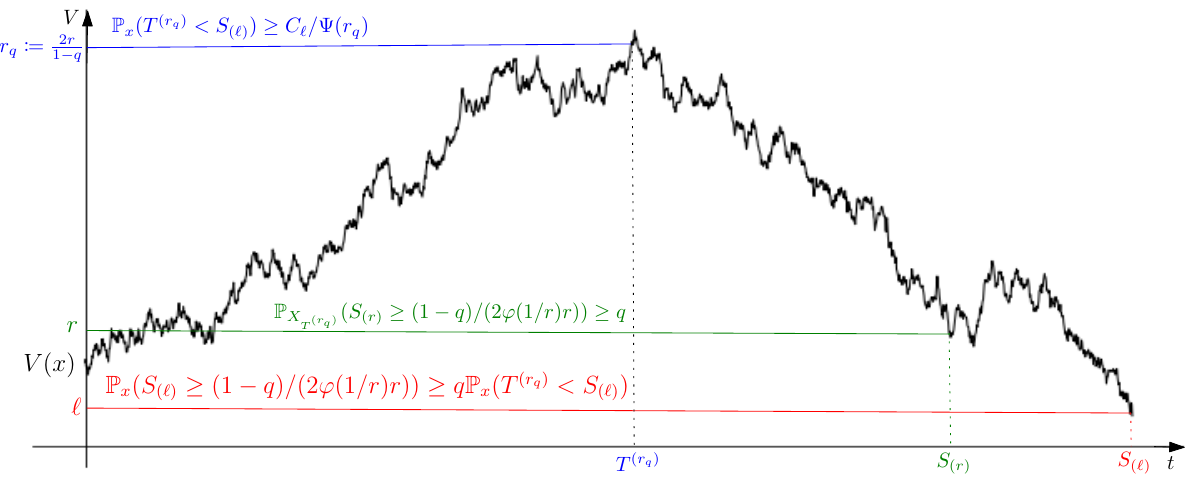}
    \caption{Establishing the lower bound in Theorem~\ref{thm:modulated_moments} (implication~(I)) requires  conditions~\ref{sub_drift_conditions(i)} and~\ref{sub_drift_conditions(ii)} in Assumption~\nameref{sub_drift_conditions} to obtain lower bounds on the \textit{duration} $S_{(\ell)}$ and the \textit{height} $\sup_{s\in[0,S_{(\ell)}]}V(X_s)$ of an excursion from the compact set $\{V\leq\ell\}$. The bound on the height (blue {\color{blue}inequality}) follows directly from~\nameref{sub_drift_conditions}\ref{sub_drift_conditions(ii)}, while the lower bound on the duration (red {\color{red}inequality}) requires a more involved argument (in Lemma~\ref{lem:return_times} below) using the supermartingale in~\nameref{sub_drift_conditions}\ref{sub_drift_conditions(i)}.}
    \label{fig:L_functions}
\end{figure}

We now describe informally each of these implications. Implication (I) in Table~\ref{tab:implications} above requires a lower bound on the return time of $X$ to a sublevel set $\{V\leq \ell\}$ of the Lyapunov function $V$ (see Figure~\ref{fig:L_functions}). 
Under the $\mathbf{L}$-drift condition~\nameref{sub_drift_conditions}\ref{sub_drift_conditions(i)},
the process $1/V(X) - \int_0^\cdot \varphi \circ (1/V)(X_s)\ud s$ 
is a supermartingale when $V(X)$ is above level $\ell$, implying an upper bound on the tail of the supremum of $1/V(X)$. This upper bound can be converted into a lower bound on the tail of the return time of $V(X)$ below a level $\ell$.
The argument requires only the supermartingale property, see Lemma~\ref{lem:return_times} below for details. 
In short, we first wait until the process $V(X)$, started from $V(x)>\ell$, reaches a large level $r_q=2r/(1-q)$ for $r\gg\ell$ and any $q\in(0,1)$, before descending below $\ell$.
By~\nameref{sub_drift_conditions}\ref{sub_drift_conditions(ii)} we have the blue {\color{blue}lower bound} on the exit probability in Figure~\ref{fig:L_functions}. 
On this event, once the process $V(X)$ is at $V(X_{T^{(r_q)}})$, we apply the supermartingale property in~\nameref{sub_drift_conditions}\ref{sub_drift_conditions(i)} to obtain the green {\color{ForestGreen}lower bound} on the tail of the return time $S_{(r)}$, which implies the red {\color{red}lower bound}.  
Since $q\in(0,1)$ is arbitrary, these estimates imply boundedness of any petite set of $X$ (Lemma~\ref{lem:bounded_petite_sets} below) as well as lower bounds on  the  tail probability 
$\P_x\left(\int_0^{\tau_D(\delta)}h\circ V(X_s)\ud s\geq r\right)$
of the additive functional of the excursion from any bounded set $D$, for $x\in D$, $\delta>0$ and $\tau_D(\delta)= \inf\{t>\delta:X_t\in D\}$ (Theorem~\ref{thm:modulated_moments}).

Implication (II) in Table~\ref{tab:implications} above uses the fact that, under~\nameref{sub_drift_conditions}, all petite sets of $X$ are bounded  
and, crucially, that the lower bounds on the probability $\P_x\left(\int_0^{\tau_D(\delta)}h\circ V(X_s)\ud s\geq r\right)$ hold  for \textit{all} non-decreasing functions $h:[1,\infty)\to[1,\infty)$. These facts, together with the well-known characterisation of Meyn and Tweedie in~\cite[Thm~1.2(b)]{tweedie1993generalized}  of the integrability with respect to $\pi$,
yield a lower bound on the tail of the invariant measure $\pi$ in Theorem~\ref{thm:invariant}. 

Implication (III) in Table~\ref{tab:implications} above requires a comparison of the lower bound on the tail of the invariant measure $\pi$, obtained via implication (II), with the upper bound on the tail of the law of $V(X_t)$, controlled by the expected growth of the process $V(X)$.
Finally, an application of  Lemma~\ref{lem:lower_bound_f_convergence_rate}, which is generalisation to $f$-variation norms of~\cite[Thm~3.6]{Hairer09} (see also~\cite[Thm~5.1]{hairer2010convergence}), yields our lower bound on the rate of convergence in Theorem~\ref{thm:f_rate}.

\subsection{Related literature}
\label{subsec:related_literature}
 \paragraph{\textbf{Lyapunov functions and lower bounds.}} 
A general Lyapunov-function approach to 
lower bounds on the total variation distance along a sequence of times  tending to infinity is developed in~\cite{Hairer09}. A key step~\cite[Thm~3.6]{Hairer09} consists of converting lower bounds on the tails of the invariant measure to the lower bounds on the total variation, an idea also exploited in the present paper (see its generalisation from total to $f$-variation in Lemma~\ref{lem:lower_bound_f_convergence_rate} below). Crucially, in~\cite[Thm~3.2]{Hairer09}, Hairer gives criteria  (based on ideas from~\cite{Wonham66}) for establishing lower bounds on total variation along a sequence of times tending to infinity when the invariant measure is not  explicitly known. These criteria can be  summarised briefly  as follows: let $W_1$ and $W_2$ be $C^2$-functions on $\R^n$ such that $W_1/W_2\to0$  as $W_1\to\infty$ and  the inequalities $\cA W_1\geq 0$  and $\cA W_2\leq F$ hold outside a compact set for some function $F$, where $\cA$ denotes the generator of $X$. Under these conditions~\cite[Thm~3.2]{Hairer09} implies $\int F\ud\pi = \infty$ and, together with~\cite[Thm~3.6]{Hairer09},  yields a sequence of times $t_n\to\infty$ at which a lower bound on the total variation (based on non-integrability of $F$) holds.

The critical step in~\cite{Hairer09} consists of finding a non-integrable function $F$, which essentially amounts to the function $\cA W_2$ being non-integrable: $\int \cA W_2\ud \pi=\infty$. Consequently, the lower bounds 
on the tails of $\pi$ 
hold only along a (possibly sparse) sequence of levels, yielding a (possibly sparse) sequence of times 
$t_n\to\infty$ at which the lower bound on the total variation can be established. In contrast, our approach 
 provides lower bounds for all times by essentially analysing lower bounds as upper bounds of $1/V(X)$ (as usual $V$ denotes a Lypuanov function for $X$). 
In the language of~\cite{Hairer09} this involves estimating the drift $\cA(1/V)$, and then using \textit{path-wise} arguments for semimartingales (Lemma~\ref{lem:return_times} below) to characterise  $\int h\circ V\ud \pi=\infty$ for \textit{all} increasing scalar functions $h$. 
This idea allows us to establish
 lower bounds on the tails of $\pi$ for all levels and hence get  lower bounds on the convergence rate in $f$-variation for all times. 
Moreover, this suggests that establishing lower bounds for all times using the approach in~\cite{Hairer09} directly would at the very least require a sufficiently rich family of functions $W_2^{(\kappa)}$, satisfying the assumptions of~\cite[Thm~3.2]{Hairer09}  (and hence $\int \cA W_2^{(\kappa)}\ud \pi=\infty$) for all  values of the parameter $\kappa$.

In the context of elliptic diffusions studied in Section~\ref{sec:diffusions} below, our results may be viewed as a generalisation of the bounds along $t_n\to\infty$, obtained via~\cite[Thm~3.2]{Hairer09}, to all times $t\in[1,\infty)$.
If the process $X$ has jumps with heavy tails as in  Section~\ref{sec:levy} below, \cite[Thm~3.2]{Hairer09} appears to be difficult to apply:   for the L\'evy-driven stochastic differential equation of Section~\ref{sec:levy}, $W_1$
is either not integrable with respect to the jump measure of $X$  or, if it is, we have $\cA W_1< 0$  outside of a compact (see the formula for $\cA$ in equation~\eqref{eq:generator_Levy} and Remark~\ref{rem:levy_integrability} below for details). This makes the condition $0\leq\cA W_1<\infty$ hard to satisfy outside of any compact set. 
Finally, 
it is feasible that our methods could yield novel insights for the hypoelliptic Hamiltonian system studied in~\cite{Hairer09} and possibly extend the result on lower bounds in~\cite[Thm~1.1]{Hairer09} to all times.
The function $W_1$ in~\cite[eq.~(5.15)]{Hairer09}, satisfying $\cA W_1\geq 0$ and thus making $W_1(X)$ a submartingale, is a natural candidate for $\Psi\circ V$ in the $\mathbf{L}$-drift condition~\nameref{sub_drift_conditions}\ref{sub_drift_conditions(ii)}. Determining the function $V$ and estimating the asymptotic behaviour of $\cA (1/V)$, so that~\nameref{sub_drift_conditions}\ref{sub_drift_conditions(i)} holds, would enable the application of our results. This is left for future research.

\paragraph{\textbf{Lyapunov functions and upper bounds.}}
The rate of convergence to invariant measures of ergodic Markov processes has been  studied extensively.  The majority of the modern literature, 
based on a  probabilistic approach using Lyapunov functions,
dates back to the seminal work of Meyn and Tweedie in the 1990s~\cite{meyn1993stability1,meyn1993stability2,tweedie1993generalized} and primarily focuses on the upper bound estimates. These results have since been further improved and generalised. Notable contemporary versions can be found in~\cite{douc2009subgeometric, bakry2008rate} and Hairer's lecture notes~\cite{hairer2010convergence}, see also monograph~\cite{Moulin18} and the references therein. Briefly, if 
for some continuous $V:\cX\to[1,\infty)$, increasing, differentiable, concave $\phi:[1,\infty)\to[1,\infty)$, a closed petite set (defined in~\eqref{eq:petite} below) $A\in\cB(\cX)$ and a constant $b\in\RP$, the process
\begin{equation}
\label{eq:douc_drif_condition}
V(X) + \int_0^\cdot\phi\circ V (X_u)\ud u - b\int_0^\cdot \mathbbm{1}\{X_u\in A\}\ud u\qquad\text{is a supermartingale,}
\end{equation}
the following statements hold (see e.g.~\cite{douc2009subgeometric}):
\begin{myenumi}
\item[(a)] the process $X$ is Harris recurrent with invariant  measure $\pi$ and $\int_{\cX} \phi \circ V(x)\pi(\ud x) <\infty$;
\item[(b)]  $r_{*}(t)\|\P_x(X_t \in \cdot) - \pi(\cdot)\|_{\text{TV}} \leq C V(x)$, where $r_{*}(s) \coloneqq \phi\circ H_{\phi}^{-1}(s)$ and $H_{\phi}(u) = \int_1^u \frac{\ud s}{\phi(s)}$, $u\geq 1$.
\end{myenumi}

 Since this approach relies on  transforming the state space with a Lyapunov function $V$, key information may be lost,  potentially resulting in poor upper bound estimates on the convergence rate (see the motivating example in Section~\ref{sec:diffusions} below; see also Example~\ref{example:missmatch}). 
This naturally motivates a general study of \textit{lower} bounds in the context of Lyapunov functions presented in this paper, particularly in the case when the upper bounds on the rate of convergence are subexponential.   As explained in Section~\ref{subsec:generator}, our results naturally augment the existing Lyapunov function approach for the stability of Markov processes and provide a robust method for checking the quality of upper bound estimates obtained via a given Lyapunov function. 

\paragraph{\textbf{Poincar\'e inequalities and the rates of convergence.}} These functional analytic techniques typically work directly with the infinitesimal generator and are thus not dependent on the potentially suboptimal choice of a Lyapunov function (see e.g.~\cite{Rockner02,bakry2008rate}). However, due to their analytical nature, such methods often require more restrictive (typically global rather than local) assumptions on the behavior of the transition operator of the process. Additionally, even under strong assumptions (e.g. reversibility),  the literature addressing lower bounds remains sparse.

\paragraph{\textbf{Potential theory and  L\'evy-driven Ornstein-Uhlenbeck (OU) processes.}} A potential theoretic approach for the stability of a class of  L\'evy-driven OU-processes, 
arising as limits of multiclass many-server queues, has been developed in~\cite{Sandric19,Sandric2022}. In contrast to our results, the lower bounds~\cite[Thm~1.2]{Sandric2022} for the general case of their models  hold along a sequence of times $t_n\to\infty$. The theory in~\cite{Sandric2022} gives no information on the sparsity of the sequence $(t_n)_{n\in\N}$.
However, in some special cases of the models in~\cite{Sandric19,Sandric2022}, matching lower and upper bounds for all $t\in[1,\infty)$ are established using Hairer's result~\cite[Thm~5.1]{hairer2010convergence}. In these special cases, the precise decay of the tails of the invariant measure of the model in~\cite{Sandric19} is  established using analytical methods. 

\paragraph{\textbf{Mixing coefficients of Rosenblatt and Kolmogorov}}
Convergence to stationarity of an ergodic Markov process can also be quantified via  Rosenblatt's and Kolmogorov's mixing coefficients. Veretennikov~\cite{veretenn97} provides upper bounds on the  coefficients for polynomialy ergodic elliptic diffusions, cf. Section~\ref{subsubsec:Polynomial} below.  In the special case when the noise is additive, the corresponding lower bounds on the coefficients are given in~\cite{Veretennikov06} (for discrete-time  additive-noise case see~\cite{Klokov}). 

\paragraph{\textbf{Markov Chain Monte Carlo.}} The vast majority of the convergence theory for
Markov Chain Monte Carlo algorithms
concerns upper bounds, see e.g.~\cite{Moulin18} and the references therein. In comparison, lower bounds on the rate of convergence are sparse. Some recent model dependent results are given in~\cite{Brown22,brown24}. However, the tools used appear to be hard to extend to general ergodic Markov processes discussed in this paper.

\paragraph{\textbf{Ergodic theory.}} Convergence of a Markov process to its invariant measure  is closely related to the decay of correlations and the mixing  of dynamical systems in ergodic theory.
In this setting, seminal papers~\cite{sarig02,Melbourne12} establish asymptotically matching polynomial upper and lower bounds on the mixing rates using \textit{operator renewal theory}.
This approach extends a classical probabilistic result of Rogozin, bounding the remainder term in the renewal theorem~\cite{Rogozin}. 

It would be natural to consider extending Rogozin's approach to obtain lower bounds on the convergence rates of Markov processes towards their invariant measures. While this might be achievable, as suggested in~\cite{Menshikov95}, employing renewal theory would  likely require much stronger assumptions on the process than in the present paper. 
Moreover, in this context, lower bounds on the rates of convergence can be established only if both upper and lower bounds on the model parameters are assumed to match. In contrast, as demonstrated by the examples in Section~\ref{sec:examples}, our approach yields asymptotically matching lower bounds on the rate of convergence even for the models with oscillating parameters. Adapting the techniques from~\cite{sarig02,Melbourne12} to the context of ergodic Markov processes studied in this paper would appear  to lead to weaker results.

\section{Applications to elliptic diffusions, L\'evy-driven SDEs and hypoelliptic stochastic damping Hamiltonian systems}
\label{sec:examples}
We now apply  the results of Section~\ref{sec:main_results} to stochastic models used across probability and its application. We mostly consider models where  Lyapunov drift conditions have been developed to obtain upper bounds on the rates of convergence. The examples are chosen to demonstrate the robustness and the complementary nature of our approach to the existing theory for upper bounds: we  cover the classes of models studied in~\cite[Sec.~5]{douc2009subgeometric} and~\cite[Sec.~3]{Fort2005}. In particular, we analyse models exhibiting polynomial, stretched exponential and exponential ergodicity and in all cases provide  lower bounds, which asymptotically match known upper bounds. More precisely, we give lower bounds in $f$-variation for  polynomially ergodic elliptic and hypoelliptic diffusions, studied in~\cite[Sec.~3]{Fort2005} and~\cite{douc2009subgeometric,Wu2001}, respectively. For brevity, in all other examples we consider total variation   (even though our methods could handle $f$-variation in these cases) and show that our lower bounds match asymptotically the upper bounds in~\cite{Fort2005,douc2009subgeometric}. 

The proofs of the results  in this section, contained in Section~\ref{sec:examples_proofs} below, typically involve verifying Assumption~\nameref{sub_drift_conditions}  via Theorem~\ref{thm:generator} and the condition in~\ref{assumption:f_convergence_a} of Theorem~\ref{thm:f_rate} via Lemma~\ref{lem:bounded_generator}. In all examples of this section, the dependence of the multiplicative constant 
on the initial position $x\in\cX$ can be obtained explicitly from our main estimate in~\eqref{eq:main_result_rate_of_convergence} of Theorem~\ref{thm:f_rate}. However, for ease of presentation, the explicit dependence on the starting point has been omitted.

\subsection{Elliptic diffusions}
\label{sec:diffusions}
In this section we apply our results to obtain lower bounds on the rate of convergence to stationarity of elliptic diffusions. We first introduce the general form of the  model and give a simple motivating example. 
In Sections~\ref{subsubsec:Polynomial} and~\ref{subsubsec_Subexponential}   we discuss the polynomial and stretched exponential cases, respectively. 
We stress that our assumptions allow  multiplicative noise with unbounded instantaneous variance of the process. 
Example~\ref{example:missmatch} of 
Section~\ref{subsubsec:Polynomial}
demonstrates the necessity of two-sided asymptotic assumptions on the coefficients (used in this section)  for obtaining the actual rate of convergence to stationarity. 
It also shows that upper bounds on the rate of convergence, established in~\cite{Fort2005, douc2009subgeometric, veretenn97}
using asymptotic upper bounds on the coefficients only,
may be orders of magnitude larger than the actual convergence rate.

Section~\ref{sec:diffusions} concludes with the class of elliptic diffusions exhibiting exponential ergodicity (see
Section~\ref{subsubsec:Exponential} below).
Even though our main focus  is on subexponential ergodicity, Section~\ref{subsubsec:Exponential} shows that our methods for lower bounds are also applicable to the exponentially ergodic case.

For  $n\in\N$,
let the process $X=(X_t)_{t\in\RP}$ with state space $\cX\coloneqq \R^n$
be the unique strong solution of the stochastic differential equation (SDE)
\begin{equation}
\label{eq:elliptic_SDE}
X_t = X_0 + \int_0^t b(X_s)\ud s + \int_0^t \sigma(X_s)\ud B_s,
\end{equation}
where the functions $b:\R^n\to\R^n$ and $\sigma:\R^n\to\R^{n\times n}$ are locally Lipschitz (i.e. for every $l>0$ there exists a finite constant $c_l$ such that $\lvert b(x)-b(y)\rvert + \lvert \sigma(x)-\sigma(y)\rvert \leq c_l\lvert x-y\rvert$ for all
$x,y\in\R^n$ with  $|x|,|y|<l$), 
 and $(B_t)_{t\in\RP}$  a standard $n$-dimensional Brownian motion. Here and throughout we denote by $|\cdot|$ and  $\langle \cdot,\cdot\rangle$ the Euclidean norm and standard scalar product on $\R^n$, respectively.
 In particular, $|x|^2=\langle x,x\rangle$ for all $x\in\R^n$. 
 Let $\Sigma\coloneqq \sigma\sigma^\intercal$, where $\sigma^\intercal$ is the transpose of the matrix $\sigma$, be uniformly elliptic: 
 $\langle \Sigma(x)y,y\rangle \geq \delta_a |y|^2$ for some $\delta_a>0$ and all $y,x\in\R^n$. 
 The diffusion $X$ in~\eqref{eq:elliptic_SDE} is assumed to be ergodic with an invariant measure $\pi$.

\noindent \underline{Motivating example}: Consider the following toy example of an elliptic SDE~\eqref{eq:elliptic_SDE} with $n = 1$ (i.e. $\cX=\R$), $\sigma \equiv 1$, and $b(x) = -x/|x|$ for all $|x|>1$.
Applying the generator $\cA$ of $X$ to the Lyapunov function $V(x) = 1+ x^2$ yields $\cA V(x) \leq -V(x)^{1/2}$ for all $x\in\R$ with $|x|$ large. Then, for every initial condition $x\in\R$, an application of the drift condition from~\cite[Thms~3.2 and~3.4]{douc2009subgeometric} (see~\eqref{eq:douc_drif_condition} above)   with  $V(x) = 1+ x^2$ (and $\phi(r)$ proportional to $\sqrt{r}$ as $r\to\infty$) implies  the following upper bound:
$
\| \P_x(X_t\in\cdot)-\pi(\cdot)\|_{\TV} \leq C_x/t$ for all $t\in[1,\infty)$ and some constant $C_x\in(0,\infty)$.
Given that the process $X$ is in fact  exponentially ergodic (use~\cite[Thms~3.2 and~3.4]{douc2009subgeometric} with  $V(x)=\exp(|x|)$ for large $|x|$), it is evident that the upper bounds obtained through Lyapunov drift condition~\eqref{eq:douc_drif_condition} may be exceedingly inaccurate. 
Furthermore, this  example demonstrates that the classical Lyapunov-function theory for upper bounds is not sufficient for characterising the processes that are subexponentially ergodic. 
We now apply the results in Section~\ref{sec:main_results} to address this problem by establishing the actual subexponential rate of convergence using our  Lyapunov-function drift conditions in~\nameref{sub_drift_conditions}.

\subsubsection{Polynomial tails}
\label{subsubsec:Polynomial}
The following assumption ensures polynomial upper bounds on the tails of the invariant measure and the rate of convergence in total variation~\cite[Sec~3.2]{Fort2005},~\cite{veretenn97}. 
Recall that a function $g:\R^n\to\R$ satisfies $g(x)=o(1)$ as $|x|\to\infty$ if $\lim_{|x|\to\infty}\sup_{u\in\R^n, |u|=1}g(u|x|)=0$.

\noindent
\textbf{\namedlabel{example:ass_diffp}{Assumption~A$_p$}.}
\textit{There exist  $\alpha,\beta,\gamma\in(0,\infty)$ and $\ell\in[0,2)$, such that $2-\ell<2+(2\alpha-\gamma)/\beta=:m_c$ and the coefficients 
$b$ and $\Sigma=\sigma\sigma^\intercal$ of~\eqref{eq:elliptic_SDE} satisfy  (as $|x|\to\infty$)
$$
\langle b(x),x/|x|\rangle/|x|^{\ell-1} = -\alpha+ o(1),\quad
\langle \Sigma(x)x/|x|,x/|x| \rangle/|x|^{\ell} = \beta  + o(1), \quad\Tr(\Sigma(x))/ |x|^{\ell} = \gamma  + o(1).
$$
}

As defined in~\cite{Fort2005}, Langevin tempered diffusion on $\R^n$,  
given by a smooth  $\tilde\pi:\R^n\to(0,\infty)$ (proportional to the density of $\pi$ on $\R^n$), satisfies SDE~\eqref{eq:elliptic_SDE}
with coefficients
$\Sigma(x)=I_n/\tilde\pi(x)^{2d}$, where  $I_n$ is the identity matrix and $d\geq 0$, and $b(x)=(1-2d)\nabla(\log \tilde\pi)(x)/(2\tilde\pi(x)^{2d})$, where $\nabla$ denotes the gradient. If $d=0$, we get the classical Langevin diffusion with bounded volatility. This is an important class of diffusions because, by construction, their invariant measure equals $\pi$. They satisfy \ref{example:ass_diffp} when 
$1/\tilde\pi(x)^{2d}$ is proportional to $|x|^{\ell}$, with $\ell\in[0,2)$, and $\langle \nabla \tilde\pi(x), x\rangle$ is proportional to $\tilde\pi(x)$ as $|x|\to\infty$.

 When $\pi$ has polynomial tails,
upper bounds for the convergence in total variation for Langevin tempered diffusions on $\R^n$ (as well as their generalisation in \ref{example:ass_diffp})  were studied  in~\cite[Sec~3.2]{Fort2005}. In this section we apply our methods under  \ref{example:ass_diffp} to obtain matching lower bounds for the results in~\cite{Fort2005},   establishing the  rate of convergence to stationarity.

\begin{thm}
\label{thm:diff_poly}
Let~\ref{example:ass_diffp} hold. Pick  $k\in[0,\ell + (2\alpha-\gamma)/\beta)$ and fix the critical exponent $\alpha_k\coloneqq m_c/(2-\ell)-1-k/(2-\ell) =(2(\alpha+\beta)-\gamma)/((2-\ell)\beta)-1-k/(2-\ell)$. Then for 
the function  $f_k(x) \coloneqq 1+|x|^k$,
a starting point $x\in\R^n$ and $\eps>0$, there exists a constant $c_{k,\eps}\in(0,\infty)$ such that 
$$
c_{k,\eps}/t^{\alpha_k+\eps}\leq \|\P_x(X_t\in \cdot)-\pi(\cdot)\|_{f_k}\quad\text{for all $t\in[1,\infty)$.}
$$
\end{thm}

The next result provides lower bounds on the tail of the invariant measure $\pi$ and the return time $\tau_D(\delta)=\inf\{t>\delta: X_t\in D\}$ of the diffusion $X$ to a bounded set $D\in\cB(\R^n)$ after time $\delta>0$.

\begin{thm}
\label{thm:eliptic_poly_invariant}
Let \ref{example:ass_diffp} hold and recall  $0<2-\ell<m_c=2+(2\alpha-\gamma)/\beta$.
\begin{myenumi}
\item[(a)] For every $\eps>0$, there exists $c_\pi\in(0,\infty)$ such that 
$$
c_\pi/ r^{m_c+\ell-2+\eps}\leq\pi(\{|x|\geq r\})\qquad\text{for all $r\in[1,\infty)$.}
$$
\item[(b)] For any $x\in\R^n$, any bounded set $D\in \cB(\R^n)$ and arbitrary $\eps,\delta\in(0,\infty)$, 
there exists a constant $c_\tau\in(0,\infty)$,  such that
$$
c_\tau /t^{m_c/(2-\ell)+\eps}\leq \P_x(\tau_D(\delta)\geq t) \qquad \text{for all $t\in[1,\infty)$.}
$$
\end{myenumi}
\end{thm}

The proofs of Theorems~\ref{thm:diff_poly} and~\ref{thm:eliptic_poly_invariant} use the $\mathbf{L}$-drift condition~\nameref{sub_drift_conditions}, where $V$, $\varphi$ and $\Psi$ exhibit polynomial growth. The functions $\varphi$ and $\Psi$ are obtained directly from the generator inequalities of Theorem~\ref{thm:generator} and are also polynomial, see Section~\ref{subsubsec:poly_proofs} below for details.

\begin{rem}[\textbf{matching rates}]
As mentioned above, our lower bounds in Theorem~\ref{thm:diff_poly} on the $f$-variation distance matches the upper bounds in~\cite{Fort2005}. Recall from Theorem~\ref{thm:diff_poly} the parameters $k$, $\alpha_k$ and the function $f_k$. Then, for every $\eps>0$, there exist constants $c_{k,\eps},C_{k,\eps}\in(0,\infty)$ such that
\begin{equation}
\label{eq:matching_poly}
c_{k,\eps}/t^{\alpha_k+\eps}\leq \|\P_x(X_t\in\cdot)-\pi(\cdot)\|_{f_k}\leq C_{k,\eps}/t^{\alpha_k-\eps}\quad\text{for all $t\in[1,\infty)$.}
\end{equation}
Analogous bounds could be established for the tails of the invariant measure $\pi$ and the return time $\tau_D(\delta)$. Put differently, in conjunction with~\cite{Fort2005}, Theorems~\ref{thm:diff_poly} and~\ref{thm:eliptic_poly_invariant} imply the  convergence rate to stationarity and the decay of the tails of the invariant measure and the return times in logarithmic scale:
\begin{equation*}
\begin{aligned}
\lim_{t\to\infty} \frac{\log \|\P_x(X_t\in\cdot)-\pi(\cdot))\|_{f_k}}{\log t} = \alpha_k, \quad &\quad \lim_{t\to\infty} \frac{\log \P_x(\tau_D(\delta)>t)}{\log t} = m_c/(2-\ell), \\  
\lim_{t\to\infty} \frac{\log \pi(\{|x|\geq r\})}{\log r}&= m_c+\ell-2.
\end{aligned}
\end{equation*}
\end{rem}

\begin{rem}[\textbf{oscillating coefficients}]
Note that~\ref{example:ass_diffp} requires matching upper and lower bounds on the asymptotic behaviour (as $|x|\to\infty$) of the coefficients of SDE~\eqref{eq:elliptic_SDE}. 
The following is a natural generalisation of~\ref{example:ass_diffp}, allowing oscillating coefficients:
for all $x\in\R^n$ outside of some compact set it holds that $\gamma_L\leq\Tr(\Sigma(x))/|x|^{\ell}\leq \gamma_U$, 
$$
-\alpha_L\leq\langle b(x),x/|x|\rangle/|x|^{\ell-1}\leq -\alpha_U \qquad\&\qquad \beta_L\leq \langle \Sigma(x)x/|x|,x/|x|\rangle/|x|^\ell\leq \beta_U
$$
for some constants $\gamma_L,\gamma_U,\alpha_L,\alpha_U,\beta_L,\beta_U\in(0,\infty)$.
Under this assumption, the lower bounds in Theorem~\ref{thm:diff_poly} and~\ref{thm:eliptic_poly_invariant} remain valid with with $m_c^L = 2+(2\alpha_L-\gamma_L)/\beta_L$ (instead of $m_c$). While a slight modification of the proofs of Theorems~\ref{thm:diff_poly} and~\ref{thm:eliptic_poly_invariant} is necessary, the same polynomial class of  Lyapunov functions can be used. The details are omitted for brevity.
The upper bounds depend on $m_c^U=2+(2\alpha_U-\gamma_U)/\beta_U$ and can be obtained from the classical theory~\cite{Fort2005}. The discrepancy between 
$m_c^L>m_c^U$ leads to an asymptotic gap between 
the lower and upper bounds on the rate of convergence.
As demonstrated by the following example, this gap is not an  artefact of our methods. 
Example~\ref{example:missmatch} below also shows that  upper bounds on the coefficients alone, assumed  in~\cite[Remark after Thm~16]{Fort2005},  are  insufficient to deduce the rate of convergence in the sense of~\eqref{eq:matching_poly}. 
\end{rem}

\begin{exam}
\label{example:missmatch}
Let the function $\pi:\R\to (0,\infty)$ be  a positive thrice continuously differentiable density on $\R$ (up to a normalizing constant). The one-dimensional SDE
\begin{equation}
\label{eq:one_dimension_Langevin}
\ud X_t = (\log \pi)'(X_t)\ud t + \sqrt{2}\ud B_t.
\end{equation}
is an example of an elliptic diffusion of the form~\eqref{eq:elliptic_SDE} in $\R$.

\ref{example:ass_diffp} 
 above is given in terms of the limits of the coefficients of SDE~\eqref{eq:elliptic_SDE} as $|x|\to\infty$. This differs from the assumptions in~\cite[Sec~3.2]{Fort2005}, 
 where only the asymptotic upper bounds on the model parameters are assumed. 
 In particular, in the context of  SDE~\eqref{eq:one_dimension_Langevin}, the results from~\cite[Sec~3.2]{Fort2005} imply the following:
if 
$\limsup_{|x|\to\infty}x(\log \pi)'(x)= -\alpha$ for some $\alpha\in(1,\infty)$, then for every $x\in\R$ and $\eps>0$ there exists a constant $C>0$, such that $\|\P_x(X_t\in\cdot)-\pi(\cdot)\|_{\TV}\leq C/t^{(\alpha-1)/2-\eps}$ holds for all $t\in[1,\infty)$.
However, as we shall now see,
an upper bound on the drift in~\eqref{eq:one_dimension_Langevin} is not sufficient to determine the actual rate of convergence to the invariant measure.

Fix $\alpha\in(1,\infty)$. Then for every $k\in(\alpha,\infty)$, there exists an invariant density $\pi$, such that the following statements hold: $\limsup_{|x|\to\infty}x(\log \pi)'(x)= -\alpha$ and, for every $x\in\R$  and
$\eps>0$,
there exist constants $c',C'\in(0,\infty)$ such that 
\begin{equation}
\label{eq:oscilating_langevin}
c'/t^{(k-1)/2} \leq \| \P_x(X_t\in\cdot)-\pi(\cdot)\|_{\TV}\leq C'/t^{(k-1)/2-\eps}\quad \text{for all $t\in[1,\infty)$,}
\end{equation}
where the process $X$ follows SDE~\eqref{eq:one_dimension_Langevin} with $X_0=x$. The proof of~\eqref{eq:oscilating_langevin}
is in Section~\ref{subsubsec:poly_proofs} below.

This example demonstrates that, assuming only upper bounds on the coefficients of the diffusion in~\eqref{eq:elliptic_SDE} (as in~\cite[Sec~3.2]{Fort2005}), may result in the upper bound on the rate being greater than the actual rate of convergence by any polynomial order.
In contrast, if we assume matching lower and upper bounds on the drift of~\eqref{eq:one_dimension_Langevin} (i.e. $(\log\pi)'(x) = -\alpha/x(1+o(1))$ as $|x|\to\infty$), the bounds in~\eqref{eq:matching_poly} imply~\eqref{eq:oscilating_langevin} with $k=\alpha$.
\end{exam}

\subsubsection{Stretched exponential tails}
\label{subsubsec_Subexponential}
Under the following assumption, the process $X$ exhibits stretched exponential ergodicity.

\noindent
\textbf{\namedlabel{example:ass_diffse}{Assumption~A$_{se}$}.}
\textit{There exist constants $p\in(0,1)$ and $\ell \in[0,2p)$, such that the coefficients 
$b$ and $\Sigma=\sigma\sigma^\intercal$ in~\eqref{eq:elliptic_SDE} satisfy  the following asymptotic assumptions: $\limsup_{|x|\to\infty}\Tr(\Sigma(x))/|x|^{\ell}<\infty$ and
$$
-\alpha_L\leq\langle b(x),x/|x|\rangle/|x|^{\ell-p}\leq -\alpha_U \qquad\&\qquad \beta_L\leq \langle \Sigma(x)x/|x|,x/|x|\rangle/|x|^\ell\leq \beta_U,
$$
for all $x$ outside of some compact set and some constants $\alpha_L,\alpha_U,\beta_L,\beta_U\in(0,\infty).$}

\ref{example:ass_diffse} covers tempered Langevin diffusions exhibiting stretched exponential ergodicity, studied in~\cite[Sec~5.2]{douc2009subgeometric}. In particular,  parameters $\beta\in (0,1)$ and $d\in[0,1/\beta-1)$ (in the notation of\cite[Sec~5.2]{douc2009subgeometric}), are included in the cases $p=\beta-1$ and $\ell = 2\beta d$. Moreover, the case $\ell =0$ includes Langevin diffusions.

 The main result, Theorem~\ref{thm:example_se_inv_rate} below, establishes lower bounds on the tail of the invariant measure $\pi$, the rate of convergence towards the invariant measure in the $\TV$-distance and  the tail of the return time $\tau_D(\delta)=\inf\{t>\delta: X_t\in D\}$  of $X$ to a bounded set $D\in\cB(\R^n)$ after time $\delta>0$. 

\begin{thm}
\label{thm:example_se_inv_rate}
Let \ref{example:ass_diffse} hold. 
\begin{myenumi}
\item[(a)]  There exist $c_\pi,u_\pi\in(0,\infty)$ such that 
$$
c_\pi/\exp(u_\pi r^{1-p}))\leq\pi(\{|x|\geq r\}) \quad\text{for all $r\in[1,\infty)$.}
$$
\item[(b)]  For every $x\in\R^n$, $\delta>0$ and a bounded set $D\in\cB(\R^n)$, there exist $c_\tau,u_\tau\in(0,\infty)$  such that
$$
 c_\tau/\exp(u_\tau t^{(1-p)/(1+p-\ell)}) \leq \P_x(\tau_D(\delta)\geq t)\quad\text{for all $t\in[1,\infty)$}.
$$
\item[(c)] For every $x\in\R^n$ there exist $c_\TV,u_\TV\in(0,\infty)$ such that
$$
c_\TV/\exp(u_\TV t^{(1-p)/(1+p-\ell)})\leq \| \P_x(X_t\in\cdot)-\pi(\cdot)\|_{\TV} \quad\text{for all $t\in[1,\infty)$.}
$$
\end{myenumi}
\end{thm}

In the proof of Theorem~\ref{thm:example_se_inv_rate} we work with the $\mathbf{L}$-drift condition~
\nameref{sub_drift_conditions}, where the Lyapunov function $V$ grows polynomially, while $\Psi$ grows at the stretched exponential rate with the ``stretching'' parameter $1-p$. The function $\varphi$ is again polynomial and is obtained directly from the 
generator inequalities in Theorem~\ref{thm:generator}. Like $\Psi$, the function $h$ in Assumption~\ref{assumption:f_convergence_a} of Theorem~\ref{thm:f_rate} grows at the stretched exponential rate with parameter $1-p$,  see 
Section~\ref{subsubsec:subexp_proofs} below 
for the details of the proof of Theorem~\ref{thm:example_se_inv_rate}.

\begin{rem}[\textbf{matching rates}]
\label{cor:rate_diff_se}
The class of models defined by~\ref{example:ass_diffse} contains elliptic and tempered Langevin diffusions studied in the context of upper bounds in~\cite[Sections~5.1 and~5.2]{douc2009subgeometric}. The techniques of~\cite{douc2009subgeometric} can be applied to obtain upper bounds for all models covered by~\ref{example:ass_diffse}. Our lower bounds in Theorem~\ref{thm:example_se_inv_rate} match these upper bounds. More precisely, there exist constants   $c_\TV,C_\TV,u_\TV,U_\TV\in(0,\infty)$ 
such that 
\begin{equation}
\label{eq:tv_exponents_se}
c_\TV/\exp(u_\TV t^{(1-p)/(1+p-\ell)})\leq \| \P_x(X_t\in\cdot)-\pi(\cdot)\|_{\TV} \leq C_\TV/\exp(U_\TV t^{(1-p)/(1+p-\ell)})
\end{equation}
holds for all $t\in[1,\infty)$.
Analogous bounds hold for the tails of the invariant measure $\pi$ and the return time $\tau_D(\delta)$.
\end{rem}

\begin{rem}[\textbf{matching constants}]
\label{rem:example_se_inv_rate}
\ref{example:ass_diffse} cannot ensure that constants $u_\pi$, $u_\tau$, $u_\TV$ in the exponents in Theorem~\ref{thm:example_se_inv_rate} are only $\eps$-away from  optimal constants in upper bounds. This is because, under \ref{example:ass_diffse}, the coefficients of the SDE in~\eqref{eq:elliptic_SDE} may oscillate asymptotically as $|x|\to\infty$. However, assuming that 
$$
\langle b(x),x/|x|\rangle/|x|^{-p} = -\alpha+ o(|x|^{p-1}), \quad
\langle \Sigma(x)x/|x|,x/|x| \rangle = \beta + o(|x|^{p-1})\quad\&\quad \Tr(\Sigma(x)) = \gamma+ o(1),
$$
as $|x|\to\infty$,
for some constants $p\in(0,1)$, $\alpha, \beta,\gamma\in(0,\infty)$, it is possible to prove
$$
\lim_{r\to\infty} r^{(1-p)}\log \pi(\{|x|\geq r\}) = 2\alpha/(\beta(1-p)) = \alpha^{-(1-p)/(1+p)}\lim_{t\to\infty} t^{(1-p)/(1+p-\ell)}\log \P_x(\tau_D(\delta)\geq t).
$$
 Limit inferiors in these limits follow from the theory developed in this paper. However, they require a  more involved Lyapunov function (i.e.\hspace{-1mm} a product of a polynomial and stretched exponential functions) than the one in the proof of Theorem~\ref{thm:example_se_inv_rate}. 
 Upper bounds in these limits follow from~\cite[Thm~5.3]{douc2009subgeometric}.
 Moreover, for the constants  in the exponents for the total variation distance in~\eqref{eq:tv_exponents_se}, we can show $U_\TV/u_\TV<2+\eps$. 
We omit the details for brevity.
\end{rem}

\subsubsection{Exponential tails}
\label{subsubsec:Exponential}
In this section we demonstrate that our methods may also be used to derive lower bounds on the rate of convergence of certain exponentially ergodic processes.  

\noindent
\textbf{\namedlabel{example:ass_diffexp}{Assumption~A$_{e}$}.}
\noindent \textit{The coefficients 
$b$ and $\Sigma=\sigma\sigma^\intercal$ in~\eqref{eq:elliptic_SDE} satisfy  the following asymptotic assumptions: $\limsup_{|x|\to\infty}\Tr(\Sigma(x))<\infty$  and 
$$
-\alpha_L\leq\langle b(x),x/|x|\rangle\leq -\alpha_U \qquad \& \qquad \beta_L\leq \langle \Sigma(x)x/|x|,x/|x|\rangle\leq \beta_U
$$
for all $x$ outside of a compact set and some constants $\alpha_L,\alpha_U,\beta_L,\beta_U\in(0,\infty).$
}

\begin{thm}
\label{thm:example_exp}
Let~\ref{example:ass_diffexp} hold.
\begin{myenumi}
\item[(a)] There exist $c_{\pi},u_\pi\in(0,\infty)$ such that 
$$
c_{\pi}/\exp(u_\pi r)\leq \pi(\{|x|\geq r\})\quad\text{for all $r\in[1,\infty)$.}
$$
\item[(b)] For every $x\in\R^n$, $\delta\in(0,\infty)$ and a bounded set $D\in\cB(\R^n)$ there exist $c_\tau,u_\tau\in(0,\infty)$ such that 
$$
c_\tau/\exp(u_\tau t)\leq \P_x(\tau_D(\delta) \geq t)\quad\text{for all $t\in[1,\infty)$.}
$$
\item[(c)] For every $x\in\R^n$ there exist $c_\TV,u_\TV\in(0,\infty)$ such that
$$
c_\TV/\exp(u_\TV t)\leq \|\P_x(X_t\in\cdot)-\pi(\cdot)\|_{\TV}\quad\text{for all $t\in[1,\infty)$.}
$$
\end{myenumi}
\end{thm}

 The $\mathbf{L}$-drift condition \nameref{sub_drift_conditions}
in the proof of Theorem~\ref{thm:example_exp}
uses
 $V$ and $\Psi$, which exhibit polynomial and  exponential growth, respectively. As in the examples of Sections~\ref{subsubsec:Polynomial} and~\ref{subsubsec_Subexponential}, the function $\varphi$ is polynomial and follows directly from the generator inequalities in Theorem~\ref{thm:generator}.
The details of the proof of Theorem~\ref{thm:example_exp}  are  in 
Section~\ref{subsubsec:exponential_proofs} below.

\begin{rem}[\textbf{matching rates}]
\label{rem:example_exponential}
Exponential upper bounds on the results in Theorem~\ref{thm:example_exp} can be obtained, for example, by applying results from~\cite[Sec~3]{douc2009subgeometric} and estimates in Proposition~\ref{prop:deterministic_exp} below. In particular, there exist constants $C_\TV,U_\TV\in(0,\infty)$ such that 
$$
c_\TV/\exp(u_\TV t)\leq \|\P_x(X_t\in\cdot)-\pi(\cdot)\|_{\TV}\leq C_\TV/\exp(U_{\TV}t)\quad\text{for all $t\in[1,\infty)$.}
$$
Similar bounds hold for the tails of the invariant measure $\pi$ and the return time $\tau_D(\delta)$. Moreover,
ensuring that the constants $u_\pi,u_\tau,u_\TV$ in the exponents in Theorem~\ref{thm:example_exp} are  close to their optimal constants in the upper bounds would require stronger assumptions and a more involved choice of Lyapunov functions. As the focus of the present paper is on subexponential ergodicity,  the details are omitted.
\end{rem}

\subsection{L\'evy-driven SDE}
\label{sec:levy}

In this section we apply the results of Section~\ref{sec:main_results} to a solution $X$ of a L\'evy-driven SDE taking values in $\cX = \R$. 
More precisely, let $X$ follow the SDE 
\begin{equation}
\label{eq:levy_OU}
\ud X_t = -\mu X_t\ud t + \sigma(X_{t-})\ud L_t,
\end{equation}
where $\mu\in(0,\infty)$, $\sigma:\R\to\RP$ is Lipschitz  with $0<\inf_{x\in\R}\sigma(x)\leq\sup_{x\in\R}\sigma(x)<\infty$ and $L$ is a pure-jump L\'evy process with L\'evy measure $\nu$ (cf.~\cite[Ch.~1]{MR3185174}). 
Such $X$ exists and is unique for any starting point $X_0=x\in\R$~\cite[Thm~V.6]{MR2273672}.
The drift and the Gaussian component of $L$ are zero as they do not influence the asymptotic behaviour of $X$ when  $\nu$ has heavy tails and $\sigma$ is bounded. Adding them would require only a minor modification of the proof of Theorem~\ref{thm:levy_OU} below. 

\noindent \textbf{\namedlabel{example:ass_Levy}{Assumption~A$_{\text{L}}$}.}
\textit{Assume that for some $m_c\in(1,\infty)$, 
the L\'evy measure $\nu$ satisfies
\begin{align*}
0 &<\liminf_{r\to\infty}(\log r)^{m_c}\nu([r,\infty)) \leq
\limsup_{r\to\infty}(\log r)^{m_c}\nu([r,\infty)) <\infty\quad\text{and}\\
0 &= \limsup_{r\to-\infty}\nu((-\infty,r])(\log |r|)^{m_c+1}.
\end{align*}
}

As we shall see, when the jumps of $L$ have very heavy tails (as in \ref{example:ass_Levy}), $X$ does not exhibit exponential ergodicity.  
Recall that  $\tau_D(\delta)=\inf\{t>\delta: X_t\in D\}$ is the return time
of the process $X$ to a bounded set $D\in\cB(\R^n)$ after time $\delta>0$.

\begin{thm}
\label{thm:levy_OU}
Let \ref{example:ass_Levy} hold and the process $X$ satisfy SDE~\eqref{eq:levy_OU} above.
\begin{myenumi}
\item[(a)] For any $\eps>0$ there exists a constant $c_{\pi}\in(0,\infty)$ such that $$
c_\pi/ (\log r)^{m_c-1+\eps} \leq \pi([r,\infty))  \quad \text{for all $r\in [1,\infty)$.}
$$
\item[(b)]  For every $x\in\R$, $\eps,\delta>0$ and bounded set $D\in\cB(\R)$, there exists $c_\tau\in(0,\infty)$ such that
$$
 c_\tau/t^{m_c+\eps} \leq \P_x(\tau_D(\delta)\geq t)\quad\text{for all $t\in[1,\infty)$}.
$$
\item[(c)] For any $x\in\R$ and $\eps>0$ there exists a constant $c_\TV\in(0,\infty)$ such  that $$
c_\TV/t^{m_c-1+\eps}\leq\|\P_x(X_t\in\cdot)-\pi\|_{\TV} \quad\text{for all $t\in[1,\infty)$.}
$$
\end{myenumi}
\end{thm}

The $\mathbf{L}$-drift condition \nameref{sub_drift_conditions}
in the proof of Theorem~\ref{thm:levy_OU}
uses the Lyapunov function 
 $V$ with logarithmic growth. This is necessary to ensure the integrability of the marginals of the process $V(X)$.
The function $\Psi$ has polynomial growth as does $\varphi$ (the latter is again obtained from the generator inequalities in Theorem~\ref{thm:generator}). The function $h$ in Assumption~\ref{assumption:f_convergence_a} of Theorem~\ref{thm:f_rate}
is taken to be  polynomial with the ``largest'' growth rate, such that  the marginals of the process $h\circ V(X)$ remain integrable.
The proof of Theorem~\ref{thm:levy_OU} 
is in
Section~\ref{subsec:levy_proofs} below.

\begin{rem}[\textbf{matching rates}]
\label{rem:levy_rate}
Theorem~\ref{thm:levy_OU} is applicable to general pure-jump L\'evy drivers with two-sided jumps and arbitrary path variation. In particular, it 
covers the simpler model studied in~\cite{Fort2005} with additive noise, where it is assumed that $\sigma \equiv 1$ and $\nu((-\infty,1])=0$ (making 
$L$ a compound Poisson process with positive jumps 
and, consequently, $\cX = \RP$ for $X_0=x\in\RP$). 
Our lower bounds match the upper bounds from~\cite{Fort2005}, i.e.
for some  $c_\TV,C_\TV\in(0,\infty)$ we have
$$
c_\TV/t^{m_c-1+\eps} \leq\|\P_x(X_t\in\cdot)-\pi\|_{\TV} \leq C_\TV /t^{m_c-1-\eps}\quad\text{for all $t\in[1,\infty)$.}
$$
Analogous bound holds for the tail of the invariant measure $\pi$ and the return time $\tau_D(\delta)$. 
\end{rem}

\begin{rem}
\label{rem:levy_inbalance}
Theorem~\ref{thm:levy_OU} and Remark~\ref{rem:levy_rate} show that  the process $X$ in~\eqref{eq:levy_OU} converges  to stationarity at a polynomial rate even though its invariant measure has a logarithmically heavy positive tail. Since \ref{example:ass_Levy}
stipulates that the negative tail of the L\'evy measure $\nu$ is orders of magnitude thinner than its positive tail, we prove Theorem~\ref{thm:levy_OU} in Section~\ref{subsec:levy_proofs} below using a Lyapunov function~\eqref{eq:log_function}, which is bounded on the negative half-line. 

The asymmetry in \ref{example:ass_Levy}  is also visible in the invariant measure $\pi$. If, for example, $\nu((-\infty,r])$ decays polynomially, then the 
form of the extended generator of $X$ in~\eqref{eq:generator_Levy} below, applied to an appropriate polynomial  Lyapunov function $V$, tends to $+\infty$ (as $r\to-\infty$) at the same rate as $V$ itself. This follows by a similar argument to the one in the proof of Proposition~\ref{prop:OU_deterministic}\ref{prop:deterministic_levy_a} below. In this case, by~\cite[Prop.~3.1]{douc2009subgeometric} and the Markov inequality, we obtain a polynomial  upper bound on the tail $\pi((-\infty,r])$ as $r\to-\infty$. Moreover, by the same argument, if $\nu((-\infty,r])$ decays exponentially as $r\to-\infty$, so does $\pi((-\infty,r])$.
\end{rem}

\subsection{Stochastic damping Hamiltonian system}
\label{subsec:hamiltonian}

Consider the  hypoelliptic diffusion $X=(Z,Y)$, where   $Z_t$ (resp. $Y_t$) is the position (resp. velocity) at time $t$ of a physical system moving in $\R^n$, satisfying the following SDE:
$\ud Z_t = Y_t\ud t$ and $\ud Y_t = \sigma(Z_t,Y_t)\ud B_t -(c(Z_t,Y_t)Y_t+\nabla U(Z_t))\ud t$. Here the gradient $\nabla U(Z_t)$ is a friction force, $c(Z_t,Y_t)Y_t$ a damping force and $\sigma(Z_t,Y_t)\ud B_t$ a random force, with $B$ being the standard $n$-dimensional Brownian motion, acting 
on the system at $(Z_t,Y_t)$.
Exponential convergence to the invariant measure of this class of models has been studied extensively, see e.g.~\cite{Wu2001,Mattingly02}.

Our primary focus is on deriving lower bounds in cases where subexponential upper bounds on the rate of convergence have been established. We thus consider a one-dimensional example previously studied in the context of subexponential upper bounds~\cite{douc2009subgeometric}:
\begin{equation}
\label{eq:Hamiltonian_damping_onedim}
\ud Z_t = Y_t\ud t, \qquad \ud Y_t = \sigma\ud B_t-(cY_t+U'(Z_t))\ud t,
\end{equation}
where $U:\R\to\R$ is in $C^2(\R)$ and $\sigma,c\in(0,\infty)$ are positive constants.
We are interested in the case when the invariant measure  of the process $X=(Z,Y)$
has polynomial tails.

\noindent \textbf{\namedlabel{example:ass_Hamilton}{Assumption~A$_{\text{H}}$}.} 
\textit{Let the function $U\in C^2(\R)$ be such that   $zU'(z) = a + o(1)$ as $|z| \to \infty$
for some $a\in(0,\infty)$.
Assume also that the constants $\sigma, c\in(0,\infty) $  in~\eqref{eq:Hamiltonian_damping_onedim} satisfy $ac/\sigma^2>1/2$.}

The next result provides matching polynomial lower and upper bounds on the rate of convergence to stationarity in  $f$-variation  of the hypoelliptic diffusion in~\eqref{eq:Hamiltonian_damping_onedim} (proof is in Section~\ref{subsec:hamiltonian_proofs} below). The parameter $m=0$ in Theorem~\ref{thm:hamiltonian} corresponds to total variation. 

\begin{thm}
\label{thm:hamiltonian}
Let \ref{example:ass_Hamilton} hold and pick $m\in[0,2ac/\sigma^2-1)$. Then for 
the function $f_m(z,y) \coloneqq 1+|z|^m$,
any $x=(z,y)\in\R^n$ and $\eps>0$, there exist  $c_{m,\eps},C_{m,\eps}\in(0,\infty)$ such that
$$
c_{m,\eps}/t^{ac/\sigma^2-1/2-m/2+\eps}\leq \|\P_{x}(X_t\in\cdot)-\pi(\cdot)\|_{f_m}\leq C_{m,\eps}/t^{ac/\sigma^2-1/2-m/2-\eps} \quad \text{for all $t\in[1,\infty)$.}
$$
\end{thm}

The upper bound on the rate of convergence in Theorem~\ref{thm:hamiltonian} is obtained by applying the drift condition in~\cite{douc2009subgeometric} (see~\eqref{eq:douc_drif_condition} above) to an appropriate Lyapunov function (inspired by~\cite{Wu2001}). 
Note that the model in Theorem~\ref{thm:hamiltonian} has not been analysed in~\cite{douc2009subgeometric}, where a hypoelliptic diffusion in~\eqref{eq:Hamiltonian_damping_onedim} with stretched exponential tails is considered.
Following~\cite{Hairer09}, the matching lower bound in Theorem~\ref{thm:hamiltonian} is  obtained by comparing the tails of $X_t$ and $\pi$ via  Lemma~\ref{lem:lower_bound_f_convergence_rate}. This lower bound does not require a verification of the  $\mathbf{L}$-drift condition~\nameref{sub_drift_conditions}, because 
the invariant measure $\pi$ of the process in~\eqref{eq:Hamiltonian_damping_onedim} has a known density proportional to $(z,y)\mapsto \exp(-2c/\sigma^2(y^2/2 + U(z)))$.

The Lyapunov function $V_u$ used to obtain the upper bound in Theorem~\ref{thm:hamiltonian}  also yields polynomial upper bounds on the tail of return times. By analogy with all other models discussed in Section~\ref{sec:examples}, it would be natural to use the same $V_u$ to establish the $\mathbf{L}$-drift condition~\nameref{sub_drift_conditions}. However, a function  $\varphi$, which makes 
 the process $1/V_u(X)$ into a supermartingale, cannot satisfy the growth conditions  in~\nameref{sub_drift_conditions}\ref{sub_drift_conditions(i)}. 
Thus, unlike in the other models of this section, a different Lyapunov function  $V_l$ for lower bounds is needed. Such a $V_l$ exists but it only yields exponential lower bounds on the tails of return times. The reason for this discrepancy is that  $V_u$ necessarily mainly depends on the heavy-tailed component $Z$, while $V_l$ has to mostly depend on the light-tailed component $Y$.

\section{Return times to bounded sets for semimartingales}
\label{sec:return_times}

This section develops a general theory for the analysis of return times of continuous-time semimartingales to bounded sets. The main result in this section, Lemma~\ref{lem:return_times}, is a far reaching  generalisation of the approach dealing with return times initiated in~\cite[Sec~3]{brevsar2023brownian}
to stochastic processes  with jumps and/or unbounded variance.
Throughout this section we fix a probability space $(\Omega,\cF,\P)$ with 
 a right-continuous filtration $(\cF_t)_{t \in \RP}$.
We begin with an elementary maximal inequality. 

\begin{prop}[Maximal inequality]
\label{prop:maximal}
Let $(\cF_t)_{t\in\RP}$ be a right-continuous filtration and $\xi = (\xi_t)_{t\in\RP}$ an $(\cF_t)$-adapted process with  c\`adl\`ag paths taking values in $[0,1]$. 
Define an $(\cF_t)$-stopping time
$\tau_r\coloneqq\inf\{t\in\RP:\xi_t> r\}$ 
(recall $\inf \emptyset =\infty$)  
and assume that, for some $r>0$
and a locally bounded measurable function 
$f:\RP\times[0,1] \to \RP$, the process 
$(\xi_{t\wedge \tau_r} - \int_0^{t\wedge \tau_r} f(u,\xi_u)\ud u)_{t\in\RP}$
is an $(\cF_t)$-supermartingale. 
Then, for any $s\in(0,\infty)$, we have $$\P\left(\sup_{0\leq u < s}\xi_u> r\Big\vert\cF_0\right)\leq r^{-1}\left(\xi_0 + \E\left[\int_0^{s\wedge \tau_r} f(u,\xi_u)\ud u\Big\vert\cF_0\right]\right)\quad\text{a.s.}$$
\end{prop}

\begin{proof}
Pick any $s\in(0,\infty)$ and consider an  $(\cF_t)$-stopping time $\tau_r \wedge s$, bounded above by $s$. Note that $\sup_{0\leq u<\tau_r\wedge s} \xi_u\leq r$. Since $f$ is bounded on the compact set $[0,s]\times [0,r]$, there exists a constant $C\in(0,\infty)$ such that $\sup_{0\leq u < \tau_r\wedge s}f(u,\xi_u)\leq C$ a.s. Thus we obtain
\begin{equation}
\label{eq:maximal_ineq_upper_bound}
0\leq \E\left[\int_0^{\tau_r \wedge s} f(u,\xi_u)\ud u \Big\vert \cF_0\right]\leq Cs<\infty\quad\text{a.s.}
\end{equation}

Since $(\xi_{t\wedge \tau_r} - \int_0^{t\wedge \tau_r} f(u,\xi_u)\ud u)_{t\in\RP}$
is an $(\cF_t)$-supermartingale, $\xi_{\tau_r \wedge s} -\int_0^{\tau_r \wedge s} f(u,\xi_u)\ud u$ is integrable and
$\E[\xi_{\tau_r \wedge s} -\int_0^{\tau_r \wedge s} f(u,\xi_u)\ud u\vert \cF_0] \leq \xi_0$. The inequality in~\eqref{eq:maximal_ineq_upper_bound}
 and the fact that $\xi$ is non-negative imply the following:
\begin{align}
\nonumber
\E[\xi_{\tau_r \wedge s}\vert \cF_0]&=
\E\left[\xi_{\tau_r \wedge s} -\int_0^{\tau_r \wedge s} f(u,\xi_u)\ud u\Big\vert \cF_0\right]+\E\left[\int_0^{\tau_r \wedge s} f(u,\xi_u)\ud u \Big\vert \cF_0\right]
\\
&\leq \xi_0 + \E\left[\int_0^{\tau_r \wedge s} f(u,\xi_u)\ud u \Big\vert \cF_0\right].
\label{eq:maximal_supermart_inequality}
\end{align}
Moreover, by the definition of $\tau_r$ in the proposition we have $\{\sup_{u\in[0,s)}\xi_u> r\} = \{\tau_r < s\}$
a.s. Since $\xi$ is c\`adl\`ag, on the event $\{\tau_r < s\}$ we have $\xi_{\tau_r\wedge s} = \xi_{\tau_r} \geq r$ a.s. Thus, by~\eqref{eq:maximal_supermart_inequality},
we have
\begin{align*}
\P\left(\sup_{0\leq u < s}\xi_u > r\Big\vert\cF_0\right) &=  \P\left(\tau_r < s\Big\vert\cF_0\right)
\leq r^{-1}\E[\xi_{\tau_r \wedge s}\mathbbm1\{\tau_r <s\}\vert \cF_0]\leq r^{-1}\E[\xi_{\tau_r \wedge s}\vert \cF_0]
\\&\leq r^{-1}\left(\xi_0 + \E\left[\int_0^{\tau_r \wedge s} f(u,\xi_u)\ud u\Big\vert \cF_0\right]\right),
\end{align*}
implying the proposition.
\end{proof}

Proposition~\ref{prop:maximal} will be applied in the proof of Lemma~\ref{lem:return_times} with a continuous function $f$. To state the lemma, consider an $(\cF_t)$-adapted  process $\kappa \coloneqq (\kappa_t)_{t \in \RP}$ with c\`adl\`ag paths, taking values in $[1,\infty)$.
Let $\cT$ denote the set of all $[0,\infty]$-valued stopping times with respect to $(\cF_t)_{t\in \RP}$.
For any $\ell,r \in \RP$ and stopping time $T \in \cT$, define the first entry times (after $T$) by
\begin{align}
\label{eq::lambda}
    \lambda_{\ell,T} \coloneqq  T + \inf\{s \in \RP: T < \infty,~ \kappa_{T+s} < \ell\}, \\
    \rho_{r,T} \coloneqq  T + \inf\{s \in \RP: T < \infty, ~\kappa_{T+s} > r\},
    \label{eq::rho}
\end{align}
where $\inf\emptyset =\infty$. If $T = 0$, we write $\lambda_\ell \coloneqq \lambda_{\ell,0}$ and $\rho_r \coloneqq \rho_{r,0}$.


\begin{lem}
\label{lem:return_times}
Let $\kappa = (\kappa_t)_{t\in\RP}$ be a $[1,\infty)$-valued $(\cF_t)$-adapted process with c\`adl\`ag paths, satisfying $\limsup_{t\to\infty}\kappa_t =\infty$ a.s.
Suppose that there exist a level $\ell\in(1,\infty)$ and a non-decreasing continuous function $\varphi:(0,1]\to\RP$, such that 
the process 
$$
\left(1/\kappa_{(\rho_{r_q}+t)\wedge \lambda_{r,\rho_{r_q}}} - \int_{\rho_{r_q}}^{(\rho_{r_q} + t) \wedge \lambda_{r,\rho_{r_q}}} \varphi(1/\kappa_{u})\ud u\right)_{t\in\RP}\quad\text{is an $(\cF_{\rho_{r_q}+t})$-supermartingale}
$$
for every $q\in(0,1)$ and $r\in(\ell,\infty)$, where $r_q \coloneqq 2r/(1-q)$. 
Pick any $q\in(0,1)$,
$\eps\in(0,(1-q)/2]$ and
a non-decreasing function $f:\RP \to \RP$.
Then for every $r\in(\ell,\infty)$ we have 
\begin{equation}
\label{eq:return_lower_bound}
\P\left(\int_0^{\lambda_{\ell}} f(\kappa_s)\ud s \geq f(r)\eps/(r\varphi(1/r))\Big\vert \cF_0\right)\geq q\P(\rho_{r_q}<\lambda_{\ell}\vert \cF_0)\quad\text{a.s.}
\end{equation}
\end{lem}

\begin{rem}
\noindent (I) The  assumption $\limsup_{t\to\infty}\kappa_t =\infty$ a.s.
in Lemma~\ref{lem:return_times}
implies $\P(\rho_r<\infty)=1$ for all $r\in[1,\infty)$, making $\kappa_{\rho_{r_q}}$ well defined. The main step in the proof of inequality~\eqref{eq:return_lower_bound} in Lemma~\ref{lem:return_times} consists of establishing the following: with probability at least $q$, after reaching the level $r_q$, 
the process $\kappa$ spends more than $\eps/(r\varphi(1/r))$ units of time before returning below the level $r$. 

\noindent (II) Note that  $f\equiv1$ in Lemma~\ref{lem:return_times} yields a lower bound on the tail probability  
$\P(
\lambda_{\ell} \geq t\vert\cF_0)$.
In applications of  Lemma~\ref{lem:return_times} it is crucial that $q$ can be taken arbitrarily close to 1.
This allows us to conclude that for any fixed time $t$ and  starting point $\kappa_0$ larger than the level $2G_f(t)/(1-q)$,
the process does not leave $(\ell,\infty)$  before time $t$ with probability $q$. In particular, this will imply that petite sets of a Markov process satisfying~\nameref{sub_drift_conditions} are necessarily bounded (Lemma~\ref{lem:bounded_petite_sets} below).

\noindent (III) The null set where the inequality in~\eqref{eq:return_lower_bound}  fails to hold may vary with $r$. However, when applying Lemma~\ref{lem:return_times} in this paper, we only require the case where $\cF_0$ is a trivial $\sigma$-algebra, making~\eqref{eq:return_lower_bound} hold for all $r\in(\ell,\infty)$ simultaneously.
\end{rem}

\begin{proof}[Proof of Lemma~\ref{lem:return_times}]
Pick $q\in (0,1)$ and $r\in(\ell,\infty)$.
Note that the inequality in~\eqref{eq:return_lower_bound} holds for all $\eps\in(0,(1-q)/2]$ if it holds for 
$\eps=(1-q)/2$ (the right-hand side of~\eqref{eq:return_lower_bound} does not depend on $\eps$, while the probability on the left-hand side is decreasing in $\eps$).
We may thus fix $\eps=(1-q)/2$. 

We start by 
showing that, once the process $\kappa$ reaches the level $r_q=2r/(1-q)$, with probability at least $q$ it takes 
$\eps /(r\varphi(1/r))$ units of time for $\kappa$ to return to the interval $[1,r)$. More precisely, we now
establish the following inequality: 
\begin{equation}
\label{eq:r_squared_time_to_come_back}
\P(\lambda_{r,\rho_{r_q}} \geq \rho_{r_q } + \eps /(r\varphi(1/r))\vert \cF_{\rho_{r_q}}) \geq q\quad\text{a.s.}
\end{equation}

By the non-confinement assumption $\limsup_{t\to\infty}\kappa_t =\infty$, we have $\rho_{r_q}<\infty$ a.s. Define the c\`adl\`ag process $(\xi_t)_{t\in\RP}$ by  $\xi_t \coloneqq 1/\kappa_{\rho_{r_q} + t}$. Note that
$\tau_{1/r}=\inf\{t>0:\xi_t > 1/r\}=\lambda_{r,\rho_{r_q}}-\rho_{r_q}$
by~\eqref{eq::lambda}
and
hence
$t\wedge\tau_{1/r}=((\rho_{r_q}+t)\wedge \lambda_{r,\rho_{r_q}})-\rho_{r_q}$.
Moreover, by the definition of $\tau_{1/r}$, we have 
$\{\sup_{u\in[0,s)}\xi_u> 1/r\} = \{\tau_{1/r} < s\}$
a.s. for any $s\in\RP$.
By assumption, the process $(\xi_{t\wedge\tau_{1/r}} - \int_0^{t\wedge\tau_{1/r}} \varphi(\xi_u)\ud u)_{t\in\RP}$ is an $(\cF_{\rho_{r_q}+t})$-supermartingale.  
Moreover, since $\varphi:(0,1]\to\RP$ is non-decreasing and continuous,
it has a unique extension (via its right-limit at $0$) to a continuous function $\varphi:[0,1]\to\RP$.
Applying Proposition~\ref{prop:maximal} (with a continuous function $f(u,s)=\varphi(s)$) to $\xi$  and the stopping time $\tau_{1/r}$ yields
\begin{align*}
    \P(\lambda_{r,\rho_{r_q }} < \rho_{r_q }+t\vert \cF_{\rho_{r_q }}) &= \P(\tau_{1/r}< t\vert \cF_{\rho_{r_q}} )=\P(\sup_{0\leq u < t} \xi_u > 1/r\vert \cF_{\rho_{r_q}}) \\
    &\leq r\left( \xi_0 + \E\left[\int_0^{t\wedge\tau_{1/r}} \varphi(\xi_u)\ud u\Big\vert \cF_{\rho_{r_q}}\right] \right) \\
    &= r\left(1/\kappa_{\rho_{r_q}} +\E\left[ \int_{\rho_{r_q}}^{(\rho_{r_q}+t)\wedge \lambda_{r,\rho_{r_q}}}\varphi(1/\kappa_u)\ud u\Big\vert \cF_{\rho_{r_q }}\right]\right) \\
    &\leq r\left(1/r_q +\varphi(1/r)\E\left[(\rho_{r_q}+t)\wedge \lambda_{r,\rho_{r_q}}- \rho_{r_q}\vert \cF_{\rho_{r_q }}\right]\right) \\
    &\leq r(1/r_q + \varphi(1/r)t) = (1-q)/2 + r\varphi(1/r)t,
    \quad\text{ $t\in(0,\infty)$,}
\end{align*}
where the second inequality holds by the following facts: $\varphi$ is a non-decreasing function and
the inequality $1/\kappa_u\leq 1/r$ is valid on the event $\{\rho_{r_q}<u<(\rho_{r_q}+t)\wedge \lambda_{r,\rho_{r_q}}\}$. The third inequality is a consequence of the fact $((\rho_{r_q}+t)\wedge \lambda_{r,\rho_{r_q}})-\rho_{r_q}=\tau_{1/r}\wedge t\leq t$, while the last equality follows from the definition of $r_q$.
By taking complements, we get 
$$\P(\lambda_{r,\rho_{r_q }} \geq \rho_{r_q } + t\vert \cF_{\rho_{r_q}}) \geq 1 - ((1-q)/2 + r\varphi(1/r)t).$$
Setting $t = \eps/ (r\varphi(1/r))$ and recalling 
$\eps=(1-q)/2$,
we obtain~\eqref{eq:r_squared_time_to_come_back}.

Note that on the event $\{\lambda_{r,\rho_{r_q }} \geq \rho_{r_q} + \eps /(r\varphi(1/r))\}$, 
for any non-decreasing function $f$, 
we have $f(\kappa_{\rho_{r_q} + t}) \geq f(r)$ for all $t\in[0,\eps /(r\varphi(1/r))]$.
Since $r>\ell$, 
on the event 
$\{\rho_{r_q } < \lambda_{\ell}\}$, the inequality 
$\lambda_{\ell}\geq \lambda_{r,\rho_{r_q }}$
holds,
implying 
the following inclusion:
$$\left\{\int_0^{\lambda_{\ell}} f(\kappa_t)\ud t \geq f(r) \eps /(r\varphi(1/r))\right\}\supset \{\rho_{r_q } < \lambda_{\ell}\}\cap\{\lambda_{r,\rho_{r_q }} \geq \rho_{r_q } +\eps/ (r\varphi(1/r))\}.$$ 
By the inequality in~\eqref{eq:r_squared_time_to_come_back}, we thus obtain the inequality in~\eqref{eq:return_lower_bound}:
\begin{align*}
    \P\Bigg( \int_0^{\lambda_{\ell}} &f(\kappa_t) \ud t \geq f(r)\eps /(r\varphi(1/r))\Big \vert\cF_0\Bigg) \\ 
    &\geq \E\left[\mathbbm{1}\{\rho_{r_q } < \lambda_{\ell}\}\P\left(\lambda_{r,\rho_{r_q }} > \rho_{r_q } + \eps /(r\varphi(1/r))\vert \cF_{\rho_r}\right)  \big\vert\cF_0\right] 
    \geq q\P(\rho_{r_q }<\lambda_{\ell}\vert\cF_0)\quad\text{a.s.} \qedhere
\end{align*}
\end{proof}

\section{Lower bounds on the ergodicity of Markov processes}
\label{sec:proofs}

\subsection{A lower bound on the \texorpdfstring{$f$}{f}-variation rate of a Markov process}
In this subsection we consider a strong Markov process $X = (X_t)_{t\in\RP}$ on a general  metric space $\cX$ with invariant measure $\pi$ on $\cB(\cX)$ (see Section~\ref{subsec:Definitions} for definitions). 
The following lemma generalizes to \textit{$f$-variation} the lower bound in~\cite[Thm~3.6]{Hairer09} on
the \textit{total variation} between $\pi$ and the law of $X_t$.  The key assumption in Lemma~\ref{lem:lower_bound_f_convergence_rate} is the lower bound on the decay of the tail of the integral of $f$ with respect to $\pi$. We stress that Lemma~\ref{lem:lower_bound_f_convergence_rate} does not require the $\mathbf{L}$-drift condition~\nameref{sub_drift_conditions}.

\begin{lem}
\label{lem:lower_bound_f_convergence_rate}
Let $X$ be  a Markov process with an invariant measure $\pi$ on the state space $\cX$. Let functions $H,f,G:\cX \to [1,\infty)$ be such that $f(x)G(x) = H(x)$ for all $x\in\cX$ and~\ref{lem:lower_bound_f_convergence_rate_a} \& \ref{lem:lower_bound_f_convergence_rate_b} hold.
\begin{myenumi}[label=(\alph*)]
    \item \label{lem:lower_bound_f_convergence_rate_a}
There exists a function $a:[1,\infty)\to(0,1]$ such that the function $A(r)\coloneqq ra(r)$ is increasing, $\lim_{r\uparrow\infty}A(r)=\infty$ and $\int_{\{G\geq r\}}f(x)\pi(\ud x)\geq a(r)$ for all $r\in[1,\infty)$.
\item \label{lem:lower_bound_f_convergence_rate_b} There exists a function $v:\cX\times\RP\to[1,\infty)$, increasing in the second argument and satisfying  $\E_{x}[H(X_t)] \leq v(x,t)$ for all $x\in\cX$ and $t\in[1,\infty)$.
\end{myenumi}
Then the following bound holds for every $t\in[1,\infty)$ and $x\in\cX$:
$$
\|\pi(\cdot)-\P_{x}(X_t\in\cdot)\|_{f} \geq \left(a\circ A^{-1}\circ (2v)\right)(x,t)/2.
$$
\end{lem}

\begin{proof}
It follows from the definition of $f$-variation distance and Markov inequality that, for every $t\in\RP$ and every $r\geq 1$, one has the lower bound
\begin{align*}
\|\pi(\cdot)-\P_{x}(X_t\in\cdot)\|_{f} &\geq \int_{\{G\geq r\}}f(x)\pi(\ud x) - \E_{x}[f(X_t)\mathbbm{1}\{G(X_t)\geq r\}] \\
&\geq a(r) - \frac{1}{r}\E_{x}[f(X_t)G(X_t)\mathbbm{1}\{G(X_t)\geq r\}] \\
&\geq a(r) - \frac{1}{r}\E_{x}[H(X_t)] \geq a(r) - \frac{v(x,t)}{r}.
\end{align*}
Let $r=r(t)$ be the unique solution to the equation $ra(r) = 2v(x,t)$. Put differently we have 
$r(t)=A^{-1}(2v(x,t))$ for all $t\in[1,\infty)$ and $v(x,t)/r(t)=a(r(t))/2$.
Thus we obtain
$$
a(r(t))-v(x,t)/r(t) = a(r(t))/2=a(A^{-1}(2v(x,t)))/2,
$$
which, combined with the previous display, concludes the proof.
\end{proof}

\begin{rem}
\label{rem:lem_lower_bound}
In applications of Lemma~\ref{lem:lower_bound_f_convergence_rate}  
in practice, a good choice of $H=f\cdot G$ (recall that $f$ is given by the variation norm) requires balancing (I) and (II) below. \\
\noindent (I) It is beneficial to choose $H$ so that the lower bound $a(r)$ on the tail $\pi(\{G\geq r\})$ is such that $A(r)=ra(r)$ tends to infinity polynomially. This is because a slower (logarithmic) growth in $A$ would imply a faster (stretched exponential) growth in $v$ of $A^{-1}\circ (2v)$, making
 the lower bound $a\circ A^{-1}\circ (2v)$ smaller (recall that $a(r)\to0$ as $r\to\infty$). In particular, this requires $H$ to grow sufficiently fast. \\
\noindent (II) The growth of $v(x,t)$ is often obtained via the application of Lemma~\ref{lem:bounded_generator}. In particular this lemma relies on bounding $\cA H$ (differently put, the derivatives of $H$) by  a concave function of $H$, introducing a restriction on the growth of $H$.

To see how the choice of $H$ plays out in specific models, see applications of Theorem~\ref{thm:f_rate} and Corollary~\ref{cor:rate} in Section~\ref{sec:examples_proofs} below, where $H=h\circ V$, $V$ is the Lyapunov function and $h$ an arbitrary function chosen with (I) and (II) above in mind.
\end{rem}

\subsection{Return time estimates and petite sets}
\label{subsec:return_times}
The main estimate required in the proofs of our main theorems, stated in Section~\ref{sec:main_results} above, is given in Proposition~\ref{prop:lyapunov_return_times}. It essentially bounds from below the tail of  the return time $S_{(\ell)}$ of the process $X$ into the set $\{V< \ell\}$.

\begin{prop}~\label{prop:lyapunov_return_times}
Let Assumption~\nameref{sub_drift_conditions} hold.
Then there exists $\ell_0\in[1,\infty)$  such that the following holds: for any  $\ell\in(\ell_0,\infty)$ there exists $C_\ell\in(0,\infty)$, such that for any
$x\in\{\ell+1\leq V\}$,
a non-decreasing continuous function $h:\RP\to\RP$, $q\in(0,1)$ and $\eps \in (0,(1-q)/2]$,  inequality~\eqref{eq:Lyapunov_lower_bound} holds for
all times $t>h(\ell)\eps/(\ell\varphi(1/\ell))$,
\begin{multline}
\label{eq:Lyapunov_lower_bound}
\P_x\left(\int_0^{S_{(\ell)}} h\circ V(X_s)\ud s \geq t\right) \\\geq q\1{V(x)<2G_h(t)/(1-q)}\frac{C_\ell}{\Psi(2G_h(t)/(1-q))}+q\1{V(x)\geq 2G_h(t)/(1-q)}.
\end{multline}
In~\eqref{eq:Lyapunov_lower_bound},
$G_h:(h(\ell)\eps/(\ell\varphi(1/\ell)),\infty)\to(\ell,\infty)$ is the inverse of  $r\mapsto h(r)\eps/(r\varphi(1/r))$ on $(\ell,\infty)$. 
\end{prop}

The  proof of Proposition~\ref{prop:lyapunov_return_times} is based on Lemma~\ref{lem:return_times} and requires us to show that Assumption~\nameref{sub_drift_conditions} 
implies the assumptions of Lemma~\ref{lem:return_times} for the process $\kappa=V(X)$.

\begin{proof}
Consider the process $\kappa =V(X)$. 
Recall the definition of the return time $\lambda_{\ell,T}$ and the first-passage time $\rho_{r,T}$ (where $\ell,r\in(0,\infty)$ and $T$ an $(\cF_t)$-stopping time) for the process $\kappa$ in~\eqref{eq::lambda} and~\eqref{eq::rho} respectively. As in Section~\ref{sec:return_times}, we  denote $\lambda_\ell=\lambda_{\ell,0}$
and 
$\rho_r=\rho_{r,0}$. 
Note that, by Assumption~\nameref{sub_drift_conditions},  
we have $S_{(\ell)}=\lambda_\ell$ and 
$T^{(r)}=\rho_r$
and the process $\kappa$ satisfies the non-confinement property for every starting point $x\in\cX$,
 i.e.
 $\P_x(\limsup_{t\to\infty}\kappa_t=\infty)=1$.
Moreover, by~\nameref{sub_drift_conditions}\ref{sub_drift_conditions(i)}, there exists $\ell_0\in[1,\infty)$ such that, for every $r\in(\ell_0,\infty)$, the process 
$$
\left(1/\kappa_{(\rho_{r_q}+t)\wedge \lambda_{r,\rho_{r_q}}} - \int_{\rho_{r_q}}^{(\rho_{r_q} + t) \wedge \lambda_{r,\rho_{r_q}}} \varphi(1/\kappa_{u})\ud u\right)_{t\in\RP}
$$
is an $(\cF_{\rho_{r_q}+t})$-supermartingale under $\P_x$ (since $\int_{\rho_{r_q}}^{\lambda_{r,\rho_{r_q}}}\1{\kappa_u\leq \ell_0} \ud u=0$ $\P_x$-a.s.) for every $x\in\cX$, where $r_q \coloneqq 2r/(1-q)$.
Thus, by the inequality in~\eqref{eq:return_lower_bound} of Lemma~\ref{lem:return_times},  for any $\ell\in[\ell_0,\infty)$, $r\in(\ell,\infty)$, non-decreasing function $h:\RP\to\RP$, $q\in(0,1)$ and $\eps \in (0,(1-q)/2]$
we have
\begin{equation}
\label{eq:Markov_lower_bound}
\P_x\left(\int_0^{\lambda_\ell} h(\kappa_s)\ud s \geq\eps h(r)/(r\varphi(1/r))\right)\geq q\P_x(\rho_{r_q}<\lambda_\ell).
\end{equation}

Recall that by~\nameref{sub_drift_conditions}\ref{sub_drift_conditions(i)} the function $r\mapsto h(r)\eps/(r\varphi(1/r))$ is continuous and
 increasing  on $[1,\infty)$
(and thus 
invertible on $(\ell_0,\infty)$), with inverse $G_h$ is defined on $t\in(h(\ell_0)\eps/(\ell_0\varphi(1/\ell_0)),\infty)$.
For any $t>h(\ell)\eps/(\ell\varphi(1/\ell))$, set
$r=G_h(t)>\ell$ and note  $r_q=2G_h(t)/(1-q)>\ell+1$. 
The inequality in~\eqref{eq:Lyapunov_lower_bound} 
follows from~\eqref{eq:Markov_lower_bound}
and inequality~\eqref{eq:assumption_heavy_tail_bound} in~\nameref{sub_drift_conditions}\ref{sub_drift_conditions(ii)} for all 
$x\in\{\ell+1\leq V\}$, since on the subset $x\in\{r_q\leq V\}$
we have 
$\P_x(\rho_{r_q}<\lambda_\ell)=\P_x(T^{(r_q)}<S_{(\ell)})=1$ by definition.
\end{proof}

A non-empty measurable set $B\in\cB(\cX)$ is  \textit{petite} (for the Markov process $X$) if there exist a probability measure $a$ on $\cB(\RP)$ and a finite measure $\nu_a$ on $\cB(\cX)$ with  $\nu_a(\cX)>0$, satisfying
\begin{equation}
\label{eq:petite}
\int_0^\infty \P_x(X_t\in \cdot)a(\ud t) \geq \nu_a(\cdot) \quad \text{for all $x\in B$}.
\end{equation}

The following lemma shows that, under Assumptions~\nameref{sub_drift_conditions}, every petite set for $X$ belongs to a sublevel set of the Lyapunov function $V$. 

\begin{lem}[Under $\mathbf{L}$-drift condition, petite sets are bounded]
\label{lem:bounded_petite_sets}
Let~\nameref{sub_drift_conditions} hold. Assume that a set $B\in\cB(\cX)$ is petite for the process $X$. Then there exists $r_0\in(1,\infty)$ such that $B\subset \{V\leq r_0\}$.
\end{lem}

The proof of this lemma is based on a simple idea, which we first explain informally. Since $\nu_a$ in~\eqref{eq:petite} is a non-zero measure, we have $\nu_a(D)>0$ for some compact set $D$.  Denote by 
$\tau_D(0) \coloneqq \inf\{t>0:  X_t\in D\}$
the first time $X$ is in  $D$
and let $\tau_a$ be an independent random time  with law $a$ (in~\eqref{eq:petite}). Pick $t_0\in(0,\infty)$ such that $\P(\tau_a> t_0)
\leq \nu_a(D)/2$
and note $\{X_{\tau_a}\in D,\tau_a\leq t_0\}\subset 
\{\tau_D(0)\leq t_0\}$.
Since, by~\eqref{eq:petite}, it holds
$$\nu_a(D)\leq \P_x(X_{\tau_a}\in D)\leq \nu_a(D)/2+\P_x(X_{\tau_a}\in D,\tau_a\leq t_0)\leq \nu_a(D)/2+\P_x(\tau_D(0)\leq t_0),$$ we get
$0<\nu_a(D)/2\leq \P_x(\tau_D(0)\leq t_0)$ for all starting points $x$ in the petite set $B$.
However, under the $\mathbf{L}$-drift condition~\nameref{sub_drift_conditions} (by Proposition~\ref{prop:lyapunov_return_times}) we have
$\P_x(\tau_D(0)\leq t_0)\leq 1-q$ for any $q\in(0,1)$ and all $x$ with $V(x)$ sufficiently large. Hence $B$ must be contained in a sublevel set of $V$.

\begin{proof}[Proof of Lemma~\ref{lem:bounded_petite_sets}]
Let $B$ be an arbitrary petite set with a probability measure $a$ on $\cB(\RP)$ and a non-zero measure $\nu_a$ on $\cB(\cX)$ such that~\eqref{eq:petite} holds. Since $\cup_{\ell=1}^\infty \{V\leq \ell\} = \cX$ and, by~\eqref{eq:petite}, $1\geq \nu_a(\cX)>0$, there exists $\ell_1\in(1,\infty)$ such that $c \coloneqq\nu_a(\{V\leq \ell_1\})\in(0,1]$. 
By Proposition~\ref{prop:lyapunov_return_times} (with $h\equiv1$), there exist $\ell_0\in[\ell_1,\infty)$ such that for every $q\in(0,1)$, $\eps=(1-q)/2$ and $x\in\cX$ we have
\begin{equation}
\label{eq:return_bound_petite}
\P_x(S_{(\ell_0)} \geq t) \geq q\quad\text{ for all $t\in(\eps/(\ell_0 \varphi(1/\ell_0)),\infty)$ and 
$x\in\{V\geq 2G_1(t)/(1-q)\}$,}
\end{equation}
 where $G_1:(\eps/(\ell_0\varphi(1/\ell_0)),\infty)\to(\ell_0,\infty)$ is the inverse of the function $r\mapsto \eps/(r\varphi(1/r))$. Since $a$ is a probability measure on $\cB(\RP)$, there exists $t_1\in(\eps/(\ell_0 \varphi(1/\ell_0),\infty)$ with $a([t_1,\infty))<c/2$.

Pick 
$q\in(1-c/2,1)$ and
define $r_0\coloneqq 2G_1(t_1)/(1-q)$. 
Since $G_1$ is increasing, we have
$r_0>G_1(t_1)>\ell_0\geq \ell_1$. Moreover,
since the return times satisfy $S_{(\ell_0)}\leq S_{(\ell_1)}$,
for any  $x\in \{V\geq r_0\}$ the inequality in~\eqref{eq:return_bound_petite} yields
$\P_x(S_{(\ell_1)}<t_1)\leq\P_x(S_{(\ell_0)}<t_1)<1-q<c/2$.

For  $x\in \{V\geq r_0\}$,  the inequalities $\P_x(V(X_t)\leq \ell_1)\leq \P_x(S_{(\ell_1)}<t)\leq \P_x(S_{(\ell_1)}<t_1)<c/2$ hold for all $t\in[0,t_1]$. 
Since $a([t_1,\infty))<c/2$, by~\eqref{eq:petite} the following inequalities hold for all $x\in B
\cap \{V\geq r_0\}$,
$$
c =\nu_a(\{V\leq \ell_1\})\leq\int_0^\infty \P_x(V(X_t)\leq \ell_1)a(\ud t) \leq \int_0^{t_1}\P_x(V(X_t)\leq \ell_1)a(\ud t) + a([t_1,\infty)) <c,
$$ 
implying $B\cap\{V\geq r_0\}=\emptyset$. Put differently, $B\subset\{V<r_0\}$ and the lemma follows.
\end{proof}

\subsection{Proofs of the main results}
We begin with the proof of the lower bounds on modulated moments stated in Theorem~\ref{thm:modulated_moments} above. This theorem will play a crucial role in the analysis of the stability of $X$ and, more specifically, in the proof of Theorem~\ref{thm:invariant}.

\begin{proof}[Proof of Theorem~\ref{thm:modulated_moments}]
Fix a set $D\in\cB(\cX)$, contained in $\{V\leq m\}$
for some $m\in(1,\infty)$, $q\in(0,1)$ and $\eps = (1-q)/2$. Since the function 
$h:[1,\infty)\to[1,\infty)$ in Theorem~\ref{thm:modulated_moments}
is continuous and non-decreasing by assumption, Proposition~\ref{prop:lyapunov_return_times} implies that there exist $\ell_0\in(m,\infty)$ and $C_{\ell_0}\in(0,1)$, such that
the inequality in~\eqref{eq:Lyapunov_lower_bound} holds for $h$, $\ell=\ell_0$ and all $r\in(\ell_0,\infty)$. 

By~\nameref{sub_drift_conditions}, 
the  function $r\mapsto \eps h(r)/(r\varphi(1/r))$ on $(\ell_0,\infty)$ 
is  increasing  and tends to infinity.
Define $r_0\coloneqq \eps h(\ell_0)/(\ell_0\varphi(1/\ell_0))$
and denote by $G_h:(r_0,\infty)\to(\ell_0,\infty)$
its increasing inverse. 
Since $C_{\ell_0}\in(0,1)$ and $\Psi:[1,\infty)\to[1,\infty)$, for all $x\in \{ V\geq \ell_0 +1 \}$ and $r\in(r_0,\infty)$,
the inequality in~\eqref{eq:Lyapunov_lower_bound} yields
\begin{equation}
\label{eq:lower:bound_additive_functional}
 \P_x\left(\int_0^{S_{(\ell_0)}}h\circ V(X_s)\ud s \geq r\right) \geq \frac{qC_{\ell_0}}{\Psi(2G_h(r)/(1-q))}.
\end{equation}

Note that $\E_x[\int_0^\infty \1{V(X_s)> \ell_0+1} \ud s]>0$.
Indeed, if $\P_x( V(X_s)> \ell_0+1)=0$ for Lebesgue almost every $s\in\RP$, the right-continuity of $X$ would imply $\sup_{s\in\RP}V(X_s)\leq \ell_0+1$ $\P_x$-a.s., contradicting the assumption
$\limsup_{t\to\infty}V(X_t)=\infty$ $\P_x$-a.s. In particular, since the expectation is positive, 
there exists $\delta>0$ satisfying
$\P_x(V(X_\delta)>\ell_0+1)>0$. 

Recall that $\tau_D(\delta) = \inf\{t>\delta:  X_t\in D\}$ is the first time, after time $\delta\geq0$, the process  $X$ hits the set $D\in\cB(\cX)$ fixed above.
By conditioning at time $\delta$, applying the Markov property of $X$ and the inequality in~\eqref{eq:lower:bound_additive_functional}, we obtain the following lower bound
\begin{align*}
\P_x\left(\int_0^{\tau_D(\delta)}h\circ V(X_s)\ud s \geq r\right) &\geq \P_x\left(\int_\delta^{\tau_D(\delta)}h\circ V(X_s)\ud s \geq r,V(X_\delta)>\ell_0+1)\right) \\
&\geq \E_x\left[\1{V(X_\delta) > \ell_0+1}\cdot \P_{X_\delta}\left(\int_0^{\tau_D(0)}h\circ V(X_s)\ud s \geq r\right)\right]\\
&\geq \E_x\left[\1{V(X_\delta) > \ell_0+1}\cdot \P_{X_\delta}\left(\int_0^{S_{(\ell_0)}}h\circ V(X_s)\ud s \geq r\right)\right] \\ 
&\geq qC_{\ell_0}\P_x(V(X_\delta)>\ell_0+1)/\Psi(2G_h(r)/(1-q))\quad \text{for $r\in(r_0,\infty)$},
\end{align*}
where the third inequality follows from the fact that, since $D\subset \{V\leq m\}$ and $m\leq \ell_0$, 
starting from any point in $\{V>\ell_0+1\}$
the first hitting time $\tau_D(0)$ satisfies
 $S_{(\ell_0)}\leq\tau_D(0)$.
Since $\delta$ was chosen so that $\P_x(V(X_\delta)>\ell_0+1)>0$, setting 
$C:=qC_{\ell_0}\P_x(V(X_\delta)>\ell_0+1)$ concludes the proof of part~\ref{thm:modulated_moments_a}.
Part~\ref{thm:modulated_moments_b} is a special case of part~\ref{thm:modulated_moments_a} for the function $h \equiv 1$.
\end{proof}

The following corollary combines the lower bounds of Theorem~\ref{thm:modulated_moments} with the fact that,
under Assumption~\nameref{sub_drift_conditions}, any petite set of $X$ is contained in a sublevel set of the Lyapunov function $V$ (see Lemma~\ref{lem:bounded_petite_sets} above). The result provides a sufficient condition (in the form of an integral test) for the divergence of the expectation with respect to invariant measure $\pi$ of a non-decreasing function composed with $V$.   

\begin{cor}
\label{cor:criteria_for_infinite_moments}
Let Assumption~\nameref{sub_drift_conditions} hold. Then for every $q\in(0,1)$ and a non-decreasing function $h:[1,\infty)\to[1,\infty)$, the following implication holds:
\begin{equation}
\label{eq:invariant_infinite_moment_condition}
\text{$\exists r'\in(0,\infty)$ s.t.}\quad \int_{r'}^\infty \frac{1}{\Psi(2G_h(r)/(1-q))}\ud r = \infty \quad
\implies \quad
\int_{\cX} h \circ V(x)\pi( \ud x) = \infty,
\end{equation}
where $G_h$ is the inverse of the increasing function $r \mapsto (1-q)h(r)/(2r\varphi(1/r))$. 
\end{cor}

\begin{proof}
By Assumption~\nameref{sub_drift_conditions} the process $X$ is positive Harris recurrent. The seminal result~\cite[Thm~1.2(b)]{tweedie1993generalized} implies that a measurable  $h:[1,\infty) \to[1,\infty)$ satisfies the implication:
$$
\int_\cX h\circ V(x)\pi(\ud x) <\infty  \>\implies\>  \text{$\exists$ closed petite  $D$ s.t. $\forall \delta>0$, } \sup_{x\in D}\E_x\left[\int_0^{\tau_D(\delta)}h\circ V(X_s)\ud s\right]<\infty.
$$ 
By Lemma~\ref{lem:bounded_petite_sets}, \textit{every} petite set $D$ for $X$ satisfies $D\subset \{V\leq r_0\}$ for some $r_0\in[1,\infty)$. Thus, Theorem~\ref{thm:modulated_moments} implies that for a non-decreasing $h:[1,\infty)\to[1,\infty)$, every closed petite set $D$ and any $x\in D$, there exist $\delta>0$ and $C\in(0,\infty)$ such that 
$$
\E_x\left[\int_0^{\tau_D(\delta)} h\circ V(X_s)\ud s\right]
=
\int_0^\infty \P_x\left(\int_0^{\tau_D(\delta)} h\circ V(X_s)\ud s\geq r\right)\ud r
\geq \int_{r_0}^\infty \frac{C}{\Psi(2G_h(r)/(1-q))}\ud r
$$ 
for some sufficiently large $r_0\in(0,\infty)$, where $G_h$ is the inverse of the increasing function $r\mapsto (1-q)h(r)/(2r\varphi(1/r))$. 
If the assumption in the implication in~\eqref{eq:invariant_infinite_moment_condition} holds,
then the last integral in the previous display must also be infinite because the function $r\mapsto1/\Psi(2G_h(r)/(1-q))$ is continuous and thus locally bounded. 
The criterion in~\cite[Thm~1.2(b)]{tweedie1993generalized} 
stated above thus yields the conclusion of the implication in~\eqref{eq:invariant_infinite_moment_condition}.
\end{proof}

The implication in~\eqref{eq:invariant_infinite_moment_condition} in Corollary~\ref{cor:criteria_for_infinite_moments} is at the core of the proof of Theorem~\ref{thm:invariant}.
It is key that  integral test~\eqref{eq:invariant_infinite_moment_condition} covers all non-decreasing functions $h$, not only the  polynomial ones.

\begin{proof}[Proof of Theorem~\ref{thm:invariant}]
 Pick $q,\eps\in(0,1)$ and note that the statement in display~\eqref{eq:main_result_invariant}
of the theorem is equivalent to the following: 
$$\exists r_0\in(0,\infty)\text{ such that, }\quad
1/L_{\eps,q}(r)\leq\pi(\{x\in\cX:V(x)\geq r\})\quad
\text{ for all $r\in[r_0,\infty)$,}
$$
where $L_{\eps,q}(r)=r\varphi(1/r)\Psi(2r/(1-q)) (\log\log r)^{\eps}$. Assume~\nameref{sub_drift_conditions} holds.

The proof is by contradiction.
Assume  that there exists $\eps>0$, such that \textit{for every} $r_0\in(0,\infty)$ there exists $r_1\in[r_0,\infty)$ satisfying
$1/L_{\eps,q}(r_1)>\pi(\{x\in\cX:V(x)\geq r_1\})$.
We may pick $r_0>1$ and $r_1>\exp(\exp(\exp(1)))r_0$. Recursively we can define an increasing sequence $(r_n)_{n\in\N}$, satisfying 
$r_{n+1}>\exp(\exp(\exp(n+1)))r_n$ and $1/L_{\eps,q}(r_n)>\pi(\{x\in\cX:V(x)\geq r_n\})$ for all $n\in\N$. In particular, since $r_0>1$, we have
\begin{equation}
\label{eq:growth_of_r_n}
    \log\log r_n>\exp(n)\qquad\text{for all $n\in\N$.}
\end{equation}
Using the sequence $(r_n)_{n\in\N}$, we construct a non-decreasing  function $h:[1,\infty)\to[1,\infty)$, 
satisfying $\int_\cX h \circ V(x)\pi(\ud x)<\infty$ 
\textit{and} the assumption of the  implication in~\eqref{eq:invariant_infinite_moment_condition}.

Define the function 
$\mu:\RP\to\RP$ by 
$\mu(r)\coloneqq1$ for $r\in[0,r_1)$ and
$\mu(r) \coloneqq 1/L_{\eps,q}(r_n)$ for $r\in [r_n,r_{n+1})$, $n\in\N$.  
Since the function $r\mapsto \pi(\{V\geq r\}) $ is non-increasing, we have  
$\pi(\{V\geq r\})\leq \mu(r)$ for all $r\in\RP$. 
Let $h:[1,\infty)\to[1,\infty)$ be a differentiable function such that $h(r) = 1$ 
for $r\in[1,r_1)$. For $n\in\N\setminus\{1\}$ and $r\in[r_n,r_{n+1})$ we define the derivative of $h$ by 
\begin{equation}
\label{eq:def_h_prime}   
h'(r) = \begin{cases}
r\varphi(1/r)\Psi(2(r_n+1)/(1-q))(\log\log r_n)^{\eps/2},& r\in [r_n,r_{n}+1);\\
1/(r_n(r_{n+1}-r_n)),& r\in [r_n+1,r_{n+1}).
\end{cases}
\end{equation}
Since,
by Assumption~\nameref{sub_drift_conditions},
$r\mapsto r\varphi(1/r)$ is decreasing and $\Psi$ is differentiable, increasing and submultiplicative (i.e.
$\Psi(2(r_n+1)/(1-q))\leq \Psi(2r_n/(1-q))\Psi(2/(1-q))$
for all $r_n\in[1,\infty)$; without loss of generality we assume here that the constant $C$ in definition of a submultiplicative function in footnote on page~\pageref{footnote:submultiplicative} equals one, since  we may substitute $\Psi$ in~\nameref{sub_drift_conditions}\ref{sub_drift_conditions(ii)}
with $C\Psi$ if $C>1$),
we have
$$
h'(r)\mu(r) \leq \begin{cases}
\Psi(2/(1-q))(\log\log r_n)^{-\eps/2},&
r\in[r_n,r_n+1); \\
1/(r_n(r_{n+1}-r_n)),& r\in[r_n+1,r_{n+1}).
\end{cases}
$$
The identity
$1+\int_1^{V(x)}h'(r)\ud r=h(V(x))$ for all $x\in\cX$
and Fubini's theorem
imply the equality
$\int_{\cX} h(V(x))\pi(\ud x)=1+\int_1^\infty h'(r)\pi(\{V\geq r\})\ud r$.
Recall 
$\pi(\{V\geq r\})\leq \mu(r)$ for $r\in\RP$
and note
\begin{align}
\label{eq:finite_H_pi}
\nonumber\int_{\cX} h(V(x))\pi(\ud x) &=1+\int_{r_1}^\infty h'(r)\pi(\{V\geq r\})\ud r 
 \leq 1+ \int_{r_1}^\infty h'(r)\mu(r)\ud r\\
&\nonumber=1+\sum_{n=1}^\infty \left(\int_{r_n}^{1+r_n}h'(r)\mu(r)\ud r + \int_{1+r_n}^{r_{n+1}}h'(r)\mu(r)\ud r \right)\\
&\leq1+\Psi(2/(1-q))\sum_{n=1}^{\infty}(\log\log r_n)^{-\eps/2} + \sum_{n=1}^{\infty} 1/r_n<\infty,
\end{align}
where the final inequality follows from~\eqref{eq:growth_of_r_n},
which makes both sums in~\eqref{eq:finite_H_pi} clearly finite. 

Recall that 
the function $u\mapsto (1-q) h(u)/(2u\varphi(1/u))$
is increasing on $[1,\infty)$ by~\nameref{sub_drift_conditions}
and  define $r'\coloneqq (1-q)h(1)/(2\varphi(1))>0$  (in fact $h(1)=1$).
Denote by $G_h:[r',\infty)\to[1,\infty)$ the inverse 
and 
introduce the substitution $r = (1-q) h(u)/(2u\varphi(1/u))$ into the following integral:
\begin{align}
\nonumber
\int_{r'}^\infty 2/((1-q)&\Psi(2G_h(r)/(1-q)))\ud r = \int_{1}^\infty (h(u)/(u\varphi(1/u)))'/\Psi(2u/(1-q))\ud u\\
\nonumber
&= \int_{1}^\infty (h'(u)/(u\varphi(1/u))+(1/(u\varphi(1/u)))'h(u))/\Psi(2u/(1-q))\ud u\\
\nonumber
&\geq \sum_{n=n_0}^\infty \int_{r_n}^{1+r_n}\frac{u\varphi(1/u)\Psi(2(r_n+1)/(1-q))(\log\log r_n)^{\eps/2}}{\Psi(2u/(1-q))u\varphi(1/u)}\ud u \\
\label{eq:final_lower_bound_infinite}
&\geq \sum_{n=n_0}^\infty\int_{r_n}^{1+r_n} (\log\log r_n)^{\eps/2}\ud u = \sum_{n=n_0}^\infty (\log\log r_n)^{\eps/2} = \infty.
\end{align}
The first inequality follows from the definition of $h'$ given in~\eqref{eq:def_h_prime} above and the fact that the function $u\mapsto 1/(u\varphi(1/u))$ is continuous and increasing and hence almost everywhere differentiable with a non-negative derivative. The second inequality in the previous display follows from the fact that $\Psi$ is increasing
by Assumption~\nameref{sub_drift_conditions}. The divergence of the sum is a consequence of the inequality in~\eqref{eq:growth_of_r_n}.
By Corollary~\ref{cor:criteria_for_infinite_moments}, the inequality in~\eqref{eq:final_lower_bound_infinite} implies
$\int_\cX h\circ V(x)\pi(\ud x) = \infty$, which contradicts~\eqref{eq:finite_H_pi} and concludes the proof of Theorem~\ref{thm:invariant}.
\end{proof}

The drift condition on the Lyapunov function $V$ in Assumption~\nameref{sub_drift_conditions} and the lower bound on the invariant measure $\pi$ from Theorem~\ref{thm:invariant} are the key ingredients in the proof of the lower bound on the rate of convergence in total variation. For the $f$-variation distance we require the following corollary of Theorem~\ref{thm:invariant}.

\begin{cor}
\label{cor:f_invariant_lower_bound}
Let  Assumption~\nameref{sub_drift_conditions} hold.   Let $f_\star:[1,\infty)\to[1,\infty)$  be a differentiable function and consider an increasing continuous $g:[1,\infty)\to[1,\infty)$, satisfying $\lim_{r\to\infty}g(r)=\infty$. Then, for every $\eps,q\in(0,1)$ and the function $L_{\eps,q}$ in~\eqref{eq:def_L_eps_q}, there exists  $c_{\eps,q}\in(0,\infty)$ such that
\begin{equation}
\label{eq:lower_bound_tail_f}
\int_{\{g\circ V\geq r\}} f_\star\circ V(x)\pi(\ud x) \geq c_{\eps,q} f_\star(g^{-1}(r))/L_{\eps,q}(g^{-1}(r))
\quad \text{for all $r\in[1,\infty)$.}
\end{equation}
\end{cor}

\begin{proof}
Pick $\eps,q\in(0,1)$. Then, by Theorem~\ref{thm:invariant}, there exists a constant $c_{\eps,q}\in(0,\infty)$ such that 
$$\pi(\{V\geq g^{-1}(r)\})\geq c_{\eps,q}/L_{\eps,q}(g^{-1}(r))\quad\text{for all $r\in[1,\infty)$.}
$$

Using the the inequality above along with the facts that $f$ is differentiable and $g$ is continuous increasing and thus has an inverse $g^{-1}$, we obtain  
\begin{align*}
\int_{\{g\circ V\geq r\}}  f_\star\circ V(x)\pi(\ud x) &= \int_{\{V\geq g^{-1}(r)\}} \left(f_\star(1) + \int_1^{V(x)} f_\star'(y)\ud y\right)\pi(\ud x) \\
&= f_\star(1)\pi(\{V\geq g^{-1}(r)\}) + \int_1^\infty  f_\star'(y) \pi(\{V\geq \max\{g^{-1}(r), y\}\})\ud y\\
&\geq  f_\star(1)\pi(\{V\geq g^{-1}(r)\}) + \int_1^{g^{-1}(r)} f_\star'(y) \pi({\{V\geq g^{-1}(r)\}}) \ud y\\
&=  f_\star(g^{-1}(r))\pi(\{V\geq g^{-1}(r)\})\geq c_{\eps,q}f_\star(g^{-1}(r))/L_{\eps,q}(g^{-1}(r)).\qedhere
\end{align*}
\end{proof}

We now establish the lower bound on the convergence in $f$-variation in Theorem~\ref{thm:f_rate}.

\begin{proof}[Proof of the Theorem~\ref{thm:f_rate}]
By Assumption~\ref{assumption:f_convergence_a} in  Theorem~\ref{thm:f_rate}, a differentiable  $f_\star:[1,\infty)\to[1,\infty)$ and continuous  $h,g:[1,\infty)\to[1,\infty)$ satisfy $g = h/f_\star$ on $[1,\infty)$, with $g$ increasing and $\lim_{r\to\infty}g(r)=\infty$. Moreover, there exists a function
$v:\cX\times\RP\to[1,\infty)$, increasing in $t$,
such that 
$\E_x[h\circ V(X_t)]\leq v(x,t)$
for all $x\in\cX$ and $t\in\RP$.
By Corollary~\ref{cor:f_invariant_lower_bound}, for any $\eps,q\in(0,1)$ there exists a constant $c_{\eps,q}\in(0,\infty)$, such that the inequality in~\eqref{eq:lower_bound_tail_f} holds.

Any function $a:[1,\infty)\to \RP$, satisfying 
Assumption~\ref{assumption:f_convergence_b} in  Theorem~\ref{thm:f_rate} (i.e. the inequality 
$a(r)\leq  c_{\eps,q} f_\star(g^{-1}(r))/L_{\eps,q}(g^{-1}(r))$ holds for $r\in[1,\infty)$ and the function $r\mapsto ra(r)$ is increasing with $\lim_{r\to\infty} ra(r)=\infty$), by~\eqref{eq:lower_bound_tail_f} also satisfies Assumption~\ref{lem:lower_bound_f_convergence_rate_a} of Lemma~\ref{lem:lower_bound_f_convergence_rate} with 
$f=f_\star\circ V$ and
$G\coloneqq g\circ V$.
As observed in the previous paragraph, the functions 
$H\coloneqq h\circ V$ and $v(x,t)$
satisfy the condition in 
Assumption~\ref{lem:lower_bound_f_convergence_rate_b} of Lemma~\ref{lem:lower_bound_f_convergence_rate}:
$\E_x[H(X_t)]\leq v(x,t)$ 
for all $x\in\cX$ and $t\in\RP$.
An application of Lemma~\ref{lem:lower_bound_f_convergence_rate} 
concludes the proof of the theorem.
\end{proof}

\begin{proof}[Proof of Lemma~\ref{lem:assumption_submart_exit_prob}]
Pick $\ell\in(\ell_0,\infty)$, $r\in(\ell+1,\infty)$ and $x\in\{\ell+1\leq V<r\}$, and recall the definitions $T^{(r)}\coloneqq\inf\{t\geq0:V(X_t)>r\}$ and $S_{(\ell)} \coloneqq \inf\{t\geq 0: V(X_t)<\ell\}$. Assumption of the lemma implies that for some $d\in[1,\infty)$, we have $V(X_{T^{(r)}})-V(X_{T^{(r)}-})\leq d$ $\P_x$-a.s, and since $\Psi$ is increasing we obtain $\Psi \circ V(X_{t\wedge S_{(\ell)}\wedge T^{(r)}})\leq \Psi (r+d)$ for all $t\in\RP$ $\P_x$-a.s. Moreover, given that $\ell\in(\ell_0,\infty)$ it follows that $\int_0^{\cdot\wedge S_{(\ell)}\wedge T^{(r)}} \1{V(X_u)\leq \ell_0}\equiv0$ $\P_x$-a.s. Thus, by the assumption of the lemma and the optional sampling theorem the process $\Psi \circ V (X_{\cdot \wedge S_{(\ell)}\wedge T^{(r)}})$ is an $(\cF_t)$-submartingale under $\P_x$. 

We establish a lower bound on $\P_x(T^{(r)}<S_{(\ell)})$ as follows. By assumption in Lemma~\ref{lem:assumption_submart_exit_prob} we have $\limsup_{t\to\infty} V(X_t) = \infty$ $\P_x$-a.s., which implies $T^{(r)}\wedge S_{(\ell)}\leq T^{(r)}<\infty$ $\P_x$-a.s. The dominated convergence theorem and the monotonicity of $\Psi$ yield
$$
\Psi(V(x))\leq \lim_{t\to\infty}\E_x[ \Psi \circ V (X_{t \wedge S_{(\ell)}\wedge T^{(r)}})] = \E_x[ \Psi \circ V (X_{ S_{(\ell)}\wedge T^{(r)}})] \leq \Psi(\ell) + \P_x(T^{(r)}<S_{(\ell)})\Psi(r + d).
$$
Thus, $\P_x(T^{(r)}<S_{(\ell)})\geq (\Psi( V(x))-\Psi(\ell))/\Psi(r+d)\geq (\Psi( V(x))-\Psi(\ell))/(C\Psi(r)\Psi(d))\geq C_\ell/\Psi(r)$, where $C_{\ell} \coloneqq (\Psi(\ell+1)-\Psi(\ell))/(C\Psi(d))$.
The second inequality holds since
$\Psi$ is submultiplicative (with a constant $C>0$) and 
the third holds because $\Psi$ is increasing (both properties are assumed in the lemma). 
Noting $C_{\ell} >0$ concludes the proof of the inequality in~\nameref{sub_drift_conditions}\ref{sub_drift_conditions(ii)}. 
\end{proof}

\begin{proof}[Proof of Theorem~\ref{thm:generator}]
\noindent \underline{Supermartingale condition~\ref{generator_a}}. Fix arbitrary $x\in\cX$ and let $\ell_0\in(1,\infty)$ be such that the inequality in~\eqref{eq:generator_condition_super} holds. By assumption 
in Theorem~\ref{thm:generator}\ref{generator_a},
we have $1/V \in \cD(\cA)$. Thus, by~\cite[Ch~1, Def~(14.15)]{Davis}, there exists an increasing sequence $\{T_n:n\in\N\}$ of $(\cF_t)$-stopping times, satisfying
$T_n\uparrow \infty$ as $n\to\infty$ $\P_x$-a.s. 
and the localised process
$$
1/V(X_{\cdot \wedge T_n})-1/V(x) - \int_0^{\cdot\wedge T_n} \cA (1/V)(X_s)\ud s \quad\text{is an $(\cF_t)$-martingale under $\P_x$ for all $n\in\N$. }
$$ 
By assumptions of the theorem, we have 
$0<1/V(x')\leq 1$ for  $x'\in\cX$
and $0\leq \varphi(u)\leq \varphi(1)$ for  $u\in(0,1]$,
implying 
$\E_x[1/V(X_{t\wedge T_{n}})]<\infty$ and
 $\E_x[\int_0^{t} \varphi (1/V(X_s))\ud s]<\infty$ for all $t\in\RP$, respectively.
 By
the inequality in~\eqref{eq:generator_condition_super},
there exists $b\in\RP$, such that for every $n\in\N$ we have
\begin{align*}
\E_x[1/V(X_{t\wedge T_{n}})] &- \E_x\left[\int_0^{t\wedge T_{n}} \varphi (1/V(X_s))\ud s\right] \\
&= 1/V(x) + \E_x\left[ \int_0^{t\wedge T_{n}}(\cA (1/V)(X_s)- \varphi  (1/V(X_s)))\ud s\right] \\
&\leq 1/V(x) +b\E_x\left[\int_0^{t\wedge T_n} \1{V(X_s)\leq \ell_0}\ud s\right]\quad\text{ for all $t\in\RP$.}
\end{align*}
This inequality, Fatou's lemma and the monotone convergence theorem, yield 
\begin{align*}
\E_x[1/V(X_t)] &= \E_x[\liminf_{n\to\infty} 1/V(X_{t\wedge T_n})]\leq \liminf_{n\to\infty}\E_x[1/V(X_{t\wedge T_n})] \\
&\leq\liminf_{n\to\infty}\left(V(x) + \E_x\left[\int_0^{t\wedge T_n}\varphi(1/V(X_s))\ud s\right]+ b \E_x\left[\int_0^{t\wedge T_n} \1{V(X_s)\leq \ell_0}\ud s\right]\right) \\
&= 1/V(x) + \E_x\left[\int_0^{t}\varphi  (1/V(X_s))\ud s\right] + b \E_x\left[\int_0^{t} \1{V(X_s)\leq \ell_0}\ud s\right].
\end{align*}
This proves the condition~\ref{sub_drift_conditions(i)} in Assumption~\nameref{sub_drift_conditions}.

\noindent \underline{Exit probability  condition~\ref{generator_b}}. We employ analogous arguments to show that for 
some $\ell_0,c\in(0,\infty)$ and all $r\in(\ell_0,\infty)$, the process 
$$
\Psi \circ V(X_{\cdot\wedge T^{(r)} }) + c\int_0^{\cdot\wedge T^{(r)}}\1{V(X_u) \leq\ell_0}\ud u, \quad\text{is an $(\cF_t)$-submartingale under $\P_x$ for all $x\in\cX$.}
$$
The condition in~\nameref{sub_drift_conditions}\ref{sub_drift_conditions(ii)} then follows from the application of Lemma~\ref{lem:assumption_submart_exit_prob}.

Fix arbitrary $x\in\cX$ and let $\ell_0\in(1,\infty)$ be such that the inequality~\eqref{eq:generator_condition_sub} hold. 
Since $\Psi \circ V\in\cD(\cA)$ (by assumption in Theorem~\ref{thm:generator}\ref{generator_b}), as before there exists 
a localising sequence of $(\cF_t)$-stopping times 
$\{T_n:n\in\N\}$, such that $(M^{n}_t)_{t\in\RP}$, where 
$$
M^{n}_t\coloneqq \Psi \circ V(X_{t\wedge T_n})-\Psi \circ V(x)-\int_0^{t\wedge T_n}\cA (\Psi \circ V)(X_s)\ud s,
$$
is an $(\cF_t)$-martingale under $\P_x$. 
Thus, for any $r\in(\ell_0,\infty)$, the stopped process $(M^{n}_{t\wedge T^{(r)}})_{t\in\RP}$
is also an $(\cF_t)$-martingale under $\P_x$ (recall $T^{(r)}=\inf\{t\geq0:V(X_t)>r\}$). Moreover, by
assumption in~Theorem~\ref{thm:generator}\ref{generator_b} there exists $d\in[1,\infty)$ such that $V(X_{t\wedge T^{(r)}})-V(X_{t\wedge T^{(r)}-}) \leq r+d$ for all $t\in\RP$ and $r\in(\ell_0,\infty)$ $\P_x$-a.s. The fact that $\Psi$ is increasing implies $\Psi\circ V(X_{t\wedge T_n\wedge T^{(r)}})\leq \Psi(r+d)$ for all $t\in\RP$, $r\in(\ell_0,\infty)$ and $n\in\N$ $\P_x$-a.s.
Since $\E_x[M^{n}_{t\wedge T^{(r)}}]=0$, by the inequality in~\eqref{eq:generator_condition_sub}, there exists $c\in\RP$, such that for every $n\in\N$,
\begin{align*}
\E_x[\Psi \circ V(X_{t\wedge T_n\wedge T^{(r)}})] &= \Psi \circ V(x) + \E_x\left[\int_0^{t\wedge T_n\wedge T^{(r)}} \cA(\Psi \circ V)(X_s)\ud s\right]\\
&\geq \Psi \circ V(x) - c\E_x\left[\int_0^{t\wedge T_n\wedge T^{(r)}}\1{V(X_s)\leq \ell_0}\ud s\right],\>\text{$t\in\RP$, $r\in(\ell_0,\infty)$.}
\end{align*}
The dominated convergence theorem (as $n\to\infty$), applied to both sides of the inequality, yields the submartingale condition, and by Lemma~\ref{lem:assumption_submart_exit_prob}, Assumption~\nameref{sub_drift_conditions}\ref{sub_drift_conditions(ii)}.
\end{proof}

\begin{proof}[Proof of Lemma~\ref{lem:bounded_generator}]
Pick  $x\in\cX$ and recall  $H\in\cD (\cA)$.
Thus, by~\cite[Ch~1, Def~(14.15)]{Davis}, there exists an increasing sequence $\{T_n:n\in\N\}$ of $(\cF_t)$-stopping times, such that
$T_n\uparrow \infty$ as $n\to\infty$ $\P_x$-a.s. 
and  $M^{(n)}\coloneqq H(X_{\cdot \wedge T_n})-H(x) - \int_0^{\cdot\wedge T_n} \cA H(X_s)\ud s$ is an $(\cF_t)$-martingale under $\P_x$ for all $n\in\N$. Set $S_m \coloneqq \inf\{t>0:H(X_t)\geq m\}\wedge T_m$ for all $m\in\N$ and note $S_m\uparrow \infty$ as $m\to\infty$ $\P_x$-a.s.  Since $M^{(m)}_t$ is integrable and, on the event  $\{t<S_m\}$,  $H(X_{t\wedge S_m})$ is bounded, we have $\E_x[\int_0^{t\wedge S_m} \cA H(X_s)\ud s]<\infty$. Moreover, since $\cA H\leq \xi \circ H$ on $\cX$, for any $t\in\RP$ and $m\in\N$ we obtain
\begin{align*}
\E_x[H(X_{t\wedge S_m})] = H(x) + \E_x\left[\int_0^{t\wedge S_m}\cA H(X_s)\ud s\right] \leq H(x) + \E_x\left[\int_0^{t\wedge S_m} \xi\circ H(X_s)\ud s\right]
\end{align*}
This inequality, Fatou's lemma and the monotone convergence theorem, yield
\begin{align*}
\E_x[H(X_t)] &=\E_x[\liminf_{m\to\infty} H(X_{t\wedge S_m})]\leq \liminf_{m\to\infty}\E_x[H(X_{t\wedge S_m})]\\
&\leq H(x) + \liminf_{m\to\infty}\E_x\left[\int_0^{t\wedge S_m} \xi \circ H(X_s)\ud s\right] = H(x) + \E_x\left[\int_0^{t} \xi \circ H(X_s)\ud s\right]. 
\end{align*}
Since $\xi:[1,\infty)\to[1,\infty)$ is concave, Tonelli's theorem and Jensen's inequality imply $$
\E_x\left[\int_0^{t} \xi \circ H(X_s)\ud s\right]=\int_0^t \E_x\left[\xi\circ H(X_s)\right]\ud s\leq \int_0^t \xi\left(\E_x[ H(X_s)]\right)\ud s\quad \text{for all $t\in\RP$.}
$$
Denote $g(t)\coloneqq\E_x[H(X_t)]\geq 1$. Thus $g(0)=H(x)$ and 
\begin{equation}
\label{eq:generator_bound_growth}
g(t)\leq g(0)+\int_0^t\xi(g(s))\ud s\quad \text{for all $t\in\RP$.}
\end{equation}

The increasing function $\Xi:[1,\infty)\to\RP$, given by $\Xi(t) \coloneqq \int_1^t \ud s/\xi(s)$, has a differentiable inverse $\Xi^{-1}$. Denote $G(t)\coloneqq \int_0^t \xi(g(s))\ud s$. By~\eqref{eq:generator_bound_growth} we have $G'(v)=\xi(g(v))\leq \xi(g(0)+G(v))$ for all $v\in\RP$, since  $\xi$ is non-decreasing. This yields
$$
\Xi(g(0)+G(t))- \Xi(g(0))=
\int_{g(0)}^{g(0)+G(t)} \frac{\ud z}{\xi(z)}=
\int_0^{G(t)} \frac{\ud z}{\xi(g(0)+z)}=\int_0^t \frac{G'(v)}{\xi(g(0)+G(v))}\ud v\leq t.
$$
Thus
$\Xi(g(0)+G(t))\leq \Xi(g(0))+t$ and hence
$g(t)\leq g(0)+G(t)\leq \Xi^{-1}(\Xi(g(0))+t)$.
\end{proof}

\section{Proofs of the examples in Section~\ref{sec:examples}}
\label{sec:examples_proofs}

This section is dedicated to proving the theorems presented in Section~\ref{sec:examples}.
All strong Markov processes $X$ considered 
in this section are ergodic with an invariant measure $\pi$ and Feller continuous (see~\cite[Sec.~5]{douc2009subgeometric} for more details), and
thus positive Harris recurrent by~\cite[Thm~1.1]{tweedie1993generalized}. Moreover, all the models in this section are irreducible, either because of uniform ellipticity and the irreducibility of the driver (a Brownian motion or, more generally, a L\'evy process) or by a direct argument in the hypoelliptic case. Thus, for any Lyapunov function $V$ for $X$, we will have  $\P_x(V(X_t)\geq r_0)>0$ for all starting points $x$, times $t>0$  and levels $r_0$, thus satisfying the assumptions of Lemma~\ref{lem:non_confinement} below and implying non-confinement.
Verifying the $\mathbf{L}$-drift condition~\nameref{sub_drift_conditions} in our examples will thus  reduce to finding functions $V,\varphi,\Psi$, which satisfy the  conditions~\ref{sub_drift_conditions(i)} and~\ref{sub_drift_conditions(ii)} in~\nameref{sub_drift_conditions}. Since we are working with Feller processes, we will obtain these conditions by applying the generator to the relevant functions, establishing the appropriate point-wise inequalities and using Theorem~\ref{thm:generator}.

\subsection{Diffusions from Section~\ref{sec:diffusions}}
\label{subsec:diffusions_proofs}

Since the SDE in~\eqref{eq:elliptic_SDE} possesses a unique strong solution, the process $X$ is strong Markov. Feller continuity follows from~\cite[Thm~3.4.1]{khasminski}.
By It\^o's formula applied to $g(X)$, the extended generator (see Section~\ref{subsec:generator} above for definition) of the diffusion $X$ takes the following form for any twice continuously differentiable $g\in C^2(\R^n)$:
\begin{equation}
\label{eq:extended_gen_diffusion}
\cA g(x) = \langle b(x),\nabla g(x)\rangle +\frac{1}{2}\Tr\left(\Sigma(x) \Hessian(g)(x)\right),\qquad\text{for all $x\in\R^n$,}
\end{equation}
where $\Sigma=\sigma\sigma^\intercal$ is the instantaneous covariance of $X$, $\nabla g$ is the gradient of $g$,  $\Hessian(g)$ is the Hessian (i.e. the matrix  of the second derivatives of $g$) and  $\Tr(\cdot)$ denotes the trace of a matrix in $\R^{n\times n}$.

\subsubsection{Polynomial tails: proofs for  Section~\ref{subsubsec:Polynomial}}
\label{subsubsec:poly_proofs}
For any $m\in\R\setminus\{0\}$, consider a function $p_m:\R^n\to(0,\infty)$
in $C^2(\R^n)$, satisfying $p_m(x)\leq 1+|x|^m$ for all $x\in\R^n$ and
\begin{equation}
\label{eq:power_function}
p_m(x)= 
|x|^m\quad\text{ for all $x\in\R^n$ with $|x|$ sufficiently large.}
\end{equation}
If $m>0$, we assume in addition that $p_m:\R^n\to[1,\infty)$ takes values in $[1,\infty)$ only.
Recall parameters $\alpha,\beta,\gamma\in(0,\infty)$, $\ell\in[0,2)$ and $m_c$ in~\ref{example:ass_diffp}. 
The following (deterministic) proposition allows us to construct the functions in~\nameref{sub_drift_conditions}.

\begin{prop}
\label{prop:deterministic_eliptic_poly}
Under~\ref{example:ass_diffp}, extended generator~\eqref{eq:extended_gen_diffusion} 
of the diffusion $X$ in~\eqref{eq:elliptic_SDE}
satisfies 
the following asymptotic inequalities:
if $m\in(-\infty,m_c)\setminus \{0\}$ (resp.  $m\in(m_c,\infty)$),
then $\cA p_m(x)\leq C_0 |x|^{m+\ell-2}$ 
(resp. $\cA p_m(x) \geq C_0 |x|^{m+\ell-2}$)
for all $x\in\R^n$ with $|x|$ sufficiently large
and $C_0\coloneqq \max\{-m\beta (m_c-m), -m\beta (m_c-m)/4\}\in\R\setminus\{0\}$ (resp. $C_0\coloneqq m\beta (m-m_c)/4\in(0,\infty)$).
\end{prop}

\begin{proof}
Since $\nabla p_m(x)=m|x|^{m-2}x$ and
$\Hessian(p_m)(x)=m|x|^{m-2}((m-2)xx^\intercal/|x|^2+I_n)$ for all 
$m\in\R\setminus \{0\}$
and 
$x\in\R^n$ with sufficiently large $|x|$ ($I_n\in\R^{n\times n}$ is the identity matrix), 
the representation of $\cA$ in~\eqref{eq:extended_gen_diffusion}, the identity $\langle \Sigma(x)x/|x|,x/|x| \rangle =\Tr\left(\Sigma(x)xx^\intercal/|x|^2\right)$ for all points $x\in\R^n\setminus\{0\}$
and \ref{example:ass_diffp}
yield
\begin{align}
\label{eq:extended_gen_diff_polynomial}
\cA p_m (x) &=
-\frac{m\beta}{2} \left(\frac{2\alpha -\gamma}{\beta} -(m-2) +o(1)\right)|x|^{m+\ell-2}=
-\frac{m\beta}{2} (m_c-m+o(1))|x|^{m+\ell-2},
\end{align}
for all $x\in\R^n$ with $|x|$ sufficiently large. Recall $m_c>0$ by~\ref{example:ass_diffp}. If $m\in(-\infty,m_c)\setminus \{0\}$ (resp. $m\in(m_c,\infty)$), then a constant $C_0$
and the inequalities follow from representation~\eqref{eq:extended_gen_diff_polynomial}.
\end{proof}

As with all of the results in this section, Theorem~\ref{thm:eliptic_poly_invariant} is a direct consequence of the theory developed in Section~\ref{sec:main_results} applied to an appropriate class (in this case polynomial) of Lyapunov functions. Its proof is straightforward, but somewhat tedious. It consist of verifying the assumptions of Theorem~\ref{thm:generator} and translating them into lower bounds via Theorems~\ref{thm:invariant} and~\ref{thm:modulated_moments}\ref{thm:modulated_moments_b}. 

\begin{proof}[Proof of Theorem~\ref{thm:eliptic_poly_invariant}]
Pick $\eps\in(0,m_c)$. 
Let $V_\eps:=p_{m_c-\eps}$ be a $C^2(\R^n)$ 
function in~\eqref{eq:power_function}.
By Proposition~\ref{prop:deterministic_eliptic_poly}, we have 
$\cA(1/V_\eps)(x)=\cA p_{\eps-m_c}(x)\leq C_0 |x|^{\eps-m_c+\ell-2}$ for $C_0>0$ and all $x\in\R^n$ with large norm $|x|$. For $r\in[1,\infty)$, define
$\varphi_\eps(1/r)\coloneqq C_0 r^{(\eps-m_c+\ell-2)/(m_c-\eps)}$. Then $\cA(1/V_\eps)\leq \varphi(1/V_\eps)$  holds outside of a large ball centered at the origin, implying condition~\eqref{eq:generator_condition_super} in Theorem~\ref{thm:generator}.
Moreover, since
the function $r\mapsto 1/(r\varphi_\eps(1/r))=r^{(2-\ell)/(m_c-\eps)}/C_0$ is increasing with infinite limit as $r\to\infty$, all assumptions in Theorem~\ref{thm:generator}\ref{generator_a} concerning $\varphi_\eps$ are satisfied.
Define differentiable, increasing and submultiplicative function $\Psi_\eps(r)\coloneqq r^{1+2\eps/(m_c-\eps)}$, $r\in[1,\infty)$, and note (by~\eqref{eq:power_function}) that
$\Psi_\eps\circ V_\eps= p_{m_c+\eps}$.
By Proposition~\ref{prop:deterministic_eliptic_poly}, we have 
$\cA(\Psi_\eps\circ V_\eps)\geq  0$
outside of a large ball centered at the origin,
implying 
conditions in Theorem~\ref{thm:generator}\ref{generator_b}.
Thus, by Theorem~\ref{thm:generator}, $\mathrm{L}$-drift condition~\nameref{sub_drift_conditions} holds
with $(V,\varphi,\Psi)\coloneqq (V_\eps,\varphi_\eps,\Psi_\eps)$
for the diffusion $X$ in SDE~\eqref{eq:elliptic_SDE} and every $\eps\in(0,m_c)$.

The function $L_{\eps,q}$ in~\eqref{eq:def_L_eps_q} (with $q\in(0,1)$) satisfies $L_{\eps,q}(r)= r\varphi(1/r)\Psi(2r/(1-q))(\log \log r)^{\eps}\leq 
r^{1+(\ell-2+2\eps)/(m_c-\eps)}(\log \log r)^{\eps}/C_1$
 for some constant $C_1>0$ and all $r\in[1,\infty)$. Since $V=V_\eps$, 
  by Theorem~\ref{thm:invariant},  there exists $c_\pi'>0$ such that
$$\pi(\{|x|\geq r\})=c_\pi'/L_{\eps,q}(r^{m_c-\eps})\geq c_\pi'C_1 / (r^{\ell-2 + m_c+\eps} (\log \log (r^{m_c-\eps}))^{\eps})\geq c_\pi / r^{\ell-2 + m_c+2\eps}$$
for all $r\in[1,\infty)$ and a sufficiently small constant $c_\pi>0$, implying part~(a) of the theorem.

Theorem~\ref{thm:modulated_moments}\ref{thm:modulated_moments_b} provides a lower bound on the tail of the return time $\tau_D(\delta)$. The inverse function $G_1(t)$ (of the function proportional to $r\mapsto 1/(r\varphi_\eps(1/r))=r^{(2-\ell)/(m_c-\eps)}/C_0$) is proportional to
$t\mapsto t^{(m_c-\eps)/(2-\ell)}$, implying in particular that there exist $t_0,C'>0$ such that 
$\Psi(2G_1(t)/(1-q))\leq C' t^{(m_c+\eps)/(2-\ell)}$ for all $t\in(t_0,\infty)$.
Since $\eps\in(0,m_c)$ can be chosen to be arbitrarily small, part~(b) follows. By Remark~\ref{rem:modulated_every_delta}, following Theorem~\ref{thm:modulated_moments} above, the lower bound holds for all $\delta>0$ since the diffusion $X$ in SDE~\eqref{eq:elliptic_SDE} has full support  at every positive time.
\end{proof}

\begin{proof}[Proof of Theorem~\ref{thm:diff_poly}]
Pick $\eps\in(0,(2-\ell)/3)$. Let $V_\eps = p_{m_c-\eps}$ (where $p_{m_c-\eps}$ is a $C^2(\R^n)$ function in~\eqref{eq:power_function}), $\Psi_{\eps}(r) = r^{1+2\eps/(m_c-\eps)}$ and $\varphi_\eps(1/r) = C_0r^{(\eps-m_c+\ell-2)/(m_c-\eps)}$, for $r\in [1,\infty)$ and some constant $C_0\in(0,\infty)$.
Recall, from the proof of Theorem~\ref{thm:eliptic_poly_invariant}, that the $\mathbf{L}$-drift condition~\nameref{sub_drift_conditions} holds with $(V,\varphi,\Psi) := (V_\eps,\varphi_\eps,\Psi_\eps)$.

Pick $k\in[0,\ell + (2\alpha-\gamma)/\beta)$ and note  $k<2+(2\alpha-\gamma)/\beta =m_c$. Consider the functions $h,f_\star,g:[1,\infty)\to[1,\infty)$ given by $h(r) = r$, $f_\star(r) = r^{k/(m_c-\eps)}$ and $g(r)=h(r)/f_\star(r)=r^{1-k/(m_c-\eps)}$. By Proposition~\ref{prop:deterministic_eliptic_poly}, there exists $C_h'\in(1,\infty)$ such that $\cA(h\circ V_\eps)(x)\leq C_h'$ for all $x\in\R^n$. Thus, Lemma~\ref{lem:bounded_generator} (with $H=h\circ V_\eps$ and $\xi\equiv C_h'$) yields $\E_x[h\circ V_\eps(X_t)]\leq C_h(h\circ V_\eps(x)+t)$ for all $x\in\R^n$, $t\in\RP$ and some $C_h\in(1,\infty)$. 

The function $L_{\eps,q}$ in~\eqref{eq:def_L_eps_q} (with $q\in(0,1)$) satisfies $L_{\eps,q}(r)= r\varphi(1/r)\Psi(2r/(1-q))(\log \log r)^{\eps}\leq 
r^{1+(\ell-2+2\eps)/(m_c-\eps)}(\log \log r)^{\eps}/C\leq r^{1+(\ell-2+3\eps)/(m_c-\eps)}/C'$
 for some constants $C,C'>0$ and all $r\in[1,\infty)$. We define the function $a:[1,\infty)\to\RP$ by
\begin{align*}
\frac{c_{\eps,q} f_\star(g^{-1}(t))}{L_{\eps,q}(g^{-1}(t))}
&\geq\frac{\tilde c_{\eps,q} t^{k/(m_c-\eps-k)}}{t^{(\ell-2+m_c+2\eps)/(m_c-\eps-k)}}
=\tilde c_{\eps,q}t^{-1-(\ell-2+3\eps)/(m_c-k-\eps)} \\
&=\tilde c_{\eps,q}t^{-1+(2-\ell)/(m_c-k)-b(\eps)}=: a(t),
\end{align*}
where $b(\eps) \downarrow 0$ as $\eps\downarrow0$.  Let $A(t) \coloneqq ta(t) = \tilde c_{\eps,q} t^{(2-\ell)/(m_c-k)-b(\eps)}$ for $t\in[1,\infty)$. As $(2-\ell)/(m_c-k)>0$, for all sufficiently small $\eps>0$, we have $\lim_{t\to\infty}A(t)=\infty$.

 Applying Theorem~\ref{thm:f_rate}  with functions $h$, $f_\star$, $g$, $a$ and $A$  defined above yields: for every $x\in\R^n$, there exists a constant $c\in(0,\infty)$ such that $$\|\P_x(X_t\in\cdot)-\pi(\cdot)\|_{f_\star \circ V_\eps}\geq a\circ A^{-1}(2C_h(h\circ V_\eps(x)+t)) \geq c /t^{\eta_\eps},$$ where $\eta_\eps \coloneqq (m_c-k)/((2-\ell)-b(\eps)(m_c-k))-1$. Since $b(\eps)\downarrow 0$ as $\eps\downarrow 0$ and $k<m_c$, we 
get $\eta_\eps\downarrow \alpha_k=m_c/(2-\ell)-1-k/(2-\ell)$ as $\eps\downarrow 0$ (recall the definition of $\alpha_k$ in the statement of Theorem~\ref{thm:diff_poly}). Moreover, since we have $V_\eps(x)\leq 1+|x|^{m_c-\eps}$ for all $x\in\R^n$ and $f_\star(r) = r^{k/(m_c-\eps)}$ is concave on  $[0,\infty)$, it follows that $f_\star \circ V_\eps(x)\leq f_\star(1+|x|^{m_c-\eps})\leq f_\star(1)+f_\star(|x|^{m_c-\eps})= 1+|x|^{k}=f_k(x)$ for all $x\in\R^n$. This implies $\|\cdot\|_{f_{\star\circ V_\eps}}\leq \|\cdot\|_{f_k}$, which concludes the proof.
\end{proof}

Having given full details of the proofs of Theorems~\ref{thm:diff_poly} and~\ref{thm:eliptic_poly_invariant}, in the  proofs of the remaining results of Section~\ref{sec:examples}, we will be less explicit in the applications of our theory.

In the last proof of this section, we will show that the assumptions on the upper bounds of the model parameters, used in~\cite{douc2009subgeometric,Fort2005,veretenn97}, are not sufficient for determining the rate of convergence. We will prove that, under assumption on upper bounds from~\cite{Fort2005}, one-dimensional Langevin diffusions may achieve polynomial ergodicity of \textit{any} order.

\begin{proof}[Proof of the inequalities in~\eqref{eq:oscilating_langevin}]
Fix $\alpha\in(1,\infty)$,
pick arbitrary $k\in(\alpha,\infty)$ and set $b\in(0,\infty)$ such that $\alpha=kb/(1+b)$ holds. Consider $\pi$ satisfying $\pi(x) = |x|^{-k}((1+b)+\sin(k|x|)/|x|)$ for $|x|$ sufficiently large and note that  $\limsup_{|x|\to\infty}x(\log\pi)'(x)= -kb/(1+b) = -\alpha$.

For $\eta\in(0,1]$ define the Lyapunov function $V_\eta:\R\to[1,\infty)$ satisfying $V_\eta(x) = |\int_0^{x}1/\pi(y)^\eta\ud y|$ for all $|x|$ sufficiently large. Choosing $\eta\in(0,1)$ implies the equalities $\cA V_\eta(x) =((\log \pi)'V_\eta' +V_\eta'')(x)= (1-\eta)(\log\pi)'(x)/\pi(x)^\eta$
for all $x$ outside of some compact set,
where $\cA$ is the generator of the process in~\eqref{eq:one_dimension_Langevin}. Since there exist $c,C\in(0,\infty)$ such that $c/|x|^k\leq \pi(x)\leq C/|x|^k$ for all $x\in\R$ outside a compact set, there also exist  constants $c_1,c_2\in(0,\infty)$ satisfying $V_\eta(x) \geq c_1x^{1+\eta k}$ and $(\log\pi)'(x)/\pi(x)^\eta \leq c_2 x^{\eta k-1}$ for all $x$ with large $|x|$.
Thus, $\cA V_\eta (x)\leq c_3V_\eta(x)^{(\eta k-1)/(\eta k +1)}$ for all $x\in\R$ with $|x|$ sufficiently large and some constant $c_3\in(0,\infty)$. By~\cite[Thms~3.2 and~3.4]{douc2009subgeometric},  for each $x\in\R$ there exists  $C_0\in(0,\infty)$ such that $ \|\P_x(X_t\in\cdot)-\pi(\cdot)\|_{\TV}\leq C_0t^{-(\eta k-1)/2}$ for all $t\in[1,\infty)$.

To obtain the matching lower bound, note that the inequality $\cA V_1 =((\log \pi)'V_1' +V_1'') \leq C'$ holds on $\R$
 for some constant $C'\in(0,\infty)$. Thus, Lemma~\ref{lem:bounded_generator} (with $H=V_1$ and $\xi\equiv C'$) yields a constant $C''\in(1,\infty)$, such that $\E_x[V_1(X_t)]\leq C''(V_1(x) + t)$ for all $x\in\R$ and $t\in\RP$. Since $\pi(\{V\geq r\})\geq c_\pi r^{(1-k)/(1+k)}$ for some $c_\pi>0$ and all large $r\in[1,\infty)$, by Lemma~\ref{lem:lower_bound_f_convergence_rate} with $f \equiv 1$, for every $x\in\R$ there exist $c,C\in(0,\infty)$ such that
 $$
 ct^{(1-k)/2} \leq \|\P_x(X_t\in\cdot)-\pi(\cdot)\|_{\TV}\leq Ct^{(1-\eta k)/2} \quad \text{for all $t\in[1,\infty)$}.
 $$
 The proof of~\eqref{eq:oscilating_langevin} is complete since  $\eta\in(0,1)$ was chosen arbitrarily.
\end{proof}

\subsubsection{Stretched exponential tails: proofs for Section~\ref{subsubsec_Subexponential}}
\label{subsubsec:subexp_proofs}
Recall from~\ref{example:ass_diffse} parameters $p\in(0,1)$, $\ell\in[0,2p)$. 
For any  $u\in\R\setminus\{0\}$, consider a function $g_u:\R^n\to(0,\infty)$
in $C^2(\R^n)$, satisfying
\begin{equation}
\label{eq:stretched_exponential_function}
g_u(x)= 
\exp(u|x|^{1-p})\quad\text{ for all $x\in\R^n$ with $|x|$ sufficiently large.}
\end{equation}
We may assume that $g_u$ is a function of $|x|$ for all $x\in\R^n$ and,
if $u>0$, then $g_u\geq 1$ on $\R^n$. , 

The following (deterministic) proposition allows us to construct the functions in~\nameref{sub_drift_conditions}.

\begin{prop}
\label{prop:deterministic_eliptic_subgeom}
Under \ref{example:ass_diffse}, the extended generator $\cA$ in~\eqref{eq:extended_gen_diffusion} 
of the diffusion $X$ in~\eqref{eq:elliptic_SDE}
satisfies 
the following.
\begin{myenumi}
\item[(i)] For some $u_c\in(0,\infty)$ and all $u\in(u_c,\infty)$, there exists a constant $C_u\in(0,\infty)$ such that 
$$
0\leq \cA g_u\leq C_u g_u/(\log g_u)^{(2p-\ell)/(1-p)}\qquad\text{outside of some compact set. }
$$
\item[(ii)] For the function $p_{1}$, defined in~\eqref{eq:power_function} above,  we obtain $\cA (1/p_{1})(x)\leq C_2/p_{1}(x)^{2+p-\ell}$ for some constant $C_2\in(0,\infty)$ and all $x\in\R^n$ with $|x|$ large.
\end{myenumi}
\end{prop}

\begin{proof}
Denote $g(|x|)=g_u(x)$ for $x\in\R^n$.
Note that $\nabla g_u(x)=
g'(|x|) x$ and  
\begin{equation*}
\Hessian(g_u)(x)=\Hessian (g(|x|))=g'(|x|)|x|^{-1}I_n+(g''(|x|)/|x|^{2}-g'(|x|)/|x|^3) x x^\intercal,
\end{equation*} 
where  $I_n\in\R^{n\times n}$ is the identity matrix. By~\ref{example:ass_diffse}, $\limsup_{|x|\to\infty}\Tr(\Sigma(x))/|x|^{\ell}<\infty$ and there exist $0<\alpha_U<\alpha_L$ and $0<\beta_L<\beta_U$ such that
$$
-\alpha_L\leq\langle b(x),x/|x|\rangle/|x|^{\ell-p}\leq -\alpha_U \quad\text{and}\quad \beta_L\leq \langle \Sigma(x)x/|x|,x/|x|\rangle/|x|^\ell\leq \beta_U
$$ for all $x$ outside of some compact set.
By the formula for $\cA$ in~\eqref{eq:extended_gen_diffusion}, we get
\begin{equation}
\label{eq:A_g_u}
(\beta_L u(1-p)/2-\alpha_L+o(1))/|x|^{2p-\ell}\leq \frac{\cA g_u(x)}{u(1-p)g_u(x)} \leq (\beta_U u(1-p)/2-\alpha_U+o(1))/|x|^{2p-\ell},
\end{equation}
as $|x|\to\infty$.
Thus, for $u\in(2\alpha_L/(\beta_L(1-p)),\infty)$ we have $\cA(g_u)(x)\geq 0$ for all $x$ with $|x|$ large. 
Since $|x|=(u^{-1}\log g_u(x))^{1/(1-p)}$ for $x\in\R^n$ with large $|x|$, there exists a constant $C_u>0$, such that the upper bound on $\cA g_u$ in~(i) holds. 

For 
$x\in\R^n$ outside of a compact set, $p_1(x)=|x|$, $\nabla(1/p_{1})(x)=-x/|x|^{3}$ and
$\Hessian(1/p_{1})(x)=-(-3xx^\intercal/|x|^2+I_n)/|x|^{3}$. Thus the representation of $\cA$ in~\eqref{eq:extended_gen_diffusion}
and the \ref{example:ass_diffse} imply
$$
\cA (1/p_{1}) (x) \leq
 (\alpha_L+o(1))/p_1(x)^{2+p-\ell},
$$
for all 
$x\in\R^n$ with sufficiently large $|x|$. This concludes the proof of~(ii).
\end{proof}


\begin{proof}[Proof of Theorem~\ref{thm:example_se_inv_rate}]
Let $V \coloneqq p_1$ be a $C^2(\R^n)$ Lyapunov function in~\eqref{eq:power_function}. By Proposition~\ref{prop:deterministic_eliptic_subgeom}(ii) we have $\cA(1/V)(x)\leq  C/V(x)^{2+p-\ell} = \varphi(1/V(x))$ for some constant $C\in(0,\infty)$ and all $x\in\R^n$ outside of some compact set, where the function $\varphi:(0,1]\to\R$ is given by $\varphi(1/r) \coloneqq C/r^{2+p-\ell}$.
Moreover, by Proposition~\ref{prop:deterministic_eliptic_subgeom}(i), there exists $u_c\in(0,\infty)$, such that for all $u\in(u_c,\infty)$ the function $\Psi_u(r) = \exp(ur^{1-p})$ satisfies $\cA(\Psi_u\circ V)(x) = \cA g_u \geq 0$ for all $x\in\R^n$ with $|x|$ large ($g_u$ is defined in~\eqref{eq:stretched_exponential_function}).
The functions $(V,\varphi,\Psi_u)$ (with any $u>u_c$) defined above, satisfy the $\mathbf{L}$-drift condition~\nameref{sub_drift_conditions} by Theorem~\ref{thm:generator}.
For any $u_0\in(2u,\infty)$, Theorem~\ref{thm:invariant} yields
\begin{equation}
\label{eq:pi_se_bound}
\pi(\{V\geq r\})\geq c\exp(-u_0r^{1-p}) \quad \text{for some constant $c\in(0,\infty)$ and all $r\in[1,\infty)$,}
\end{equation}
implying Theorem~\ref{thm:example_se_inv_rate}(a) (polynomial and logarithmic terms  in~\eqref{eq:main_result_invariant} of Theorem~\ref{thm:invariant} are negligible, compared to the stretched exponential decay of $1/\Psi$, implying inequality~\eqref{eq:pi_se_bound}).
Moreover, an application of Theorem~\ref{thm:modulated_moments}\ref{thm:modulated_moments_b} yields the lower bound on the tail of the return time in Theorem~\ref{thm:example_se_inv_rate}(b). 

Pick $u\in(u_0,\infty)$ (for $u_0$ in~\eqref{eq:pi_se_bound}) and define $h(r) \coloneqq \exp(ur^{1-p})$. By the inequality in~\eqref{eq:pi_se_bound} we have $\pi(\{h\circ V\geq r\})\geq c/r^{u_0/u}$ for some $c\in(0,1)$ and all $r\in[1,\infty)$. Since $h\circ V = g_u$ outside of a compact set, by Proposition~\ref{prop:deterministic_eliptic_subgeom}(i) there exists a constant $C_h'\in(0,\infty)$, such that $\cA(h\circ V)  \leq C_h' h\circ V/(\log h\circ V )^{(2p-\ell)/(1-p)} $ outside of a compact set. Lemma~\ref{lem:bounded_generator} (with $H=h\circ V$, $\xi(r)=C_h'r/(\log r)^{(2p-\ell)/(1-p)}$ and hence $\Xi(u) = \int_1^u\ud s/\xi(s)=(C_h'^{-1}\log u)^{(1+p-\ell)/(1-p)}$) implies 
$$
\E_x[h\circ V(X_t)]\leq \Xi^{-1}(\Xi (h\circ V(x))+t)\leq \exp(C_h( h\circ V(x) +t)^{(1-p)/(1+p-\ell)})\eqqcolon v(x,t)
$$ for all $x\in\R^n$, $t\in\RP$ and some $C_h\in(1,\infty)$. Thus, by Corollary~\ref{cor:rate}, applied with $a(r) = c/r^{u_0/u}$ and $v(x,t)$, we obtain the claimed lower bound in part~(c) of Theorem~\ref{thm:example_se_inv_rate}.
\end{proof}

\subsubsection{Exponential tails: proofs for Section~\ref{subsubsec:Exponential}}
\label{subsubsec:exponential_proofs}

\begin{prop}
\label{prop:deterministic_exp}
Under \ref{example:ass_diffexp}, the extended generator $\cA$ in~\eqref{eq:extended_gen_diffusion} 
of the diffusion $X$ in~\eqref{eq:elliptic_SDE}
satisfies 
the following.
\begin{myenumi}
\item[(i)] For some $u_c\in(0,\infty)$ and all $u\in(u_c,\infty)$, there exists a constant $C_u\in(0,\infty)$ such that 
$$
0\leq \cA \exp(u|x|)\leq C_u \exp(u|x|)\qquad\text{outside of some compact set. }
$$
\item[(ii)] For the function $p_{1}$, defined in~\eqref{eq:power_function} above,  we obtain $\cA (1/p_{1})(x)\leq C'/p_{1}(x)^{2}$ for some constant $C'\in(0,\infty)$ and all $x\in\R^n$ with $|x|$ large.
\end{myenumi}
\end{prop}

\begin{proof}
For any $u\in(0,\infty)$ it holds that $\nabla \exp(u |x|) = u\exp(u |x|) x/|x|$ and $\Hessian(\exp(u |x|)) = u\exp(u |x|)(I_n/|x|+ (u /|x|^2 + 1/|x|^{3})xx^\intercal)$. By~\ref{example:ass_diffexp}, $\limsup_{|x|\to\infty}\Tr(\Sigma(x))<\infty$ and there exist $0<\alpha_U<\alpha_L$ and $0<\beta_L<\beta_U$ such that
$$
-\alpha_L\leq\langle b(x),x/|x|\rangle\leq -\alpha_U \quad\text{and}\quad \beta_L\leq \langle \Sigma(x)x/|x|,x/|x|\rangle/\leq \beta_U
$$ for all $x$ outside of some compact set.
By the formula for $\cA$ in~\eqref{eq:extended_gen_diffusion}, we get
\begin{equation*}
(\beta_L u/2-\alpha_L+o(1))\leq \frac{\cA \exp(u|x|)}{u\exp(u|x|)} \leq (\beta_U u/2-\alpha_U+o(1))\qquad\text{as $|x|\to\infty$.}
\end{equation*}
Thus, for $u\in(2\alpha_L/\beta_L,\infty)$ we have $\cA(\exp(u|x|))\geq 0$ for all $x$ with $|x|$ large, implying~(i).

For all 
$x\in\R^n$ with sufficiently large $|x|$, we have $p_1(x)=|x|$, $\nabla (1/p_{1})(x)=-|x|^{-3}x$ and
 $\Hessian(1/p_{1})(x)=-|x|^{-3}(-3xx^\intercal/|x|^2+I_n)$. Thus,
$$
\cA (1/p_{1}) (x) =
 (\alpha_L+o(1))/p_1(x)^{2} \quad \text{as $|x|\to\infty$.\qedhere}
$$
\end{proof}

\begin{proof}[Proof of Theorem~\ref{thm:example_exp}]
 Let $V =  p_1$, where $p_1$ is a $C^2(\R^n)$ function in~\eqref{eq:power_function}. Then, by Proposition~\ref{prop:deterministic_exp} we have $\cA(1/V)(x) = \cA(1/p_{1})(x)\leq \varphi(1/p_{1}(x))$, for all $x\in\R^n$ outside of some compact set centered at the origin, where $\varphi(1/r) = C/r^2$ for all $r\in[1,\infty)$ and some constant $C\in(\alpha,\infty)$. Moreover, by Proposition~\ref{prop:deterministic_exp}, there exists $u_c\in(0,\infty)$ such that for any $u_0\in(u_c,\infty)$ the function $\Psi(r)\coloneqq \exp(u_0r)$ satisfies $\cA(\Psi \circ p_1) = \cA \exp(u_0|x|)\geq 0$ for all $x\in\R^n$ outside of some compact set. Thus, by Theorem~\ref{thm:generator}, the conditions of~\nameref{sub_drift_conditions} are satisfied with the aforementioned functions  $(V,\varphi,\Psi)$. 
By Theorem~\ref{thm:invariant} we obtain a lower bound on the tail of the invariant measure, implying Theorem~\ref{thm:example_exp}(a).
Moreover, an application of Theorem~\ref{thm:modulated_moments}\ref{thm:modulated_moments_b} yields the lower bound on the tail of the return time in Theorem~\ref{thm:example_exp}(b). 

The function $L_{\eps,q}$ in~\eqref{eq:def_L_eps_q} (with $q,\eps\in(0,1)$) satisfies $L_{\eps,q}(r)= r\varphi(1/r)\Psi(2r/(1-q))(\log \log r)^{\eps}\leq 
\exp(2u_0r/(1-q))/C_1$
 for some constant $C_1>0$ and all $r\in[1,\infty)$. 
Pick $u\in(2u_0/(1-q),\infty)$ and denote $h(r) \coloneqq \exp(ur^{1-p})$. Define the function $a:[1,\infty)\to\RP$ by $a(r)\coloneqq c_{\eps,q}' c/r^{2u_0/(1-q)/u}\leq c_{\eps,q}/ L_{\eps,q}(h^{-1}(r))$.  Since $h\circ V = \exp(u|x|)$ outside of a compact set, by Proposition~\ref{prop:deterministic_eliptic_subgeom}(i) there exists a constant $C_h'\in(0,\infty)$, such that $\cA(h\circ V)  \leq C_h' h\circ V $ outside of a compact set. Lemma~\ref{lem:bounded_generator} (with $H=h\circ V$, $\xi(r)=C_h'r$ and hence $\Xi(u) = \int_1^u\ud s/\xi(s)=(C_h'^{-1}\log u)$) implies 
$$
\E_x[h\circ V(X_t)]\leq \Xi^{-1}(\Xi (h\circ V(x))+t)\leq \exp(C_h( h\circ V(x) +t))\eqqcolon v(x,t)
$$ for all $x\in\R^n$, $t\in\RP$ and some $C_h\in(1,\infty)$. Thus, by Corollary~\ref{cor:rate}, applied with $a$ and $v(x,t)$, we obtain the claimed lower bound in part~(c) of Theorem~\ref{thm:example_exp}.
\end{proof}

\subsection{L\'evy-driven SDE: proofs for Section~\ref{sec:levy}}
\label{subsec:levy_proofs}
Let $X$ be  the solution of SDE~\eqref{eq:levy_OU} with a bounded Lipschitz dispersion coefficient $\sigma$, driven by a pure-jump L\'evy process $L$ with L\'evy measure $\nu$. By It\^o's formula~\cite[Thm~II.36]{MR2273672} applied to $g(X)$, where $g\in C^1(\R)$ and $\int_{\R\setminus(-1,1)}|g(y)|\nu(\ud y)<\infty$, 
the extended generator (see Section~\ref{subsec:generator} above for definition) of the process $X$ takes the form
\begin{equation}
\label{eq:generator_Levy}
\cA g(x) =-\mu xg'(x)+ \int_{\R} (g(x+\sigma(x)y)-g(x) - g'(x)\sigma(x)y\1{y\in[-1,1]})\nu(\ud y),\qquad\text{$x\in\R$.}
\end{equation}
For any $m\in[1,\infty)$, consider a non-decreasing function $g_m:\R\to [1,\infty)$ in $C^2(\R)$ satisfying
\begin{equation}
\label{eq:log_function}
g_m\equiv 1 \text{ on $(-\infty,1]$}\quad \&\quad g_m(x) = (\log x)^m \quad \text{for all $x\in\RP$ sufficiently large.}
\end{equation}
Note that, in particular, the derivatives $g',g''$ are globally bounded.

\begin{prop}
\label{prop:OU_deterministic}
Under \ref{example:ass_Levy}, extended generator~\eqref{eq:generator_Levy} 
of the process $X$ in~\eqref{eq:levy_OU}
satisfies 
the following  inequalities.
\begin{enumerate}[label=(\alph*)]
\item \label{prop:deterministic_levy_a}For any $m\in[1,m_c)$ there exists a constant $C_1\in(0,\infty)$ such that
$$
\cA g_m(x) \leq C_1 \quad \text{for all $x\in\R$.}
$$
\item\label{prop:deterministic_levy_b}
There exist constants $C_2,b,\ell_0\in(0,\infty)$, such that the following holds
$$
\cA (1/g_1)(x) \leq  C_2/g_1(x)^2 + b\1{g_1(x)\leq \ell_0} \quad \text{for all $x\in\R$.}
$$ 
\end{enumerate}
\end{prop}

\begin{proof}
For any $g\in C^2(\R)$ and  $C_1' \coloneqq \sup_{x\in\R}\sigma(x)^2\int_{[-1,1]}y^2 \nu(\ud y)\in[0,\infty)$,
Lagrange's theorem yields 
\begin{equation}
\label{eq:levy_martingale_bound}
\int_{[-1,1]} (g(x+\sigma(x)y)-g(x) - g'(x)\sigma(x)y)\nu(\ud y) \leq C_1' \sup_{u\in[x-\sigma(x),x+\sigma(x)]}|g''(u)| \quad \text{for all $x\in\R$.}
\end{equation}

\noindent \underline{Part~(a)}.
Let $m\in[1,m_c)$. 
For all large $x\in\RP$ 
we have  $g_m'(x)=m(\log x)^{m-1}/x$, implying  $g_m'(x+u)\leq g_m'(1+u)$ for all $u>0$. Thus, by Tonelli's theorem, for all large $x\in\RP$ we have
\begin{align}
\nonumber
\int_{[1,\infty)} (g_m(x+\sigma(x)y)-g_m(x)) \nu(\ud y) &= \int_{[1,\infty)}\nu(\ud y)\int_0^{\sigma(x)y} g_m'(x+u)\ud u\\
& \nonumber= \int_0^\infty \nu([\max\{1,u/\sigma(x)\},\infty))g_m'(x+u)\ud u \\
& \label{eq:levy_inequality} \leq \int_0^\infty \nu([\max\{1,u/\sup_{y\in\R}\sigma(y)\},\infty))g_m'(1+u)\ud u<\infty,
\end{align}
where the final inequality holds by~\ref{example:ass_Levy} since $m<m_c$ and $\sigma$ is bounded on $\R$.

Since $g_m$ is non-decreasing and $\sup_{u\in\R}|g_m''(u)|<\infty$ (by definition~\eqref{eq:log_function}), inequalities~\eqref{eq:levy_martingale_bound} and~\eqref{eq:levy_inequality} imply there exists $x_0\in(0,\infty)$, such that,  for all $x\in[x_0,\infty)$, the representation of $\cA$ in~\eqref{eq:generator_Levy} yields
\begin{align}
\label{eq:upper_bound_generagor_jumps}
\cA g_m(x) &\leq -\mu  m (\log x)^{m-1}+C_1'\sup_{u\in\R}|g_m''(u)|+\int_{[1,\infty)} (g_m(x+\sigma(x)y)-g_m(x)) \nu(\ud y) \\
&\leq C_1'\sup_{u\in\R}|g_m''(u)|+\int_0^\infty \nu([\max\{1,u/\sup_{y\in\R}\sigma(y)\},\infty))g_m'(1+u)\ud u<\infty.
\nonumber
\end{align}
Let $x\in(-\infty,x_0)$. Then  
$\int_{[1,\infty)}g_m(x+\sigma(x)y)\nu(\ud y)\leq \int_{[1,\infty)}g_m(x_0+\sup_{u\in\R}\sigma(u)y)\nu(\ud y)<\infty$ since $g_m$ is non-decreasing and integrable with respect to $\nu$ by~\ref{example:ass_Levy}. Hence
the inequality in~\eqref{eq:levy_martingale_bound} and the global boundedness of $g_m''$ imply
 $$
 \cA g_m(x)\leq C_1'\sup_{u\in\R} |g_m''(u)| +\int_{[1,\infty)} g_m(x_0+\sup_{y\in\R}\sigma(y)u)\nu(\ud u) <\infty\quad\text{for all $x\in(-\infty, x_0)$.}
 $$ 

\noindent \underline{Part~(b)}. By definition of $g_1$ in~\eqref{eq:log_function}, there exists large $x_0\in\RP$, such that 
\begin{align}
\nonumber
\int_{[-\frac{x}{2\sigma(x)},-1]} ((1/g_1) (x+\sigma(x)y)-(1/g_1)(x))\nu(\ud y)&= \int_{[-\frac{x}{2\sigma(x)},-1]} \frac{\log x-\log(x+\sigma(x)y)}{\log(x+\sigma(x)y)\log x}\nu(\ud y) \\
\nonumber
&\leq \nu((-\infty,-1])\log 2/(\log (x/2)\log (x))\\
& \leq 
\nu((-\infty,-1])\log 4/g_1(x)^2\qquad\text{for all  $x\geq x_0$.}
\label{eq:jumps_bound_(b)}
\end{align}
By enlarging $x_0$ if necessary, we may assume $- x(1/g_1)'(x) = 1/g_1(x)^2$ and  $g_1''(x)\leq 1/g_1(x)^2$ for all $x\geq x_0$.
By~\eqref{eq:jumps_bound_(b)}, the bound on small jumps in~\eqref{eq:levy_martingale_bound} and the representation of
$\cA$ in~\eqref{eq:generator_Levy}, we get
\begin{align*}
\cA (1/g_1)(x) &\leq C_1' \sup_{u\in[x-\sigma(x),x+\sigma(x)]}|(1/g_1)''(u)| - \mu x(1/g_1)'(x) + \int_{(-\infty,-1]}\hspace{-5pt}\frac{g_1(x)-g_1(x+\sigma(x)y)}{g_1(x)g_1(x+\sigma(x)y)}\nu(\ud y) \\
&\leq C_2'/g_1(x)^2 + \nu((-\infty,-\frac{x}{2\sigma(x)})) +\int_{[-\frac{x}{2\sigma(x)},-1]} ((1/g_1) (x+\sigma(x)y)-(1/g_1)(x))\nu(\ud y)\\
&\leq C_2/g_1(x)^2\qquad\text{for all $x\geq x_0$ and some constants $C_2',C_2\in(0,\infty)$.}
\end{align*}
Jumps in $[1,\infty)$ can be disregarded since $1/g_1$ is non-increasing and the third inequality follows from \ref{example:ass_Levy} on the negative tail of $\nu$, bound~\eqref{eq:jumps_bound_(b)} and the fact that $\sigma$ is bounded.

By~\eqref{eq:log_function}, $1/g_1$ is non-increasing, $x(1/g_1)'(x)$ is bounded from below for $x\in \R$ and $1/g_1\leq 1$. Thus, since $\mu\in(0,\infty)$, by~\eqref{eq:levy_martingale_bound} we have
$$
\cA (1/g_1)(x) \leq C_1' \sup_{u\in\R} (1/g_1)''(u) - \inf_{u\in\R}\mu u (1/g_1)'(u) + \nu((-\infty,-1])  \eqqcolon b< \infty \quad \text{ for all $x\in \R$,}
$$ implying part (b) of the proposition with $\ell_0\coloneqq g_1(x_0)$.
\end{proof}

In this section we will work with the Lyapunov function $V \coloneqq g_1$. Recall the definitions  $S_{(\ell)} = \inf\{t\geq 0: V(X_t)<\ell\}$ and $T^{(r)} = \inf\{t\geq 0: V(X_t)>r\}$, for $r,\ell\in(1,\infty)$.

\begin{prop}
\label{prop:OU_sub_condition}
Let \ref{example:ass_Levy} hold. For some $\ell_0\in(1,\infty)$ and each $\ell\in(\ell_0,\infty)$ there exists $C_\ell\in(0,\infty)$ such that 
\begin{equation*}
\P_x(T^{(r)}<S_{(\ell)}) \geq 
C_\ell/r^{m_c} \quad \text{for all $r\in(\ell+1,\infty)$ and $x\in\{\ell+1\leq V<r\}$.}
\end{equation*}
\end{prop}

The proof of this proposition is based on a simple idea, which we first explain informally.
The process $X$ 
satisfies SDE~\eqref{eq:levy_OU}, 
\begin{equation}
\label{eq:levy_decomposition}
X_t = x -\mu\int_0^t X_s\ud s + \int_0^t \sigma(X_{s-})\ud L_s^{(-)} + \int_0^t \sigma(X_{s-})\ud L_s^{(M)} + \int_0^t \sigma(X_{s-}) \ud L_s^{(+)},\quad \text{for all $t\in\RP$,}
\end{equation}
where 
the driving pure-jump L\'evy process $L= L^{(-)}+ L^{(M)}+ L^{(+)}$ is decomposed into a sum of independent pure-jump L\'evy processes $L_t^{(-)}$, $L_t^{(M)}$ and $L_t^{(+)}$ with L\'evy measures $\nu^{(-)}(\cdot) = \nu(\cdot \cap (-\infty,1])$, $\nu^{(M)}(\cdot) =\nu(\cdot \cap (-1,1))$ and $\nu^{(+)}(\cdot) =\nu(\cdot \cap [1,\infty))$, respectively.
For any $t\in\RP$, the process $X$ has
no negative jumps in $(-\infty,-1]$ on the event   
\begin{equation}
\label{eq:levy_events}
A_t \coloneqq \{L_t^{(-)} = 0\}\cap\left\{\sup_{0\leq s \leq t} \left|\int_0^s \sigma(X_{u-})\ud L_u^{(M)}\right| \leq 1/8\right\}.
\end{equation}
On $A_t$, it is thus necessary for $X$, started at $x\in \{V\geq \ell+1\}$, to accumulate sufficient negative drift  in order to return to $\{V\leq \ell\}$ before time $t$. Since the drift of the process $X$ is bounded on $\{\ell \leq V\leq \ell+1\}$, we will prove $\P_x(\{S_{(\ell)}\leq t\} \cap A_{t}) = 0$ for all sufficiently small $t>0$, implying
$$\{T^{(r)}<S_{(\ell)}\} \supset \{T^{(r)}<S_{(\ell)}\}\cap A_t \supset \{T^{(r)}<t\}\cap A_{t} \supset \left\{\sigma_-\sup_{0\leq s\leq t} (L_s-L_{s-})\geq V^{-1}(r)\right\}\cap A_{t}$$ 
$\P_x$-a.s., where $\sigma_-\coloneqq \inf_{u\in\R}\sigma(u)>0$. Evaluating the probability of the smallest event in the last display will complete the proof.

\begin{proof}[Proof of Proposition~\ref{prop:OU_sub_condition}]
Recall that $X$ follows~\eqref{eq:levy_decomposition} and $V = g_1$, where $g_1$ is a $C^2(\R)$ function in~\eqref{eq:log_function}. Choose $\ell_0\in(1,\infty)$ such that $V^{-1}(y) = \exp(y)$ for all $y\in[\ell_0,\infty)$. Note that for all $\ell,r\in[\ell_0,\infty)$, we have  $S_{(\ell)} = \inf\{t>0: X_t < \exp(\ell)\}$ and $T^{(r)} = \inf\{t>0: X_t>\exp(r)\}$.  By possibly increasing $\ell_0$,  we  may (and do) assume that $\nu(\exp(\ell_0)/\sigma_-,\infty)<2\mu \exp(\ell_0)$ holds.

Fix $\ell\geq\ell_0$ and define the upcrossing and downcrossing times of the interval $[\exp(\ell+1),\exp(\ell+2)]$ as follows: $\overline{\theta}_1 \coloneqq 0$ and for $k\in\N$, $$\underline{\theta}_k \coloneqq \inf\{t>\overline{\theta}_k:X_t<\exp(\ell + 1)\}\quad\text{and}\quad\overline{\theta}_{k+1} \coloneqq \{t>\underline{\theta}_k: X_t>\exp(\ell +2)\}.$$ Thus we have $0=\overline{\theta}_1\leq \underline{\theta}_1\leq\dots\leq \overline{\theta}_{k}\leq \underline{\theta}_{k}\leq\overline{\theta}_{k+1}\leq\dots$. 
The equality $\overline\theta_k=\underline\theta_k$
may occur if the process $X$ jumps downwards over the entire interval $[\exp(\ell+1),\exp(\ell+2)]$.

\noindent\textbf{Claim 1.} For all $x\in\{V\geq \ell+1\}$ 
and $t\in(0,1/(2\mu\exp(\ell+2)))$ 
we have $\P_x(\{S_{(\ell)}\leq t\} \cap A_{t}) = 0$.

\noindent \textit{Proof of Claim 1.}
Pick $x\in\{V\geq \ell+1\}$. Since $\{\overline{\theta}_k \leq S_{(\ell)}<\underline{\theta}_{k}\}= \emptyset$ for all  $k\in\N$, 
the following  holds 
\begin{equation}
\label{eq:union_representation}
\{S_{(\ell)}\leq t\}\cap A_{t} = \cup_{k = 1}^{\infty} \{\underline{\theta}_k\leq S_{(\ell)}\leq t\wedge \overline{\theta}_{k+1}\}\cap A_{t} \qquad \text{$\P_x$-a.s.}
\end{equation}
It is thus sufficient to show that  $\P_x(\{\underline{\theta}_k\leq S_{(\ell)}\leq t\wedge \overline{\theta}_{k+1}\}\cap A_{t})=0$ for every $k\in\N$.
By~\eqref{eq:levy_events}, on $A_t$,  the negative jumps of  $X$ can only come from the L\'evy process $L^{(M)}$. The modulus of the negative jumps of $\int_0^\cdot \sigma(X_{s-})\ud L^{(M)}_s$ on the event $A_t$ is by definition~\eqref{eq:levy_events} bounded above by $1/4$, implying $X_{\underline{\theta}_k}\geq \exp(\ell+1)-1/4$.
Clearly, for $u\in[\underline{\theta}_k,\overline{\theta}_{k+1})$, we have $X_u\leq \exp(\ell+2)$.  Hence, by~\eqref{eq:levy_decomposition}, on the event $A_{t}$  we obtain
\begin{align*}
\inf_{\underline{\theta}_k\leq s\leq t\wedge \overline{\theta}_{k+1}}\hspace{-8pt}X_{s} &=
 X_{\underline{\theta}_k}+\inf_{s\in[\underline{\theta}_k, t\wedge \overline{\theta}_{k+1}]} \left\{-\mu \int_{\underline{\theta}_k}^{s} X_u\ud u + \int_{\underline{\theta}_k}^s \sigma(X_{u-})\ud L_{u}^{(M)} + \int_{\underline{\theta}_k}^s \sigma(X_{u-})\ud L_{u}^{(+)}\right\}\\
&\geq \exp(\ell+1) -1/4 - \mu\exp(\ell+2)t-1/4>\exp(\ell+1)-1>\exp(\ell),
\end{align*} 
where the second inequality follows from
 $t< 1/(2\mu\exp(\ell+2))$.
Hence, on the event $A_t$, 
the process $X$ cannot go below $\exp(\ell)$ 
during the time interval $[\underline{\theta}_k, t\wedge \overline{\theta}_{k+1}]$ for any $k\in\N$, 
implying $\P_x(\{\underline{\theta}_k\leq S_{(\ell)}\leq t\wedge \overline{\theta}_{k+1}\}\cap A_{t})=0$ for all $k\in\N$.
By~\eqref{eq:union_representation}, the claim follows.

To conclude the proof of the proposition, pick $r\in(\ell+1,\infty)$ and $x\in\{\ell+1\leq V<r\}$. Note that by Claim 1 above we have
$\{S_{(\ell)}>T^{(r)}\}\supset \{S_{(\ell)}>t\}\cap \{T^{(r)}<t\}\cap A_{t} =  \{T^{(r)}<t\}\cap A_{t}$ $\P_x$-a.s. for any $t\in(0,1/(2\mu\exp(\ell+2)))$.
Since the positive jumps of $L$ greater than one can only come from $L^{(+)}$ in~\eqref{eq:levy_decomposition}, we have 
\begin{align*}
  \{T^{(r)}<t\}\cap A_{t}  \supset  \left\{\sup_{0\leq s\leq t} (L_s-L_{s-})\geq \frac{V^{-1}(r)}{\sigma_-}\right\}\cap A_{t}=\left\{\sigma_-\sup_{0\leq s\leq t} (L_s^{(+)}-L_{s-}^{(+)})\geq V^{-1}(r)\right\}\cap A_{t}
\end{align*} 
(recall $V^{-1}(r)=\exp(r)>\exp(\ell_0)>1$). For some $t,c\in(0,\infty)$,
Claim~2 below yields
\begin{align*}
\P_x(T^{(r)}<S_{(\ell)})&\geq \P_x(\{T^{(r)}<t\}\cap A_{t}) \geq  \P_x\left(\{\sup_{0\leq s\leq t} (L_s^{(+)}-L_{s-}^{(+)})\geq \exp(r)/\sigma_-\}\cap A_t\right)
\\&\geq c\left(1-\exp(-t\nu(\exp(r)/\sigma_-,\infty))\right) > 
c\nu(\exp(r)/\sigma_-,\infty)t/2,
\end{align*}
where the constant $c\in(0,\infty)$ is such that the third inequality holds uniformly in $r\in(\ell+1,\infty)$ and $x\in\{V\geq \ell+1\}$.
Recall 
$$t\nu(\exp(r)/\sigma_-,\infty))<t\nu(\exp(\ell_0)/\sigma_-,\infty))<\nu(\exp(\ell_0)/\sigma_-,\infty))/(2\mu\exp(\ell_0+2))<1$$ for the last inequality.
Proposition~\ref{prop:OU_sub_condition} now follows from the lower bound in \ref{example:ass_Levy}.

\noindent\textbf{Claim 2.} 
There exists  $c\in(0,\infty)$ such that $$\P_x(\{\sup_{0\leq s\leq t} (L_s^{(+)}-L_{s-}^{(+)})\geq \exp(r)/\sigma_-\}\cap A_t)>c(1-\exp(-t\nu(\exp(r)/\sigma_-,\infty)))$$ for all $x\in\{V\geq \ell+1\}$, $r\in (\ell+1,\infty)$ and  all sufficiently small $t\in(0,\infty)$. 

\noindent \textit{Proof of Claim 2.} The processes $L^{(-)}$ and $L^{(+)}$ are independent with  $\P_x(L_t^{(-)}=0)=e^{-t\nu((-\infty,-1])}$ and $\P_x(\sup_{0\leq s\leq t} (L_s^{(+)}-L_{s-}^{(+)})\geq \exp(r)/\sigma_-)=1-\exp(-t\nu(\exp(r)/\sigma_-,\infty))$ for all $t\in\RP$. For any $t>0$, let  $\mathcal{B}_t$ 
be $\sigma$-algebra generated by $(L^{(+)}_s,L^{(-)}_s)_{s\in[0,t]}$. We have to show that  there exists a constant $\tilde c\in(0,\infty)$ such that $\P_x(\sup_{0\leq s \leq t} \left|\int_0^s \sigma(X_{u-})\ud L_u^{(M)}\right| \leq 1/8\big\vert \mathcal{B}_t)>\tilde c$, for all small $t\in(0,\infty)$ and $x\in\{V\geq \ell+1\}$.
By Markov's inequality, it is sufficient to identify a constant $C\in(0,\infty)$, such that $\E_x[\sup_{0\leq s \leq t}|\int_0^s \sigma(X_{u-})\ud L_u^{(M)}|^2\big\vert\mathcal{B}_t]\leq Ct$ holds for all $t\in(0,\infty)$ and $x\in\{V\geq \ell+1\}$.

For any $y\in \R$, let the process $(Y_t^y)_{t\in\RP}$ be the unique strong solution\cite[Thm~V.6]{MR2273672} of the SDE
\begin{equation}
\label{eq:levy_decomposition2}
 Y_t^y = y -\mu\int_0^t Y_s^y\ud s + \int_0^t \sigma(Y_{s-}^y)\ud L_s^{(M)},\quad \text{for all $t\in\RP$.}
\end{equation}
For any pair of deterministic sequences  
$S=(s_k)_{k\in\N}$ and  $J=(j_k)_{k\in\N}$, such that $0<s_k\uparrow\infty$ and $j_k\in\R\setminus\{0\}$, define the process $Z$ recursively as follows:
\[
Z^{S,V}_s\coloneqq Y_s^x \1{s\in[0,s_1)}+\sum_{k\in\N}\1{s\in[s_k,s_{k+1})} Y^{Z^{S,V}_{s_k}+j_k}_{s-s_k}.
\]

Note that the conditional law $(\int_0^s \sigma(X_{u-})\ud L_u^{(M)})_{s\in[0,t]}$, given the $\sigma$-algebra $\mathcal{B}_t$,
equals the law of $(\int_0^s\sigma(Z^{S,V}_{u-})\ud L_u^{(M)})_{s\in[0,t]}$ for 
the sequences $S$ and $V$ equal to the jump times and sizes of the components of the L\'evy process
$(L^{(+)}_s,L^{(-)}_s)_{s\in[0,t]}$. This is because the L\'evy-driven SDE in~\eqref{eq:levy_decomposition2}
has a unique strong solution. Since $\sigma$ is bounded,~\cite[Thm~V.66]{MR2273672} yields a constant $C\in(0,\infty)$ such that for all $x\in\R$ and any pair of deterministic sequences $S$ and $J$ satisfying the above conditions, we have
$
\E[\sup_{0\leq s \leq t}|\int_0^s\sigma(Z^{S,V}_{u-})\ud L_u^{(M)}|^2]\leq Ct$ for all $t\in\RP$. This implies that
$$
\E_x\left[\sup_{0\leq s \leq t}\left|\int_0^s \sigma(X_{u-})\ud L_u^{(M)}\right|^2\Big\vert \mathcal{B}_t\right] \leq Ct\quad\text{for all $t\in\RP$ and $x\in\R$,}
$$
which concludes the proof of Claim~2.
\end{proof}

\begin{rem}
\label{rem:levy_integrability}
Under~\ref{example:ass_Levy}, the process $\Psi\circ V(X)=(\log X)^{m_c}$ is not a submartingale as its marginals are not integrable. As discussed in Remark~\ref{rem:why_prob_L(ii)} above, this makes a submartingale argument, similar to the one in the proof of Lemma~\ref{lem:assumption_submart_exit_prob}, to deduce a lower bound on the probability $\P_x(T^{(r)}<S_{(\ell)})$ infeasible. A slight perturbation $(\log X)^{m_c-\eps}$ (for a small $\eps>0$) makes the process integrable under~\ref{example:ass_Levy}. However, $\Psi(r)=r^{m_c-\eps}$
cannot be used because the process  $(\log X)^{m_c-\eps}$
is a supermartingale for all small $\eps>0$ by~\eqref{eq:levy_inequality}--\eqref{eq:upper_bound_generagor_jumps} in the proof of Proposition~\ref{prop:OU_deterministic}\ref{prop:deterministic_levy_a}.
\end{rem}

\begin{proof}[Proof of Theorem~\ref{thm:levy_OU}]
The process $X$ is a Feller process by~\cite[Thm~1.1]{Kuhn18}, since we can represent it as a solution of the SDE $\ud X_t = x + \int_0^t \tilde \sigma(X_{s-})\ud \tilde L_s$ for the L\'evy process $\tilde L_t=(t,L_t)$ and a covariance matrix $\tilde \sigma(x)=(-\mu x,\sigma(x))^\intercal$ (note that the condition on the L\'evy measure of~\cite[Thm~1.1]{Kuhn18} holds since $\sigma$ is bounded). Moreover, irreducibility of $X$ follows from the irreducibility of the L\'evy driver $L$ in~\eqref{eq:levy_OU}. Furthermore, Harris recurrence and ergodicity follow by applying~\cite[Thm~3.2]{douc2009subgeometric} with a Lyapunov function $g_2$ and using an argument analogous to that in~\cite[Sec.~3.1]{Fort2005}.

Recall that we are working with the Lyapunov function $V =g_1$, where $g_1$ is a $C^2(\R)$ function in~\eqref{eq:log_function}. By Proposition~\ref{prop:OU_deterministic}\ref{prop:deterministic_levy_b}, there exists a constant $C_2\in(0,\infty)$ such that the function  $\varphi(1/r) \coloneqq C_2/r^{2}$, $r\in[1,\infty)$, satisfies the assumptions of
Theorem~\ref{thm:generator}\ref{generator_a}, which in turn implies  condition~\nameref{sub_drift_conditions}\ref{sub_drift_conditions(i)}. Moreover, by Proposition~\ref{prop:OU_sub_condition} the function $\Psi(r) \coloneqq r^{m_c}$, $r\in[1,\infty)$, satisfies condition~\nameref{sub_drift_conditions}\ref{sub_drift_conditions(ii)}. Note that under~\nameref{sub_drift_conditions} the function $L_{\eps,q}$ in~\eqref{eq:def_L_eps_q} (with arbitrary $q, \eps\in(0,1)$) is bounded above by
$$L_{\eps,q}(r)= r\varphi(1/r)\Psi(2r/(1-q))(\log \log r)^{\eps}\leq 
C_1 r^{m_c-1}(\log \log r)^{\eps}\qquad \text{for all $r\in[1,\infty)$}$$
and some constant $C_1\in(0,\infty)$. By Theorem~\ref{thm:invariant}, for any $\eps>0$, we obtain the lower bound $\pi(\{V\geq r\})\geq c_{\eps,q}/r^{m_c-1+\eps}$, and hence $\pi([r,\infty))\geq c_{\eps,q}'/(\log r)^{m_c-1+\eps}$, for  all $r\in[1,\infty)$ and some constants $c_{\eps,q}, c_{\eps,q}'\in(0,1)$, implying part~(a) of the theorem. Moreover, applying Theorem~\ref{thm:modulated_moments}\ref{thm:modulated_moments_b} yields part~(b) of the theorem.

Pick $\eps\in(0,1/2)$ and consider the function $h(r) \coloneqq r^{m_c-\eps}$.  Proposition~\ref{prop:OU_deterministic}\ref{prop:deterministic_levy_a} implies the inequality $\cA(h\circ V)\leq C_h'$ on $\R$ for some $C_h'\in(1,\infty)$, and Lemma~\ref{lem:bounded_generator} (with $H=h\circ V$ and $\xi\equiv C_h'$) yields a constant $C_h\in(1,\infty)$ such that  $\E_x[h\circ V(X_t)]\leq C_h(h(V(x)) + t)$ holds for all $x\in\R$ and $t\in\RP$. Thus we may apply Corollary~\ref{cor:rate} with functions $h$ and $a(r) \coloneqq  C_a r^{(1-m_c-\eps)/(m_c-\eps)} \leq c_{\eps,q}/L_{\eps,q}(h^{-1}(r))$, $r\in[1,\infty)$, and some $C_a\in(0,1)$. For each $x\in\R$ we obtain a constant $c_\TV\in(0,\infty)$, such that the lower bound $\|\P_x(X_t\in\cdot)-\pi(\cdot)\|_{\TV}\geq c_\TV t^{(1-m_c-\eps)/(1-2\eps)}$ holds for all $t\in[1,\infty)$.
\end{proof}

\subsection{Stochastic damping Hamiltonian system: proofs for Section~\ref{subsec:hamiltonian}}
\label{subsec:hamiltonian_proofs}
Let $X=(Z,Y)$ be the hypoelliptic diffusion satisfying the stochastic damping Hamiltonian system in~\eqref{eq:Hamiltonian_damping_onedim}.
Recall that by \ref{example:ass_Hamilton} we have $\inf_{z\in\R} U(z) \eqqcolon -b>-\infty$ and constants $\sigma, c\in(0,\infty) $  in~\eqref{eq:Hamiltonian_damping_onedim} satisfy $ac/\sigma^2>1/2$.
Following~\cite{Wu2001},   for $\eps\in(0,1)$ and $k\in(0,\infty)$, define a twice differentiable function 
\begin{equation}
\label{eq:hamiltonian_function_g}
g_{k,\eps}(z,y) \coloneqq (y^2/2+U(z) + c(1-\eps)(zy+cz^2/2) + b+1)^k\geq 1\qquad\text{for all $(z,y)\in\R^2$.}
\end{equation} 

By It\^o's formula applied to $g(Z,Y)$, the extended generator (see Section~\ref{subsec:generator} above for definition) of the Hamiltonian system $(Z,Y)$ takes the following form: 
\begin{align}
\label{eq:genrator_hamilton}
\cA g(z,y) = \frac{1}{2}\sigma^2\partial_y^2 g(z,y) + y\partial_z g(z,y)-(cy+U'(z))\partial_y g(z,y)\qquad\text{for any  $g\in C^2(\R^2)$.}
\end{align}

We now apply  the generator $\cA$ to the function $g_{k,\eps}$.
\begin{prop}
\label{prop:deterministic_hamiltonian}
Under \ref{example:ass_Hamilton}, extended generator $\cA$  of the Hamiltonian system in~\eqref{eq:Hamiltonian_damping_onedim} satisfies the following: 
 if $\eps\in(0,1/2)$ and $k \coloneqq 1/2+ac(1-2\eps)/\sigma^2$, the constant $C \coloneqq k\eps a c/2>0$ satisfies   $\cA g_{k,\eps}(z,y)\leq -C (g_{k,\eps}(z,y))^{(k-1)/k}$ for all $(z,y)$ outside of some large compact set. Moreover, $\cA g_{k,\eps}\leq C'$ on $\R^2$ for some positive constant $C'>0$.
\end{prop}

\begin{proof}
Fix $\eps\in(0,1/2)$. Applying the generator $\cA$ in~\eqref{eq:genrator_hamilton} to the function $g_{1,\eps}$ defined in~\eqref{eq:hamiltonian_function_g} yields
\begin{equation*}
\label{eq:hamiltonion_g_generator}
\cA g_{1,\eps}(z,y) = \sigma^2/2 -\eps c y^2 - (1-\eps)cU'(z)z.
\end{equation*}
Since  $\cA g_{k,\eps} = kg_{k-1,\eps}\cA g_{1,\eps}+(\sigma^2/2)k(k-1) g_{k-2,\eps}(\partial_y g_{1,\eps})^2$
and $g_{k-1,\eps}=g_{1,\eps}g_{k-2,\eps}$
 we obtain
\begin{align}
\nonumber
\cA g_{k,\eps}(z,y)  &= kg_{k-1,\eps}(z,y)\left(\cA g_{1,\eps}(z,y)+\sigma^2(k-1)(\partial_y g_{1,\eps}(x,y))^2/(2g_{1,\eps}(z,y))\right) \\
\nonumber &\leq kg_{k-1,\eps}(z,y)(\sigma^2/2-\eps cy^2-c(1-\eps)U'(z)z + \sigma^2(k-1))\\
\nonumber &= kg_{k-1,\eps}(z,y)((k-1/2)\sigma^2-\eps cy^2-c(1-\eps)U'(z)z) \\
\label{eq:hamilton_drift} &\leq -k\eps a cg_{k-1,\eps}(z,y)/2,\qquad\text{as $z^2+y^2\to\infty$,}
\end{align}
where the first inequality follows since
$$\frac{(\partial_y g_{1,\eps}(z,y))^2}{2g_{1,\eps}(z,y)} \leq (y+c(1-\eps )z)^2/((y+c(1-\eps)z)^2 + c^2(1-\eps)\eps z^2+1) \leq 1 \text{ for all $(z,y)\in \R^2$}.$$
The inequality in \eqref{eq:hamilton_drift} follows from the fact that, for $z^2+y^2$ sufficiently large, the following holds: either $|z|$ is large and hence by~\ref{example:ass_Hamilton}, $c(1-\eps)U'(z)z\geq ac(1-3/2\eps)$ or $|z|$ is small and $|y|$ is large and then $\eps c y^2\geq  (k-1/2)\sigma^2 + zU'(z) +k\eps ac$, since $zU'(z)$ is bounded for $z\in \R$ by \ref{example:ass_Hamilton}.
Moreover, since $\cA g_{k,\eps}$ is continuous, it is also globally bounded from above. 
\end{proof}

Recall that the solution $X$ of~\eqref{eq:Hamiltonian_damping_onedim} is a strong Markov process, all the skeletons are irreducible with respect to the Lebesgue measure and compact sets are petite~\cite[Lem. 1.1, Prop. 1.2]{Wu2001}.

\begin{proof}[Proof of Theorem~\ref{thm:hamiltonian}]
By \ref{example:ass_Hamilton}, we have $1<1/2+ac/\sigma^2$. Thus, we may choose $\eps\in(0,1/2)$  such that $k \coloneqq 1/2+ac(1-2\eps)/\sigma^2 >1$. Let $V\coloneqq g_{k,\eps}$ be the function given in~\eqref{eq:hamiltonian_function_g}. By Proposition~\ref{prop:deterministic_hamiltonian} there exist positive constants $C,b>0$, such  that the following inequality $\cA V\leq -\phi\circ V + b\mathbbm{1}_{D}$ holds for a sufficiently large compact (hence petite) set $D$ and an increasing function $\phi(r) \coloneqq Cr^{(k-1)/k}$. Then, by~\cite[Thm~3.4]{douc2009subgeometric}, the assumptions of~\cite[Thm~3.2]{douc2009subgeometric} (i.e. the drift condition in~\eqref{eq:douc_drif_condition}) are satisfied with $(V,\phi,D,b)$. 

Recall that for any $\eta\in[0,1)$, we have $(r_1/\eta)^{\eta}(r_2/(1-\eta))^{1-\eta}\leq r_1+r_2$ for all $r_1,r_2\in[1,\infty)$, where for $\eta=0$ we take $(1/\eta)^{\eta} = 1$. Note that $g_{\eta(k-1),\eps}= (\phi \circ V)^\eta$.
Thus, by~\cite[Thm~3.2]{douc2009subgeometric}, for every $\eta\in[0,1)$ and $x\in\R^2$, there exists $C_\eta\in(0,\infty)$ such that $\|\P_{x}(X_t\in\cdot)-\pi(\cdot)\|_{g_{\eta(k-1),\eps}}\leq C_{\eta}/t^{(k-1)(1-\eta)}$ holds for all 
$t\in[1,\infty)$. Moreover, by the definition of $g_{k,\eps}$ in~\eqref{eq:hamiltonian_function_g}, it holds that 
\begin{align}
\label{eq:hamiltonian_lower_g}
(\phi \circ V)^\eta(z,y) = g_{\eta(k-1),\eps}(z,y) & \geq (1+c^2(1-\eps)\eps z^2/2)^{\eta(k-1)} \\
& \nonumber\geq \tilde C (1+|z|^{2(k-1)\eta})= f_{2(k-1)\eta}(x)
\end{align}
for some 
$\tilde C\in(0,1)$ and all $x=(z,y)\in\R^2$, where $f_m(z,y) = 1+|z|^m$ was defined in Theorem~\ref{thm:hamiltonian} for any $m\in[0,2(k-1))$. Hence the upper bound follows. 

As in~\eqref{eq:hamiltonian_lower_g}, we have 
$H(z,y) \coloneqq C_H(1+c^2(1-\eps)\eps z^2/2)^k \leq C_Hg_{k,\eps}(z,y)$ for all $(z,y)\in\R^2$. The constant $C_H\in(1,\infty)$ is chosen so that $G_m\coloneqq H/f_m\geq 1$ on $\R^2$ for any $m\in[0,2(k-1))$.
By Proposition~\ref{prop:deterministic_hamiltonian} we have $\cA g_{k,\eps}(z,y)\leq C'$ on $\R^2$ for some $C'\in(0,\infty)$. By Lemma~\ref{lem:bounded_generator} 
there exists $C''\in(0,\infty)$, such that $\E_x[H(X_t)]\leq \E_{x}[C_Hg_{k,\eps}(X_t)]\leq C''(g_{k,\eps}(x)+t)\eqqcolon v(x,t)$ holds for all $x\in\R^2$ and $t\in[1,\infty)$.
By~\cite{Wu2001}, $X = (Z,Y)$ admits an invariant measure $\pi$ with density proportional to $(z,y)\mapsto \exp(-2c/\sigma^2(y^2/2 + U(z)))$. Moreover, by \ref{example:ass_Hamilton}, we have $U(z)\leq a(1+\eps)\log |z|$ for all $z\in\R$ outside of some compact set.
Thus, $\int_\R\pi( z,\ud y) \geq c'|z|^{-2ca(1+\eps)/\sigma^2}$ for some constant $c'\in(0,1)$ and all $|z|$ sufficiently large.
We have  
$\liminf_{|z|\to\infty}G_m(z,y)/|z|^{2k-m}>0$ since $G_m=H/f_m$. Thus, for all large $r\in\RP$, there exists $c_\pi>0$ such that (recall $2-2k=1-2ac(1-2\eps)/\sigma^2$)
$$
\int_{\{G_m\geq r\}} f_m(z,y)\pi(\ud z,\ud y) \geq  \int_{\{|z|\geq r^{1/(2k-m)}\}} c' |z|^{m-2ca(1+\eps)/\sigma^2}\ud z \geq c_\pi r^{(m+2-2k-6ca\eps/\sigma^2)/(2k-m)}.
$$ 
Applying Lemma~\ref{lem:lower_bound_f_convergence_rate} with functions $f_m$, $H$, $G_m$, $v$  and $a(r) \coloneqq  c_\pi r^{(m+2-2k-6ca\eps/\sigma^2)/(2k-m)}$, for every $x\in \R^2$ yields a positive constant $c_m$ satisfying $$\|\P_{x}(X_t\in\cdot)-\pi(\cdot)\|_{f_m}\geq c_{m}t^{(m+2-2k-6ca\eps/\sigma^2)/(2(1-3ca\eps/\sigma^2))} \quad\text{for all $t\in[1,\infty)$.}$$ Since $\eps>0$ is arbitrary, the theorem follows. 
\end{proof}

\section{Conclusion}
\label{sec:conclusion}

This papers develops a general theory for establishing lower bounds in $f$-variation for Markov processes in continuous time. The applications discussed in Section~\ref{sec:examples}
demonstrate the wide applicability of our results. Nevertheless, many questions remain open. We now discuss briefly some interesting possible further directions of research.

\paragraph{\textbf{Lower bounds for the convergence in Wasserstein metrics}}
Lyapunov drift conditions have also been developed for establishing upper bounds on the rates of convergence in Wasserstein distance. Applications include convergence of solutions of certain stochastic delay differential equations on infinite-dimensional state spaces~\cite{Mattingly11}. Subgeometric upper bounds in this context have also received attention~\cite{Butkovsky}. 
It is  feasible that our $\mathbf{L}$-drift conditions, suitably adapted to this setting, could yield  lower bounds on the return times to sets bounded in an appropriate metric.
Due to the lack of local compactness in such state spaces, the notion of petite sets from~\cite{MeynTweedie} has been replaced by a weaker notion of metric-dependent small sets, suitable for applications in infinite-dimensional settings.  A natural interesting question is whether our lower bounds on modulated moments, derived via suitably adapted $\mathbf{L}$-drift conditions, could be used to characterise the decay of the tail of the invariant measure and the rate of convergence in the Wasserstein distance. Put differently, can the lower bound Lyapunov drift conditions, developed in this paper for locally compact state spaces, be adapted to the infinite-dimensional setting?

\paragraph{\textbf{Ergodic averages.}} The quantification of the asymptotic behavior of an ergodic average $\frac{1}{t}\int_0^tX_s\ud s$, as $t\to\infty$, of a Markov process $X$ represents a fundamental problem in probability and beyond. Results on the asymptotic behaviour of  ergodic averages, such as the central limit theorem, as well as moderate and large deviations, are commonly derived using Lyapunov drift conditions and associated upper bounds on the modulated moments~\cite{douc2009subgeometric,Douc08}. For example,  the growth rate of the speed function in  moderate deviations of additive functionals of Markov processes is bounded above by the tails of the modulated moments~\cite[Thm~3.3]{douc2009subgeometric}. In Theorem~\ref{thm:modulated_moments} of the present paper we establish asymptotically matching lower bounds on the modulated moments for a wide range of models (Section~\ref{sec:examples} above). It is thus reasonable to expect that our lower  bounds have a role to play in establishing the optimal growth rate of the speed function in moderate deviations.

\paragraph{\textbf{Discrete-time Markov chains}}
To the best of our knowledge, there are no general results for establishing lower bounds on the rates of convergence to the invariant measure for discrete-time Markov chains in uncountable state spaces. It is natural to expect that the results presented in this paper can be readily adapted to establish convergence in the discrete-time setting. The assumption \nameref{sub_drift_conditions} includes a 
 supermartingale condition, exit probability, ergodicity, positive Harris recurrence and non-confinement, all of which have natural counterparts in discrete time. Moreover, since all the tools used in our proofs in Section~\ref{sec:proofs} are also available in discrete time (note that the results of Section~\ref{sec:return_times} hold for discrete-time chains already, if viewed as piecewise constant continuous-time Markov chains), the conclusions of our main results (Theorems~\ref{thm:invariant},~\ref{thm:f_rate} and~\ref{thm:modulated_moments}) are expected to hold in this setting.

\paragraph{\textbf{Application to Markov chain Monte Carlo}}
Markov chain Monte Carlo (MCMC) algorithms represent a major area of application for ergodic Markov processes~\cite{MeynTweedie,Moulin18}. In particular, upper bounds on the convergence rates of MCMC algorithms are often obtained using Lyapunov drift conditions.  
Knowing whether a convergence rate of one algorithm is better than that of another can however not be deduced from the corresponding pair of Lyapunov functions used to obtain the two upper bounds on the rates.
Our results on lower bounds for the convergence rate open the door for such comparisons of MCMC algorithms.

\appendix

\section{Non-confinement of irreducible Feller continuous processes}
A strong Markov process $X=(X_t)_{t\in\RP}$ on a locally compact metric state space $\cX$ is \textit{Feller continuous} if the function $x\mapsto\E_x[f(X_t)]$ is continuous for all $t\in(0,\infty)$ and continuous bounded functions $f:\cX\to\RP$. 

\begin{lem}
\label{lem:non_confinement}
Let $V:\cX\to[1,\infty)$ be continuous and such that $\{V\leq r_0\}$ is compact for all $r_0\in(1,\infty)$. Assume that the process $X$ is Feller continuous and that for every $r_0\in(1,\infty)$ there exists $k_0\in(0,\infty)$ such that $\P_x(V(X_{k_0})> r_0)>0$ for all $x\in\cX$. Then, we have $\P_x(\limsup_{t\to\infty}V(X_t)=\infty) = 1$.
\end{lem}

In all our examples the processes will be Feller continuous and irreducible with respect to the $d$-dimensional Lebesgue measure, making the conditions of Lemma~\ref{lem:non_confinement} easy to check.

\begin{proof}[Proof of Lemma~\ref{lem:non_confinement}]
Pick an arbitrary $r_0\in(1,\infty)$. We will show that $\P_x(\limsup_{t\to\infty} V(X_t)\leq r_0) = 0$ holds for every $x\in\cX$. Pick $x\in\cX$ and estimate
\begin{align}
\nonumber\P_x(\limsup_{t\to\infty}V(X_t)\leq r_0) &\leq \P_x(\limsup_{n\to\infty}V(X_{nk_0})\leq r_0) =\P_x(\cup_{j=1}^{\infty} \cap_{n=j}^{\infty}\{V(X_{k_0n})\leq r_0\}) \\
\label{eq:non_confinment_estimtate}&\leq\lim_{j \to\infty}\P_x(\cap_{n=j}^{\infty}\{V(X_{k_0n})\leq r_0\})
\end{align} 
where we have used the continuity from below property of the measure $\P_x$.
Since $\P_{x'}(V(X_{k_0})>r_0)>0$ for every $x'\in\cX$, the set $\{V\leq r_0\}$ is compact, and the process $X$ is Feller continuous it follow that $\inf_{x'\in \{V\leq r_0\}}\P_{x'}(V(X_{k_0})>r_0) \eqqcolon \eps>0$. Thus, by recurrent application Markov property at times $nk_0$ for $j_1,j_2\in\N$ satisfying $j_2>j_2$, we have $\P_x(\cap_{n = j_1}^{j_2}\{V(X_{nk_0})\leq r_0\}) \leq (1-\eps)^{j_2-j_1}$. The continuity from above of the measure $\P_x$, implies that for any $j\in\N$ 
$$
\P_x(\cap_{k=j}^{\infty}\{V(X_{km})\leq n_0\}) = \lim_{j'\to\infty}\P_x(\cap_{k=j}^{j'}\{V(X_{km})\leq n_0\}) \leq\lim_{j'\to\infty} (1-\eps)^{j'-j} = 0.
$$
Combining the inequality in the above display with~\eqref{eq:non_confinment_estimtate} concludes the proof.
\end{proof}

\section*{Acknowledgements}
\addcontentsline{toc}{section}{Acknowledgements}
We thank Martin Hairer for useful comments that improved the presentation of the paper and for bringing 
references~\cite{Hairer09,Wonham66} to our attention.
MB and AM are supported by  EPSRC grant  EP/V009478/1. AM is also  supported by
 EPSRC grant EP/W006227/1.

\bibliography{Lower_convergence}
\bibliographystyle{amsplain}

\end{document}